\documentclass[10pt]{amsart}
\usepackage{geometry}                
\usepackage{graphicx}
\usepackage{dsfont}
\usepackage{amssymb}
\usepackage{amsmath}
\usepackage{epstopdf}
\usepackage{fullpage}
\usepackage{enumerate}
\usepackage{color,subcaption}
\usepackage{amsthm}
\usepackage{placeins}
\usepackage{array}
\usepackage{blkarray}
\usepackage{multirow}
\usepackage{float}
\usepackage{morefloats}
\usepackage{pgfplots}
\usepackage[abs]{overpic}
\usepackage{gensymb} 

\usepackage{tikz}
\usetikzlibrary{calc,positioning}
\usepackage{algorithm,algorithmic}
\usepackage{listings}

\usepackage[colorlinks=true]{hyperref}
\usepackage[nameinlink,capitalise]{cleveref}
\crefname{equation}{}{}

\makeatletter
\def\subsection{\@startsection{subsection}{3}%
  \z@{.5\linespacing\@plus.7\linespacing}{.5\linespacing}%
  {\bf}}
\makeatother

\makeatletter
\def\subsubsection{\@startsection{subsubsection}{3}%
  \z@{.5\linespacing\@plus.7\linespacing}{.5\linespacing}%
  {\it}}
\makeatother

\usepackage{amsopn}
\definecolor{mygray}{rgb}{1.0,0.95,0.90}
\definecolor{truegray}{rgb}{0.9,0.9,0.9}
\definecolor{mygreen}{rgb}{0,0.6,0}
\definecolor{mygreen2}{rgb}{0,0.8,0}
\definecolor{mypink}{rgb}{1.0,0.56,0.81}
\lstdefinestyle{ff}%
{%
    language = C++,
    basicstyle=\ttfamily,
    backgroundcolor=\color{truegray},
    frame=tlbr,framesep=4pt,framerule=1pt,
    captionpos=b,
    commentstyle=\color{mygreen},
    emph=[1]
    {%
        macro,
        border,
        fespace,
    },
    emphstyle=[1]{\bf\color{blue}},
    emph=[2]
    {%
        varf,
        matrix,
        real,
        string,
        int,
        Vh,
        Vh0,
        Vh0i,
        Vh2,
        Vhi,
        mesh,
    },
  emphstyle=[2]{\bf\color{black}},
    emph=[3]
    {%
        int2d,
        int1d,
        on,
        readmesh,
        savemesh,
         load,
        include,
        trunc,
    },
  emphstyle=[3]{\bf\color{violet}},
      emph=[4]
    {%
        EigenValue,
        buildmesh,
        adaptmesh,
        movemesh,
        checkmovemesh,
         solve,
        problem,
        laplace,
        elas,
    },
  emphstyle=[4]{\color{red} \bfseries},  
        emph=[5]
    {%
        volfrac,
        readsol,
        advectRedist,
        normalvec,
        curvature
    },
  emphstyle=[5]{\color{mygreen2} \bfseries}, 
}

\lstdefinestyle{shell}%
{%
    basicstyle=\ttfamily\color{white},
         frame=tlbr,framesep=4pt,framerule=1pt,
    backgroundcolor=\color{black},
    captionpos=b,
     emph=[1]
    {
        python,
    },
    emphstyle=[1]{\color{red}\bfseries},
}

\lstdefinelanguage{mypython}{
   basicstyle=\ttfamily\color{black},
  morekeywords = [1]{print,format},
  morekeywords = [2]{if,else,and,or,break,continue,def},
  keywordstyle = [1]\bf\color{lightgray},
  keywordstyle = [2]{\bf\color{mypink}},
  morestring=[b]",
  stringstyle = {\color{lightgray}},
  morecomment=[l]{\#},
}

\lstdefinestyle{python}%
{%
    language = mypython,
     commentstyle=\color{mygreen},   
     frame=tlbr,framesep=4pt,framerule=1pt,
     showstringspaces=false,
    backgroundcolor=\color{mygray},
    captionpos=b,
     emph=[1]
    {%
        mechtools,
        mshtools,
        inigeom,
        inout, 
        lstools,
        path,
        subprocess,
    },
    emphstyle=[1]{\bf\color{blue}},
  }

\let\oldtocsection=\tocsection

\let\oldtocsubsection=\tocsubsection

\let\oldtocsubsubsection=\tocsubsubsection

\renewcommand{\tocsection}[2]{\hspace{0em}\oldtocsection{#1}{#2}\textbf}
\renewcommand{\tocsubsection}[2]{\hspace{1em}\oldtocsubsection{#1}{#2}}
\renewcommand{\tocsubsubsection}[2]{\hspace{2em}\oldtocsubsubsection{#1}{#2}}

\newcommand{\dv}{\text{\rm div}}
\newcommand{\curl}{\text{\rm curl}}
\newcommand{\pv}{\text{\rm p.v.}}

\renewcommand{\o}{\text{\rm o}}

\renewcommand{\d}{\text{\rm d}}

\newcommand{\com}{\text{\rm com}}

\newcommand{\tr}{\text{\rm tr}}

\newcommand{\e}{\varepsilon}
\newcommand{\I}{\text{\rm I}}
\newcommand{\Id}{\text{\rm Id}}

\newcommand{\calK}{{\mathcal K}}
\newcommand{\calC}{{\mathcal C}}
\newcommand{\calB}{{\mathcal B}}

\newcommand{\calM}{{\mathcal M}}
\newcommand{\calO}{{\mathcal O}}
\newcommand{\calT}{{\mathcal T}}
\newcommand{\calTt}{{\mathcal T}^{\text{\rm temp}}}

\newcommand{\calL}{{\mathcal L}}
\newcommand{\calLt}{{\mathcal L}^{\text{\rm temp}}}

\newcommand{\Cinfty}{{\mathcal C}^{1,\infty}(\mathbb{R}^2;\mathbb{R}^2)}
\newcommand{\R}{{\mathbb R}}

\newcommand{\upc}{u_{\text{\rm pc}}}

\newcommand{\Gobs}{{\Gamma_{\text{\rm obs}}}}
\newcommand{\uobs}{u_{\text{\rm obs}}}
\newcommand{\dmin}{d_{\text{\rm min}}}

\usepackage{soul,xcolor}
\setstcolor{violet}

\begin{document}
\newtheorem{theorem}{Theorem}[section]
\newtheorem{problem}{Problem}[section]
\newtheorem{remark}{Remark}[section]
\newtheorem{example}{Example}[section]
\newtheorem{definition}{Definition}[section]
\newtheorem{lemma}{Lemma}[section]
\newtheorem{corollary}{Corollary}[section]
\newtheorem{proposition}{Proposition}[section]
\newtheorem{propdef}{Definition-Proposition}[section]
\numberwithin{equation}{section}

\title{Body-fitted tracking of 2d open curves with a level set based mesh evolution method}
\author{
C. Dapogny$^{1}$
}

\maketitle
\begin{center}
\emph{\textsuperscript{1} Sorbonne Universit\'e, Universit\'e Paris Cit\'e, CNRS, Inria, Laboratoire Jacques-Louis
Lions, LJLL, F-75005 Paris, France}.
\end{center}

\begin{abstract}
This article describes a novel numerical algorithm for tracking the motion of a collection of open curves in two space dimensions. 
The proposed strategy combines two complementary representations of these curves at each iteration of the evolution process: 
on the one hand, they are meshed explicitly, as a sub-collection of the entities of a mesh of the total computational domain. 
Concurrently, using a variant of the Level Set Method, they are captured implicitly as algebraic combinations of the negative, zero and positive subsets of two auxiliary scalar functions, defined on the whole ambient space. 
This coupling of representations allows to perform accurate geometric or mechanical computations on these curves, while leaving room for large evolution of their shape.
After the description of its main numerical ingredients, several applications examples of this methodology are proposed, where it is used to simulate the motion of open physical discontinuities, such as a vortex sheet roll-up, and to optimize the shape of open-ended curves, e.g. with respect to their anisotropic length, with the aim to improve the trajectory of a laser acting on a powder bed in the context of additive manufacturing, or to reconstruct fracture sets.
\end{abstract} 

\bigskip
\bigskip
\hrule
\tableofcontents
\vspace{-0.5cm}
\hrule
\bigskip
\bigskip


\section{Introduction}


\noindent A wide variety of physical phenomena bring into play evolving interfaces or discontinuities, that take the form of open curves in 2d, or open surfaces in 3d.
For instance, in aerodynamics, the wake generated by an aircraft operating at high Reynolds number features vortex sheets, that are infinitely thin open surfaces across which the tangential fluid velocity is discontinuous and whose complex roll-up is driven by self-induced advection \cite{kundu2024fluid}. 
In solid mechanics, fractures are open surfaces across which the displacement of the medium is discontinuous, that evolve under the effect of the stress concentration at their boundaries \cite{gdoutos2005fracture}. 

A great deal of attention has been paid in the literature to the development of methods for tracking the motion of domains, 
or equivalently, the closed curves (in 2d) or surfaces (in 3d) defining their boundaries. 
Broadly speaking, these can be classified into two categories. On the one hand, Lagrangian methods represent the evolving object by a collection of particles and all the attached physical quantities of interest (e.g. the density, velocity and pressure in the case of a fluid, the mass and displacement of a structure) are expressed directly in terms of these, for instance with the help of smoothing kernels. 
The Smoothed Particle Hydrodynamics (SPH) method, initially proposed in \cite{gingold1977smoothed} to track bulk domains in the field of astrophysics, then extended to fluid dynamics \cite{violeau2012fluid} and solid mechanics \cite{libersky1993high}, is possibly the most famous strategy in this spirit, see \cite{monaghan2005smoothed} for a review. Other Lagrangian methods, such as the Material Point Method \cite{sulsky1994particle} or the Immersed Boundary Method \cite{peskin2002immersed}, complement this representation with a fixed mesh of a larger ``hold-all'' domain; back-and-forth transfers between the particles and the nodes of the mesh allow to perform mechanical computations on the latter. 
Lagrangian methods generally suffer from the large number of particles needed to achieve a reasonable accuracy. Furthermore, repeated resampling of the particles is needed to maintain a suitable description of the curve, especially in the regions where the motion induces a large stretching.
On the other hand, Eulerian methods capture an evolving object via certain auxiliary fields or markers, that are discretized on a fixed mesh of a computational domain. Among these, the Volume Of Fluid method \cite{hirt1981volume} represents a fluid mixture by the fractions of each fluid inside the grid voxels \cite{mohan2024volume}. Alternatively, the Level Set Method accounts for a moving domain as the negative subdomain of a scalar ``level set'' function, defined on the whole computational domain, see \cite{osher2006level,osher1988fronts,sethian1999level} and \cref{sec.LSM} below for a brief outline. The Phase-Field Method features a diffuse interface, also captured by a scalar function, whose thickness is penalized by the addition of a perimeter term to the energetic formulation of the evolution \cite{qin2010phase,steinbach2009phase}. Despite their robust description of large motions, the accuracy of Eulerian representations is limited by the size of the fixed mesh. Furthermore, no explicit support of the interface is available for geometrical or mechanical computations, that have to be expressed in terms of its Eulerian descriptors, often at the expense of approximations.
Of course, a whole range of hybrid methods is available, sharing features from the Eulerian and Lagrangian viewpoints. Among these, let us mention the so-called Arbitrary Lagrangian-Eulerian (ALE) methods, where the computational mesh is allowed to move to better track the evolving interface, in a way which may nevertheless differ from the interface motion; see \cite{donea2004arbitrary} for an overview of this paradigm. Let us also mention the Particle Level Set Method from \cite{enright2002hybrid}, where a ``classical'' level set update procedure is helped with a collection of particles evolving in a Lagrangian fashion.

The above investigations primarily address the evolution of closed contours; comparatively little attention has been devoted to the treatment of open objects -- that is, curves with endpoints in 2d, as exemplified in \cref{fig.2ls}, or pieces of surfaces having contours in 3d. This task is usually more involved as it demands a careful representation of the motion of the boundary of the object, in addition to that of its bulk structure.
Among the Lagrangian strategies implemented to achieve this goal, let us mention the work \cite{leung2009grid}, devoted to tracking the motion of an open curve in two or three space dimensions. The evolving object is discretized with a collection of particles, while a fixed background mesh stores closest-point information at its nodes. The motion of the object is realized by moving the particles, and then updating the closest point information stored at the mesh nodes; a new collection of particles is reconstructed by polynomial fitting, which consistently ensures a uniform sampling of the curve, of the order of the mesh resolution. The information stored at grid points also serves to render topological changes. This framework allows to deal with multiple junctions; however, the treatment of topological changes is heuristic, and no meshed representation of the curve is available. 
A Eulerian method to track the motion of a 2d open curve is introduced in the work \cite{smereka2000spiral}, dealing with the numerical simulation of crystal growth: the evolving curve under scrutiny represents the step line of the crystal, and it is driven by a curvature-based velocity field. Elaborating on a suggestion from \cite{ambrosio1994level}, two level set functions are used to account for the evolving curve, as presented in \cref{sec.2LSM}. Such a strategy was used later in \cite{bar2014mumford,mohieddine2011open} for the purpose of image segmentation, and in \cite{solem2006reconstructing} to fit an open curve (in 2d) or surface (in 3d) to an unorganized point cloud. This idea of using two level set functions for describing an open object has since then been quite popular in the literature, see for instance \cite{moes2002nonb} about its introduction in fracture mechanics, and  \cite{burchard2001motion,cheng2002motion} where it is used to account for the motion of a curve within a surface in 3d. 
Another Eulerian viewpoint on the evolution of possibly open curves or surfaces is the so-called vector distance function method, introduced in \cite{gomes2003vector}: the considered object is represented by the datum of its unsigned distance function at the vertices of a computational mesh and that of its gradient -- which is an extension of the normal vector to its boundary. However elegant, this strategy is a little intricate to implement, as it is difficult to accurately locate the $0$ level set of the representation. We refer to \cite{niethammer2005evolution,salzman2016use} for implementation details about this elegant viewpoint and \cite{ventura2003vector} for an application to the simulation of crack propagation.

In this open setting also, Eulerian approaches generally do not provide an explicit body-fitted discretization amenable to accurate geometric or mechanical computations, whereas purely Lagrangian approaches struggle with robustness under large deformations and topological changes. The present article aims to reconcile these two requirements by proposing a robust, body-fitted methodology for tracking the motion of open curves in 2d, building on a combination of the Level Set Method and remeshing algorithms. 
This strategy draws inspiration from our previous work \cite{allaire2011topology,allaire2013mesh,allaire2014shape}, devoted to the evolution of ``bulk'' domains, and its extension \cite{bonnetier2025numerical,brito2025body} to the case of regions on surfaces.
Our method combines two representations of a two-dimensional open curve $\Gamma$, or more generally, of a collection of such curves. On the one hand, $\Gamma$ is meshed explicitly; more precisely, the computational domain $D$ is equipped with a triangular mesh $\calT$, and a collection $\calL$ of edges of $\calT$ accounts for a discretization of $\Gamma$. On the other hand, $\Gamma$ is described implicitly, with the help of two scalar ``level set'' functions $\phi, \psi: D \to \R$: the $0$ level set of $\phi$ is an extension of $\Gamma$ into a closed curve $\widetilde\Gamma$, and the negative subdomain of $\psi$ delimits $\Gamma$ within $\widetilde\Gamma$. Each operation of the evolution workflow (geometric and mechanical computations related to the evaluation of the velocity field, update of the curve...) is applied on the most suitable representation; efficient numerical algorithms allow to pass from one representation to the other whenever needed. This strategy allows to conduct precise mechanical and geometrical calculations related to $\Gamma$, while leaving room for an arbitrary evolution of the latter.

Most of the conceptual and algorithmic contents of this article hold true in two and three space dimensions, although the implementation is significantly more tedious in the latter case. This article focuses on the 2d setting, where the objects at stake are open curves, or more generally, collections of multiple, disjoint open curves. Although most of the relevant physical applications would take place in 3d, it is already possible to address a few non trivial, interesting applications with this technology. Since the methodology was designed with the three-dimensional setting in mind, the present two-dimensional implementation may be less efficient than more specialized approaches dedicated solely to planar open curves. In a similar spirit, the management of topological changes, however natural with the use of the Level Set Method, is not detailed, as it does not find relevant applications in the presented examples.
The extension of the proposed framework to three space dimensions is an ongoing work, that will be presented in the forthcoming article \cite{dapogny2025levelset}; note however that some of its ingredients have already been used in \cite{belhachmi2025tetrahedral,feppon2026fractured}.

The sources code of our numerical implementation are freely available at the following address: 
\begin{center}
\texttt{https://github.com/dapogny/openls}
\end{center}
It elaborates on the open-source library developed in \cite{dapogny2023shape}, dealing with the optimal design of ``bulk'' shapes. 

The remainder of this article is organized as follows. \cref{sec.2ls} describes the representation of an evolving open curve in $\R^2$ by the datum of two level set functions. Then, \cref{sec.algo} is devoted to our numerical method for tracking this motion: we outline the key ingredients of our strategy, namely the numerical method for generating two level set functions from a line mesh of an open curve, and, conversely, the remeshing algorithm used to discretize an open curve described by two level set functions. Since these algorithms will be the subject of a dedicated article, in the much more technical three-dimensional setting, we deliberately keep the discussion at a high level, omitting technical details. 
\cref{sec.appphys} numerically assesses our framework with two applications where the velocity field at play is ``simple'', in the sense that it is explicitly calculated from a discretization of the curve. The next \cref{sec.appso} investigate three applications of our methodology in the more challenging perspective of shape optimization, where the velocity field driving the motion typically depends on the calculation of geometrical quantities or on the solution of one or several boundary value problems -- thus demanding an accurate, meshed description of the curve. We notably consider the minimization of anisotropic perimeter functionals, the optimization of the laser path of an additive manufacturing process, and the reconstruction of the shape of a fracture within the underground from measurements of the boundary of the upper surface. A conclusion and a few leads for future work are given in \cref{sec.concl}.
This article ends with a series of appendices, containing modeling or technical details related to the different applications of the article; when relevant and not detrimental to clarity, these are provided in a context which is slightly more general than that of the main text.


\section{A two-level-set method for the representation of a moving open curve}\label{sec.2ls}


\noindent This section describes the representation of an evolving open curve by means of two ``level set'' functions
and sets some notations that are used throughout the article.
For completeness, we briefly recall in \cref{sec.LSM} the basic features of the ``classical'' Level Set Method devoted to evolving closed curves, before turning to the two-function variant used in this article to deal with open curves in \cref{sec.2LSM}.

\subsection{A brief reminder of the ``classical'' Level Set Method}\label{sec.LSM}

\noindent The Level Set Method was pioneered in \cite{osher1988fronts} as an efficient means to capture the motion of an evolving domain, and we presently recall its salient features, see for instance the reference books \cite{osher2006level,sethian1999level} for more exhaustive presentations. 

Let $D \subset \R^2$ be a large ``hold-all'' domain. The Level Set Method consists in capturing a smooth subdomain $\Omega \subset D$ as the negative region of a scalar ``level set'' function $\phi : D \to \R$, i.e. 
\begin{equation}\label{eq.1ls}
\forall x \in D, \quad \left\{
\begin{array}{cl}
\phi(x) < 0 & \text{if } x \in \Omega,\\
\phi(x) = 0 &\text{if } x \in \partial \Omega, \\
\phi(x) >0 & \text{otherwise.}
\end{array}
\right.
\end{equation}
This change in perspective does not incur loss of information about $\Omega$; in particular, all the geometric quantities attached to $\Omega$ can be expressed in terms of a sufficiently smooth level set function $\phi$; for instance, the unit normal vector $n$ to $\partial \Omega$, pointing outward $\Omega$, and its mean curvature $\kappa$ equal: 
$$n(x) = \frac{\nabla\phi(x)}{\lvert \nabla \phi(x)\lvert} \text{ and } \kappa(x) = \dv\left( \frac{\nabla\phi(x)}{\lvert \nabla \phi(x)\lvert}\right), \quad x \in \partial\Omega. $$

Let us now consider a time-dependent domain $\Omega(t)$, evolving over a time period $(0,T)$ under the effect of a velocity field $V: \R_t \times \R^2_x \to \R^2$. 
Introducing a level set function $\phi(t,\cdot)$ for $\Omega(t)$, such that \cref{eq.1ls} holds at each time $t \in (0,T)$, and a level set function $\phi_0$ for the initial domain $\Omega(0)$, an elementary application of the chain rule at the formal level shows that the motion of $\Omega(t)$ translates as the following advection-like equation for $\phi$:
\begin{equation}\label{eq.1lsadvect}
\left\{
\begin{array}{cl} 
\frac{\partial \phi}{\partial t}(t,x) + V(t,x) \cdot \nabla \phi(t,x) = 0 & \text{on } (0,T) \times D,\\
\phi(0,x) = \phi_0(x) & \text{on } D.
\end{array}
\right.
\end{equation} 
Alternatively, introducing the normal component $v(t,x) := V(t,x) \cdot \frac{\nabla\phi(t,x)}{\lvert\nabla \phi(t,x)\lvert}$ of the velocity field, \cref{eq.1lsadvect} takes the form a Hamilton-Jacobi equation: 
\begin{equation}\label{eq.1lshj}
\left\{
\begin{array}{cl} 
\frac{\partial \phi}{\partial t}(t,x) + v(t,x) \lvert \nabla \phi(t,x) \lvert= 0 & \text{on } (0,T) \times D,\\
\phi(0,x) = \phi_0(x) & \text{on } D.
\end{array}
\right.
\end{equation}

Note that \cref{eq.1lsadvect,eq.1lshj} are not ``true'' advection or Hamilton-Jacobi equations, insofar as the velocity field $V(t,x)$ may depend on $\Omega(t)$, thus on $\phi(t,x)$ itself.
The equations \cref{eq.1lsadvect,eq.1lshj} hold true in the classical sense as long as $\Omega(t)$, $\phi(t,x)$ and $V(t,x)$ are ``smooth'', but they have to be interpreted in the sense of viscosity as soon as singularities appear \cite{crandall1992user}. We refer to \cite{ambrosio1994level,evans1991motion} about the mathematical foundations of the level set approach for the mean curvature flow, and to \cite{giga2006surface} for a larger perspective.

\begin{remark}\label{rem.meetbdy}
The framework of this section also allows to represent curves $\Gamma$ that are closed in the sense that they meet the boundary $\partial D$ of the computational domain $D$. In such situation, $\Gamma$ delimits two complementary subdomains of $D$, that are characterized by different signs of the level set function $\phi$: intuitively, one may imagine that $\Gamma$ may be extended outside $D$ by a closed curve.
\end{remark} 

\begin{remark}
As we have mentioned, in applications, the velocity field $V(t,x)$ and its normal component $v(t,x)$ depend on $\Omega(t)$ in a complex way, e.g. through the solution of a boundary-value problem posed on $\Omega(t)$. 
One pragmatic numerical strategy to treat such situations consists in decomposing the time interval $(0,T)$ into a series $t^n= n \Delta t$ of sub-intervals, $n=0,\ldots$, where $\Delta t >0$, and freezing $V(t,x)$ on each such interval, setting 
$$ V(t,x) \approx V^n(x) := V(t^n,x) \:\text{ for } t \in (t^n,t^{n+1}). $$
The evolution equation \cref{eq.1lsadvect} thus boils down to a series of ``true'' advection equations. Likewise, if only the normal component $v(t,x)$ is frozen, one obtains a series of ``true'' Hamilton-Jacobi equations with time-independent normal velocity $v^n(x)$.
\end{remark}

\begin{remark}\label{rem.redist}
The capture of the motion of $\Omega(t)$ by the evolution equation \cref{eq.1lsadvect} or \cref{eq.1lshj} holds true for any choice of the level set function $\phi(t,\cdot)$ associated to $\Omega(t)$. 
In practice, however, it is well-known that very ``steep'' or ``flat'' gradients of $\phi(t,\cdot)$ cause major numerical artifacts, see \cite{chopp1993computing}. For this reason, it is often chosen to handle the signed distance function $d_{\Omega(t)}$ of $\Omega(t)$:
\begin{equation}\label{eq.sdf}
d_{\Omega(t)}(x) = \left\{
\begin{array}{cl}
-d(x,\partial\Omega(t)) & \text{for } x \in \Omega(t), \\ 
0& \text{for } x \in \partial \Omega(t), \\ 
d(x,\partial\Omega(t)) & \text{for } x \in D \setminus \overline{\Omega(t)},
\end{array}
\right.  
\text{ where } d(x,\partial \Omega(t)) = \inf\limits_{y \in \partial\Omega(t)} \lvert x- y \lvert.
\end{equation}
 Since this signed distance property is not preserved through the resolution of \cref{eq.1lsadvect} or \cref{eq.1lshj}, most numerical algorithms for the resolution of these evolution equations periodically feature a stage where $\phi(t,\cdot)$ is restored as a signed distance function -- an operation referred to as redistancing in the literature.
\end{remark}

\subsection{Representation of open curves with two level set functions}\label{sec.2LSM}

\noindent This section describes a variant of the Level Set Method adapted to the representation of open curves, whose endpoints are strictly contained in the computational domain $D$; this idea was proposed in \cite{cheng2002motion,smereka2000spiral}, building on a theoretical suggestion of \cite{ambrosio1994level}, see also \cite{moes2002nonb} about its introduction in the context of fracture mechanics. 

Still denoting by $D$ the fixed computational domain, let $\Gamma \subset D$ be a smooth, simple open 2d curve. 
Note that $\Gamma$ may actually be a collection of several disjoint such curves, but for notational brevity, throughout this article, we refer to it as an ``open curve''. 
We also denote by $\Sigma$ the boundary of $\Gamma$, i.e. $\Sigma$ is a set of isolated endpoints. We characterize $\Gamma$ via two smooth domains $\Omega, \calO \subset \R^2$, such that: 
\begin{equation}\label{eq.2LS}
\Gamma = \partial \Omega \cap \calO, \text{ and } \Sigma = \partial \Omega \cap \partial \calO,
\end{equation}
see \cref{fig.2ls} for an illustration.
We then describe $\Omega$ and $\calO$ as the negative subdomains of two respective level set functions $\phi, \psi: D \to \R$, i.e. 
\begin{equation}\label{eq.2ls}
 \forall x \in D, \quad \left\{
\begin{array}{cl}
\phi(x) < 0 & \text{if } x \in \Omega, \\
\phi(x) = 0 & \text{if } x \in \partial\Omega, \\
\phi(x) > 0 & \text{if } x \in D\setminus \overline{\Omega}, \\
\end{array}
\right. 
\text{ and } 
 \left\{
\begin{array}{cl}
\psi(x) < 0 & \text{if } x \in \calO, \\
\psi(x) = 0 & \text{if } x \in \partial\calO, \\
\psi(x) > 0 & \text{if } x \in D\setminus \overline{\calO}. \\
\end{array}
\right. 
\end{equation}
Intuitively, the $0$ isoline of $\phi$ is a closed curve $\widetilde \Gamma:= \partial \Omega$ (or a collection of such), being understood that it may reach the boundary of $D$ to realize this feature, see \cref{fig.2ls} (b). The role of the secondary level set function $\psi$ is to delimit $\Gamma$ within $\widetilde \Gamma$. 
In this setting, $\Gamma$ and $\Sigma$ are thus characterized by:
$$\Gamma = \Big\{ x \in D \text{ s.t. } \phi(x) = 0 \text{ and } \psi(x) <0 \Big\}, \text{ and } \Sigma =  \Big\{ x \in D \text{ s.t. } \phi(x) = 0 \text{ and } \psi(x) =0 \Big\}.$$

\begin{figure}[!ht]
    \centering
    \begin{tabular}{cc}
\begin{minipage}{0.47\textwidth}
\centering
\begin{overpic}[width=1.0\textwidth]{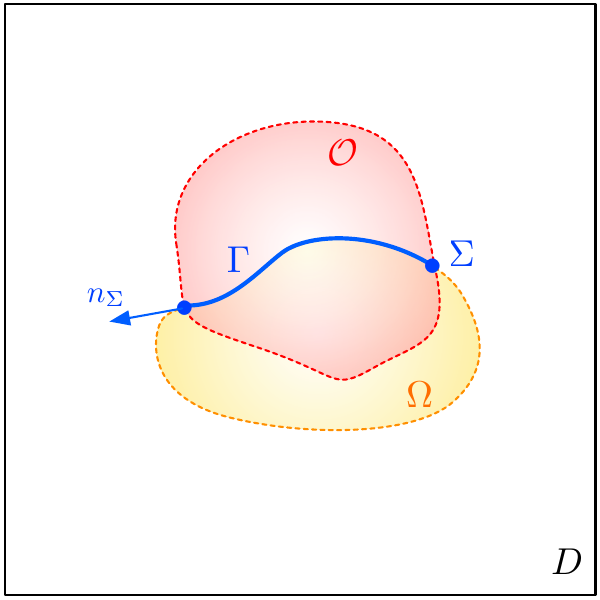}
\put(0,-3){\fcolorbox{black}{white}{a}}
\end{overpic}
\end{minipage} & 
\begin{minipage}{0.47\textwidth}
\centering
\begin{overpic}[width=1.0\textwidth]{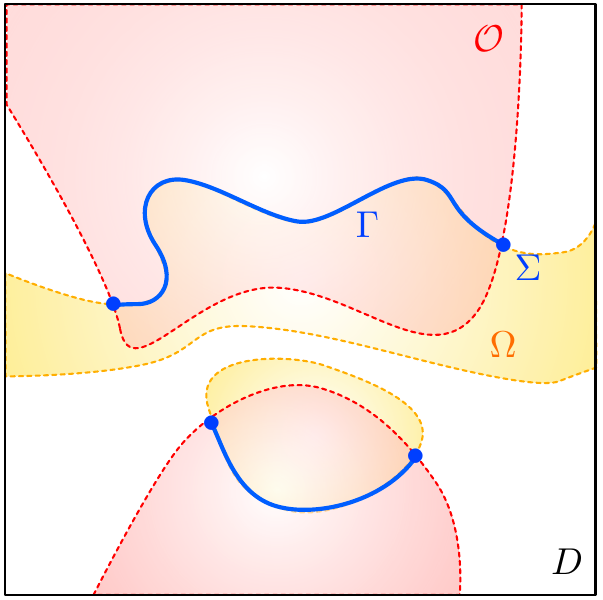}
\put(0,-3){\fcolorbox{black}{white}{b}}
\end{overpic}
\end{minipage}  
\end{tabular}  
 \caption{\it Examples of open curves described by two domains $\Omega$ and $\calO$ according to \cref{eq.2LS}; (a) Case where $\Omega$ and $\calO$ are strictly contained in $D$; (b) Case where  $\Omega$ and $\calO$ have multiple connected components and intersect the boundary of the computational domain $D$.}
  \label{fig.2ls}
\end{figure}

Analogously to the case of closed curves considered in \cref{sec.LSM}, this representation allows to capture 
the evolution of an open curve $\Gamma(t)$ over a time period $(0,T)$, according to a velocity $V: \R_t \times \R^2_x \to \R^2$.
Indeed, let $\phi(t,\cdot)$ and $\psi(t,\cdot)$ be two level set functions for $\Gamma(t)$ -- i.e. \cref{eq.2ls} holds at each time $t\in (0,T)$ -- 
and let $\phi_0, \psi_0: D \to \R$ be two level set functions for the initial set $\Gamma(0)$. The motion of $\Gamma(t)$ translates in terms of the following advection-like equations about $\phi$ and $\psi$: 
\begin{multline}\label{eq.2lsadvect}
\left\{
\begin{array}{cl} 
\frac{\partial \phi}{\partial t}(t,x) + V(t,x) \cdot \nabla \phi(t,x) = 0 & \text{on } (0,T) \times D,\\
\phi(0,x) = \phi_0(x) & \text{on } D, 
\end{array}
\right.
\text{ and }  \\
\left\{
\begin{array}{cl} 
\frac{\partial \psi}{\partial t}(t,x) + V(t,x) \cdot \nabla \psi(t,x) = 0 & \text{on } (0,T) \times D,\\
\psi(0,x) = \psi_0(x) & \text{on } D.
\end{array}
\right.
\end{multline} 

\begin{remark}
In addition to the requirement discussed in \cref{rem.redist}, whereby $\phi$ and $\psi$ should be (close to) signed distance functions,
it is crucial that these functions satisfy the following condition near $\Sigma$:
\begin{equation}\label{eq.orthoLS}
\forall x \in \Sigma, \quad \nabla \phi(x) \cdot \nabla \psi(x) = 0.
\end{equation}
Intuitively, this relation ensures that the $0$ level sets $\widetilde \Gamma = \partial \Omega$ and $\partial \calO$ of $\phi$ and $\psi$ intersect in an orthogonal fashion, 
which helps to avoid undesirable collisions of these sets during the solution of the evolution equations \cref{eq.2lsadvect}.
In practice, the condition \cref{eq.orthoLS} is not preserved through the resolution of \cref{eq.2lsadvect}, raising the need for a periodic re-orthogonalization of $\phi$ and $\psi$, see e.g. \cite{burchard2001motion,moes2002nonb}. 
\end{remark}


\section{Body-fitted evolution of open curves in a two-level set framework}\label{sec.algo}


\noindent Let $D \subset \R^2$ be a fixed computational domain, and let $\Gamma(t) \subset D$ be an open curve, evolving through a time period $(0,T)$ under a velocity field $V: \R_t \times \R^2_x \to \R^2$. 
Let $t^n = n \Delta t$, $n=0,\ldots, N := T/\Delta t$ be a discretization of $(0,T)$ based on a time step $\Delta t>0$. For each $n=0,\ldots,N$, we label with an $^n$ superscript all the instances of the time-dependent objects under scrutiny at the $n^{\text{th}}$ iteration. 

The proposed strategy for tracking the motion of $\Gamma(t)$ combines two complementary representations of its configuration $\Gamma^n$ at each iteration of the evolution process:
\begin{itemize}
\item (Meshed representation) The total domain $D$ is equipped with a valid, conforming and high-quality triangular mesh $\calT^n$; a discretization of $\Gamma^n$ is available as a collection $\calL^n$ of edges pertaining to $\calT^n$, see \cref{fig.2lsrep} (a); 
\item (Level set representation) $\Gamma^n$ is accounted for by the datum of two level set functions $\phi^n, \psi^n :D\to \R$, along the lines of \cref{sec.2LSM}, see \cref{fig.2lsrep} (b). 
\end{itemize}
Efficient numerical methods, that are described below (see \cref{sec.gen2LS,sec.remesh}) allow to pass from one of these representations to the other whenever needed. 
Thus, each operation of the evolution process can be conducted by using the most suitable representation: the geometrical or mechanical analyses involved in the calculation of the velocity field $V^n(x)$ enjoy the accuracy of an exact mesh of $\Gamma^n$, while the evolution of the curve between two successive steps is conveniently accounted for by the Level Set Method. \par\medskip

\begin{figure}[!ht]
    \centering
    \begin{tabular}{cc}
\begin{minipage}{0.3\textwidth}
\centering
\begin{overpic}[width=1.0\textwidth]{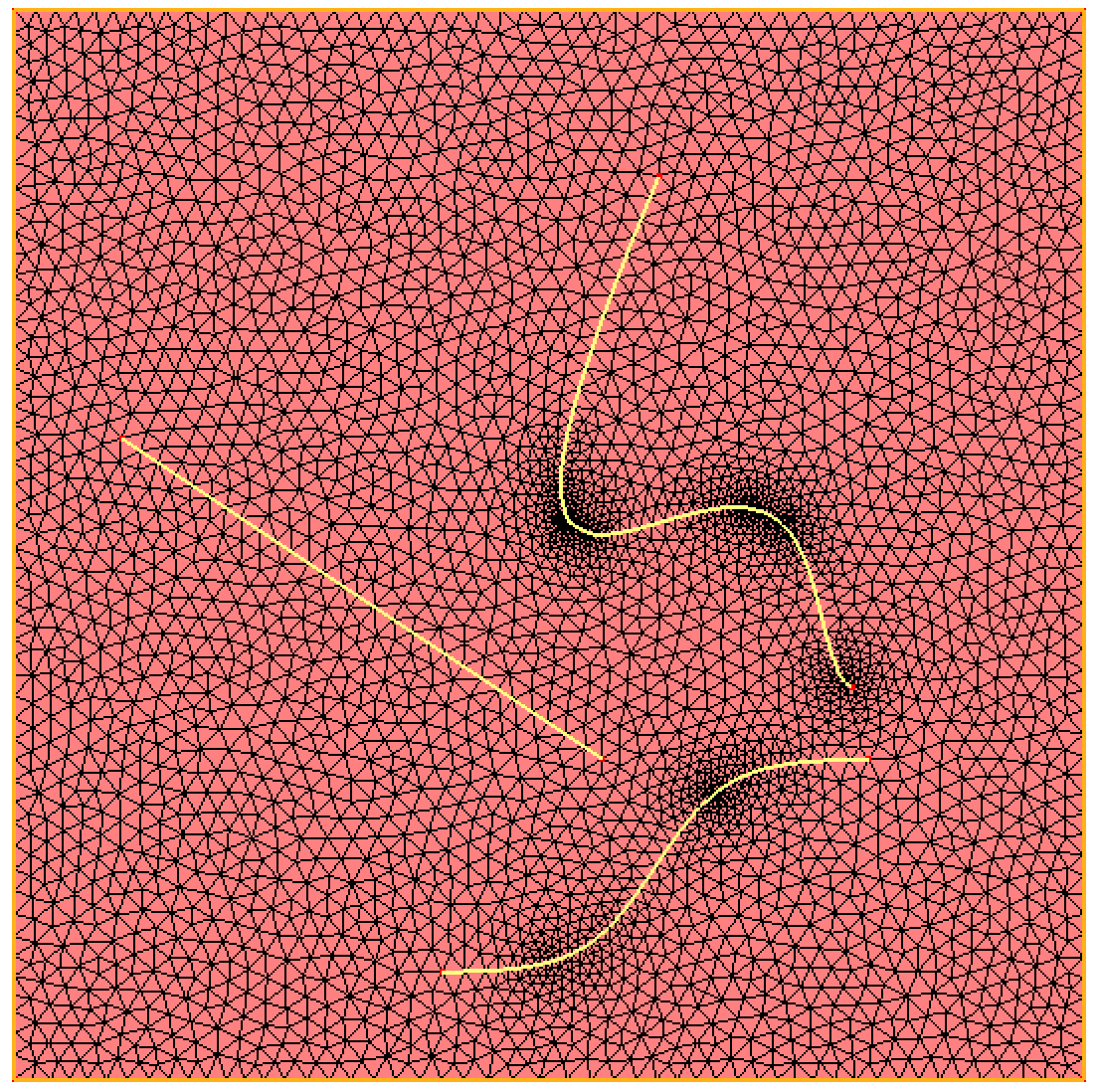}
\put(0,-3){\fcolorbox{black}{white}{a}}
\end{overpic}
\end{minipage} & 
\begin{minipage}{0.7\textwidth}
\centering
\begin{overpic}[width=1.0\textwidth]{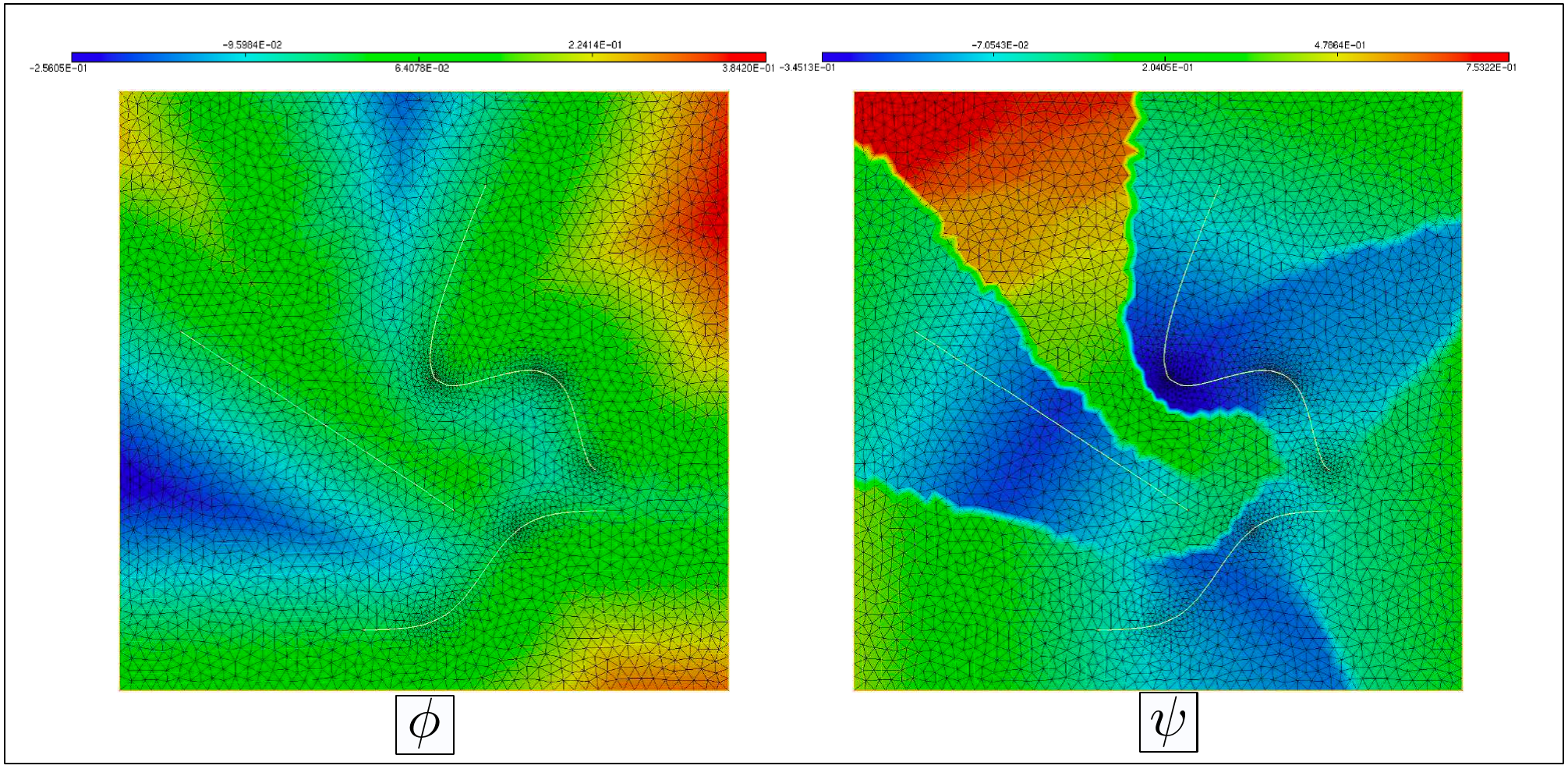}
\put(0,-3){\fcolorbox{black}{white}{b}}
\end{overpic}
\end{minipage}  
\end{tabular}  
 \caption{\it Two complementary representations of an open curve $\Gamma \subset D$ in the tracking strategy of \cref{sec.algo}; (a) Meshed representation of $\Gamma$; (b) Level set representation of $\Gamma$.}
  \label{fig.2lsrep}
\end{figure}

Our numerical strategy is summarized in \cref{algo.openls}, and its main steps are discussed in the next sections. In a nutshell, each iteration $n=0,\ldots$ starts with a meshed representation of $\Gamma^n$, i.e. $D$ is equipped with a triangular mesh $\calT^n$ and $\Gamma^n$ is discretized as a collection $\calL^n$ of edges of $\calT^n$. Two level set functions $\phi^n$, $\psi^n$ for $\Gamma^n$ are then generated at the vertices of the mesh $\calT^n$, thanks to the algorithm described in \cref{sec.gen2LS}. Then, the velocity field $V^n(x)$ is calculated; this stage depends on the nature of the motion, and it may require geometric computations (e.g. of the curvature of $\Gamma^n$), or the resolution of one or several boundary value problems, which is achieved thanks to the Finite Element Method used on the mesh $\calT^n$. The update of $\Gamma^n$ is realized in the level set framework: the equations \cref{eq.2lsadvect} are solved on $\calT^n$, with initial data $\phi^n$, $\psi^n$, over the time period $(0,\Delta t)$, thanks to the numerical method outlined in \cref{sec.lsadv}. This results in two level set functions $\phi^{n+1}$ and $\psi^{n+1}$ for the new curve $\Gamma^{n+1}$, that are known through their values at the vertices of $\calT^n$. Eventually, a new mesh $\calT^{n+1}$ of $D$ is generated from this datum, where $\Gamma^{n+1}$ appears explicitly, see \cref{sec.remesh} about this operation.

\begin{algorithm}[ht]
\caption{Body-fitted tracking of the motion of an open curve}
\label{algo.openls}
\begin{algorithmic}[0]
\STATE \textbf{Initialization:} Mesh $\calT^0$ of the computational domain $D$, in which the initial curve $\Gamma^0$ is explicitly discretized, as a collection $\calL^0$ of edges of $\calT^0$.
\FOR{$n=0,...,$ until convergence}
\STATE \begin{enumerate}
\item Calculate two level set functions $\phi^n$, $\psi^n$ for $\Gamma^n$ at the vertices of the mesh $\calT^n$, representing $\Gamma^n$ via \cref{eq.2ls}. 
\item Calculate the velocity field $V^n(x)$ on the mesh $\calT^n$. 
\item Solve the evolution equations \cref{eq.2lsadvect} on the mesh $\calT^n$ and the time interval $(0,\Delta t)$, with initial data $\phi^n$, $\psi^n$ and velocity field $V^n(x)$, 
to obtain the new level set functions $\phi^{n+1}$ and $\psi^{n+1}$.
\item Create a new mesh $\calT^{n+1}$ of $D$ in which the new surface $\Gamma^{n+1}$ is explicitly discretized. 
\end{enumerate}
\ENDFOR
\RETURN Mesh $\calT^n$, enclosing a line mesh $\calL^n$ for $\Gamma^n$. 
\end{algorithmic}
\end{algorithm}\par\medskip

\subsection{Generation of two level set functions from a meshed representation}\label{sec.gen2LS}

\noindent Let the computational domain $D$ be equipped with a triangular mesh $\calT$, and let $\Gamma \subset D$ be an oriented open curve,
supplied as a line mesh $\calL$; for the purpose of this section, the edges of $\calL$ may not belong to $\calT$. 
We aim to generate two level set functions $\phi$, $\psi: D \to \R$ at the vertices of $\calT$, representing $\Gamma$ via \cref{eq.2LS} and satisfying the orthogonality relation \cref{eq.orthoLS}.  

This section presents a general and efficient algorithm to achieve this seldom considered task, to the best of our knowledge. Since its more intricate 3d instance will be thoroughly described in a dedicated article \cite{dapogny2025levelset}, we limit ourselves with a brief description in the present 2d context.

The key building block is the celebrated Fast Marching Method \cite{kimmel1998computing,sethian1996fast,sethian1999fast}, which calculates the (unsigned) distance function $d(\cdot,K)$ to a subset $K \subset D$ at the vertices of the mesh $\calT$ of $D$. 
In a nutshell, this method starts with the (unexpensive) calculation of the exact distance to $K$ at the ``close'' vertices of the elements of $\calT$ intersecting $K$.  
This ``nearby'' information is then propagated to the whole mesh $\calT$: iteratively, the algorithm computes ``trial'' values from the known distance values and the smallest of them is definitely accepted. Crucially, the method accepts the vertices of $\calT$ in order, from those closer to those farther from $K$; this feature has inspired a variant of the method, which concurrently computes the normal extension to the whole mesh $\calT$ of a quantity defined on $K$ \cite{adalsteinsson1999fast}.\par\medskip

Our algorithm for calculating two level set functions $\phi$ and $\psi$ representing the open curve $\Gamma$ proceeds in three stages, that are illustrated on \cref{fig.illusmshdist}.
\begin{itemize}
\item \textit{Step 1:} The unsigned distance function $d(\cdot,\Gamma)$ is calculated thanks to the Fast Marching Method. In doing so, whenever a vertex $x \in \calT$ is accepted by the procedure, the value of the distance is stored in an intermediate level set function $\phi^{\text{temp}}$, and it is endowed with a sign, depending on the orientation of $x$ with respect to $\Gamma$ -- an information which is available thanks to the ordered travel of vertices guaranteed by the Fast Marching Method. By construction, the $0$ level set of $\phi^{\text{temp}}$ extends $\Gamma$ 
into a curve $\widetilde \Gamma$ which is closed, possibly because it goes up to the boundary of the computational domain $D$, see \cref{rem.meetbdy}. 
\item \textit{Step 2:} The Fast Marching Method is used to calculate the (geodesic) signed distance to $\Gamma$ within $\widetilde\Gamma$. This information is stored at the vertices of the triangles in the set $\calK$ of those intersecting the $0$ level set of $\phi^{\text{temp}}$. This step yields a function $\psi^{\text{temp}}$ which contains the values of this signed distance to $\Gamma$ within $\widetilde \Gamma$, computed at the vertices of the triangles in $\calK$. 
\item \textit{Step 3:} We use once more the Fast Marching Method to compute the signed distance function to $\widetilde{\Gamma}$, which is stored in $\phi$. 
Meanwhile, we extend the function $\psi^{\text{temp}}$ from $\widetilde{\Gamma}$ to the whole mesh $\calT$ in the normal direction to $\widetilde\Gamma$, which yields the secondary level set function $\psi$.
\end{itemize}
This strategy guarantees that the curve $\widetilde\Gamma$ extends $\Gamma$ by a straight line in the neighborhood of $\Sigma$. Moreover, the orthogonality relation \cref{eq.orthoLS} between the $0$ level sets of $\phi$ and $\psi$ is satisfied by construction. 

 This algorithm is implemented in the open-source library \texttt{mshdist} \cite{dapogny2012computation}, which is part of the \texttt{ISCD Toolbox} \cite{iscdbox}. 
 
\begin{figure}[!ht]
    \centering
\includegraphics[width=1.0\textwidth]{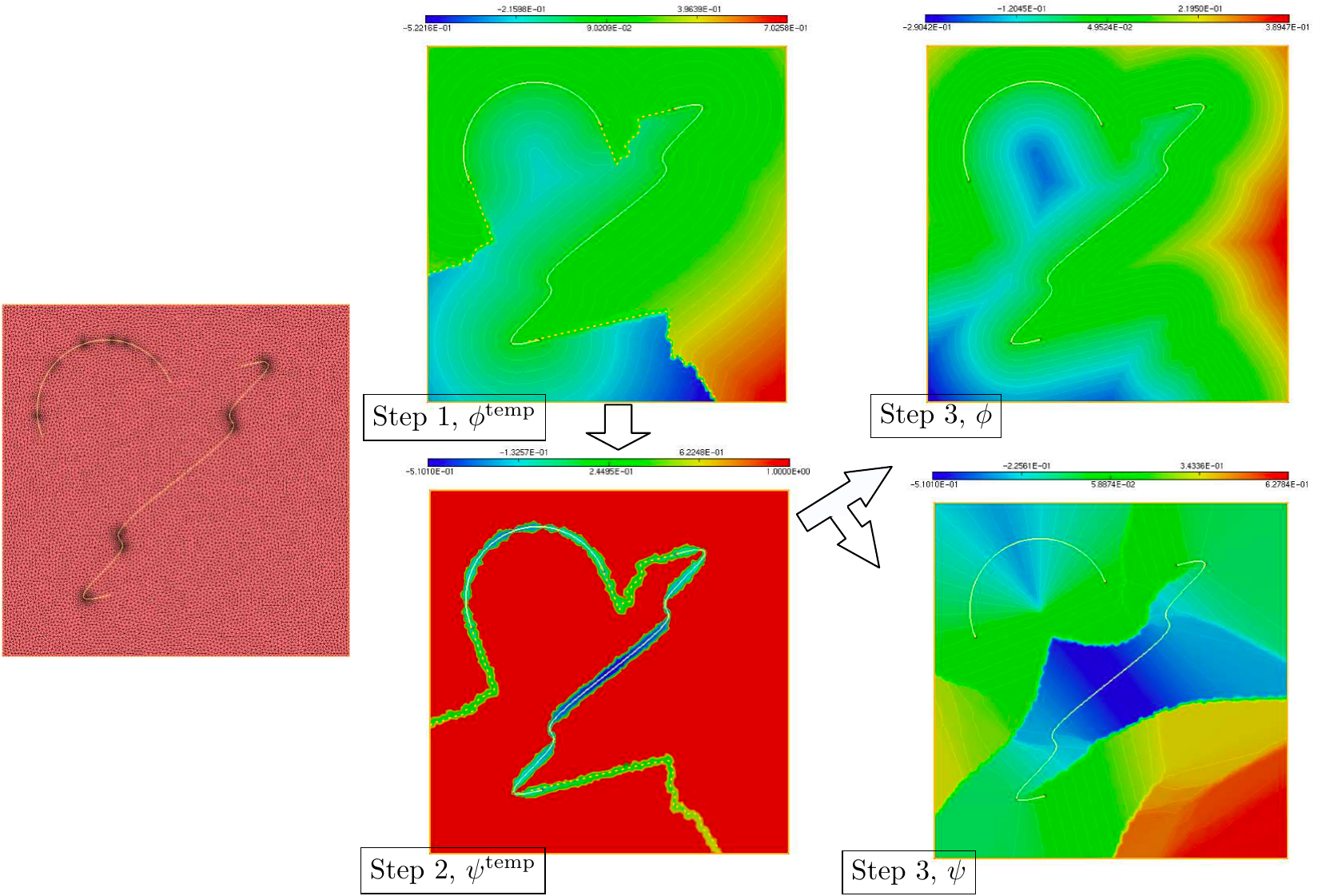}
 \caption{\it Illustration of the numerical algorithm of \cref{sec.gen2LS} for computing two level set functions associated to an input line mesh of the curve $\Gamma$ (in yellow); its closed extension $\widetilde{\Gamma}$ is depicted as a yellow dotted curve.}
  \label{fig.illusmshdist}
\end{figure}

\subsection{Resolution of the level set evolution equations}\label{sec.lsadv}

\noindent Let $\Gamma(t)$ be an open curve, evolving through a generic time period $(0,T)$ according to a stationary velocity field $V(x)$. 
For any $t \in (0,T)$, let $\phi(t,\cdot)$ and $\psi(t,\cdot)$ be two level set functions for $\Gamma(t)$, and let $\phi_0$, $\psi_0$ be two level set functions for the initial curve $\Gamma(0)$. 
We aim to solve the advection equations \cref{eq.2lsadvect} accounting for the evolutions of $\phi$ and $\psi$ on a mesh $\calT$ of the computational domain $D$; for definiteness, we focus on that involving $\phi$:
\begin{equation}\label{eq.advphi}
\left\{ 
\begin{array}{cl}
\frac{\partial\phi}{\partial t}(t,x) + V(x) \cdot \nabla \phi(t,x) = 0 & \text{for } t \in (0,T), \: x \in D, \\
\phi(0,x) = \phi_0(x) & \text{for } x \in D, 
\end{array}
\right. 
\end{equation}

In our framework, this equation is solved thanks to the well-known method of characteristics \cite{pironneau1989finite,strain1999semi}. The latter relies on the following explicit expression of the solution: 
\begin{equation}\label{eq.phicharac}
\phi(t,x) = \phi_0(X_t(0,x)), \:\: t \in (0,T), \:\: x \in D,
\end{equation}
where, for any time $t_0$, the mapping $t\mapsto X_{t_0}(t,x)$ is the characteristic curve of the vector field $V$ emerging from $x$ at time $t_0$, that is:
\begin{equation}\label{eq.characcurveexpl}
\left\{
\begin{array}{cl}
\frac{\d X_{t_0}}{\d t} (t,x) = V(t,X_{t_0}(t,x)) & \text{for } t \in (0,T), \\
X_{t_0}(t_0,x) = x;&
\end{array}
\right.
\end{equation}
intuitively, $X_{t_0}(t,x)$ is the position at time $t$ of a particle lying in $x$ at time $t_0$. 
Thus, $\phi$ account for a transport of the quantity $\phi_0$ along the trajectories induced by the velocity field $V$. 

The numerical resolution of \cref{eq.advphi} is based on the direct evaluation of the formula \cref{eq.phicharac} at $t=T$: for each vertex $x$ of $\calT$, the ordinary differential equation \cref{eq.characcurveexpl} for the trajectory $t\mapsto X_T(t,x)$ is solved thanks to a Runge-Kutta 4 method, and the initial function $\phi_0$ is interpolated at the resulting point from its values at the vertices of $\calT$.

This algorithm is implemented in the open-source code \texttt{Advection} described in \cite{bui2012accurate}, which is part of the \texttt{ISCD Toolbox} \cite{iscdbox}.

\subsection{Discretization of an open curve into a mesh from a two-level set representation}\label{sec.remesh}

\noindent Let the computational domain $D$ be equipped with a triangular mesh $\calT$, and let $\Gamma \subset D$ be an open curve. 
The latter is defined by the datum of two level set functions $\phi, \psi : D \to \R$ that are discretized at the vertices of $\calT$ and interpolated linearly from these values when evaluated inside the elements of $\calT$.
We aim to create a new, high-quality mesh $\widetilde{\calT}$ of $D$ in which $\Gamma$ is explicitly discretized. 

We proceed in two steps, that are illustrated on \cref{fig.remesh} and described with a little more details in the next sub-sections. 
\begin{enumerate}
\item We split each triangle $T \in \calT$ intersected by $\Gamma$ in such a way that the mesh $\calT^{\text{temp}}$ contains an explicit discretization of $\Gamma$ as a collection $\calLt$ of edges. This valid and conforming mesh is unfortunately of very low quality. 
\item We apply local modifications to $\calT^{\text{temp}}$ to obtain a new, high-quality mesh $\widetilde\calT$ of $D$ which still features an explicit discretization $\widetilde{\calL}$ of $\Gamma$.
\end{enumerate}

\begin{figure}[!ht]
    \centering
    \begin{tabular}{ccc}
\begin{minipage}{0.28\textwidth}
\centering
\begin{overpic}[width=0.8\textwidth]{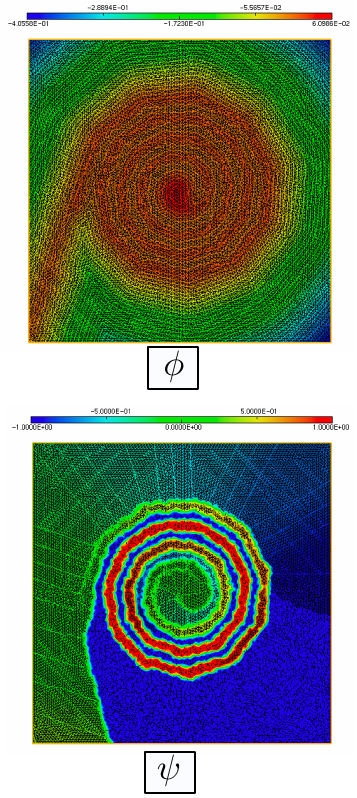}
\put(0,-3){\fcolorbox{black}{white}{a}}
\end{overpic}
\end{minipage} & 
\begin{minipage}{0.33\textwidth}
\centering
\begin{overpic}[width=1.0\textwidth]{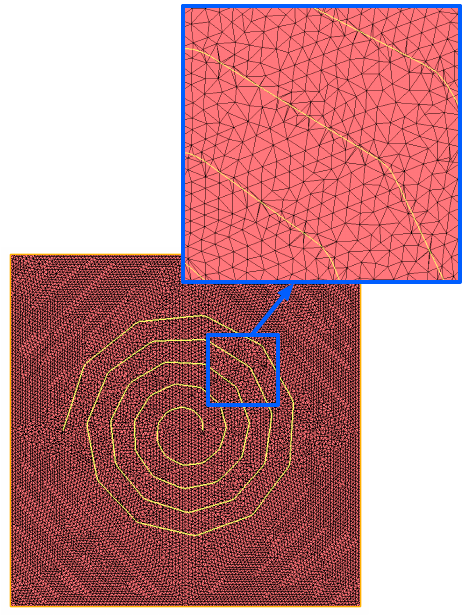}
\put(0,-3){\fcolorbox{black}{white}{b}}
\end{overpic}
\end{minipage}  & 
\begin{minipage}{0.33\textwidth}
\centering
\begin{overpic}[width=1.0\textwidth]{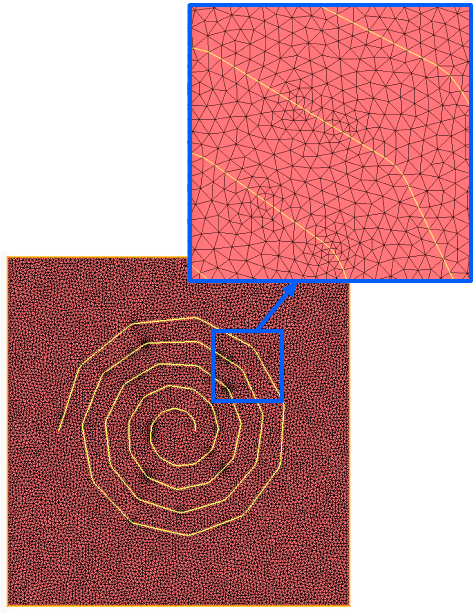}
\put(0,-3){\fcolorbox{black}{white}{c}}
\end{overpic}
\end{minipage} 
\end{tabular}  
 \caption{\it Construction of a high-quality mesh $\widetilde\calT$ of $D$ containing an explicit discretization of an open curve $\Gamma$ in \cref{sec.remesh}; (a) Two level set representation $\phi$, $\psi$ of $\Gamma$ on an initial mesh $\calT$ of $D$; (b) Intermediate, low-quality mesh $\calT^{\text{temp}}$ of $D$ which is body-fitted to $\Gamma$; (c) Desired high-quality mesh $\widetilde\calT$ of $D$ body-fitted to $\Gamma$.}
  \label{fig.remesh}
\end{figure}

\begin{remark}
This operation is interesting in a number of applications. For instance, it has recently been used in geophysics to discretize a network of faults in the underground, see \cite{legentil2022testing} and the recent three-dimensional work \cite{belhachmi2025tetrahedral}. This algorithm has also been used in the article \cite{feppon2026fractured}, devoted to a numerical method for solving boundary value problems on such fractured meshes.
\end{remark}
 
\subsubsection{Explicit discretization of an implicitly defined open curve into the computational mesh}

\noindent This stage is a variation of the marching cubes algorithm \cite{lorensen1987marching} and its simplicial version \cite{chan1998new,doi1991efficient}, designed for isosurface extraction. 
It starts from a mesh $\calT$ of the computational domain $D$, and the datum of two level set functions $\phi$, $\psi$, defined at its vertices, accounting for the curve $\Gamma$. We form the band $\calB$ of the elements $T \in \calT$ surrounding $\Gamma$, i.e. satisfying the following two conditions:
\begin{enumerate}[(i)]
\item A least one vertex of $T$ has a positive value of $\phi$, and at least another vertex has a negative value of $\phi$, i.e. $T$ intersects the $0$ level set of $\phi$;
\item At least one vertex of $T$ bears a negative value of $\psi$, i.e. $T$ lies inside or on the boundary of the negative subdomain of $\psi$.
\end{enumerate}

Then, for each triangle $T \in \calB$, we detect the intersection of the isoline $\left\{ \phi = 0 \right\}$ with the edges of $T$ by linear interpolation of the values of $\phi$ at its vertices. 
A pattern is used to split the triangles sharing at least one of the identified edges by this process in such a way that $\Gamma$ appears explicitly in the resulting mesh. This operation concerns not only the triangles of $\calB$, but also the triangles that are adjacent to those in $\calB$, whose splitting is needed to ensure conformity of the resulting mesh in spite of the fact that they do not intersect the curve $\Gamma$.

This stage results in a mesh $\calTt$ of $D$, which is valid, conforming, and contains an explicit discretization of $\Gamma$ as a collection $\calLt$ of edges.
Unfortunately, $\calTt$ is bound to be ill-shaped -- i.e. to contain very flat, nearly degenerate elements -- since the relative positions of $\Gamma$ and the vertices of the elements of $\calB$ are arbitrary, see \cref{fig.remesh} (b).

\subsubsection{Quality-oriented remeshing}

\noindent In this second stage, we repeatedly apply four local mesh modification operators towards improving the quality of the elements of $\calTt$, while ensuring a fine representation of the geometry of $\Gamma$:
\begin{itemize}
\item \textit{Edge split:} A ``long'' edge in the mesh is split into two, after addition of a new vertex in the mesh. All the triangles sharing this edge are split accordingly.
\item \textit{Edge collapse:} One of the two endpoints of a ``short'' edge is merged with the other, and the vertices of the attached triangles are updated accordingly. 
\item \textit{Edge swap:} An edge shared by two triangles is flipped so as to connect the other two vertices of the configuration.
\item \textit{Vertex relocation:} A vertex is moved, while all the connectivities in the mesh remain untouched. 
\end{itemize}
We refer to classical textbooks about meshing such as \cite{borouchaki2017meshing,frey2007mesh} for a more detailed presentation of these operations.

This stage yields the desired high-quality mesh $\widetilde{\calT}$ of $D$, where a sub-collection $\widetilde{\calL}$ of edges explicitly accounts for $\Gamma$, see \cref{fig.remesh} (c).
Note that this quality-oriented remeshing procedure allows to adapt the size of the elements of $\widetilde{\calT}$ to a user-defined size prescription, so as to improve the accuracy of geometric or mechanical computations.


\section{Applications examples in physical simulations}\label{sec.appphys}


\noindent This section presents two ``simple'' illustrations of our numerical methodology for tracking the motion of an open curve $\Gamma$ in 2d. 
All the computations are conducted on a standard \texttt{Apple MacBookPro}
laptop with a 2 GHz Quad-Core Intel Core i5  processor and 16 GB of memory.
The first \cref{sec.numvor} aims to evaluate its efficiency on the academic problem where $\Gamma$ evolves under an analytical velocity field. 
\cref{sec.vortex} then deals with an application in fluid mechanics: $\Gamma$ represents a two-dimensional vortex sheet
whose velocity is given by a weakly singular curve integral. 

\subsection{Evolution of an open curve under an analytical velocity field}\label{sec.numvor}

\noindent This first example aims to appraise the accuracy of our numerical algorithm. It draws inspiration from a classical benchmark test-case proposed in \cite{leveque1996high} to compare the accuracy of numerical methods for the resolution of the advection equation. 

The situation takes places in the 2d unit square $D = (0,1)^2$. The considered curve $\Gamma(t)$ evolves through the time period $(0,T)$, with $T=1.2$, starting from the initial configuration
$$\Gamma(0) := \Big\{(0.1+0.8s,0.8) ,\:\: s \in (0,1) \Big\} \cup \Big\{(0.4+0.2s,0.3+0.15s) ,\:\: s\in (0,1)\Big\},$$ 
represented on \cref{fig.vor} (a), according to the time-dependent velocity field $V(t,x)$ defined by: 
$$\forall t \in (0,T) ,\:\: x= (x_1,x_2) \in D, \quad  V(t,x)= \left(
\begin{array}{c}
\sin(4\pi(x_1+0.5))  \sin(4\pi(x_2-0.3))  \cos(\pi\frac{t}{T}) \\
\cos(4\pi(x_1+0.5)) \cos(4\pi(x_2-0.3)) \cos(\pi\frac{t}{T}) 
\end{array}
\right).$$
The latter induces a large deformation of $\Gamma(t)$ until the time $t=T/2$; its antisymmetry with respect to $t=T/2$ then causes the curve to return to its initial configuration at time $T$. 

We apply our numerical \cref{algo.openls} to track this motion, choosing a subdivision of $(0,T)$ with $N= 200$ intermediate times.
The remeshing algorithm is required to create elements with minimum and maximum sizes $0.006$ and $0.009$, respectively, and the largest mesh produced in the course of the evolution (at $t=T/2$) has $128,729$ vertices. The total computational time equals approximately $40$ min.

\begin{figure}[!ht]
    \centering
    \begin{tabular}{ccc}
\begin{minipage}{0.3\textwidth}
\centering
\begin{overpic}[width=1.0\textwidth]{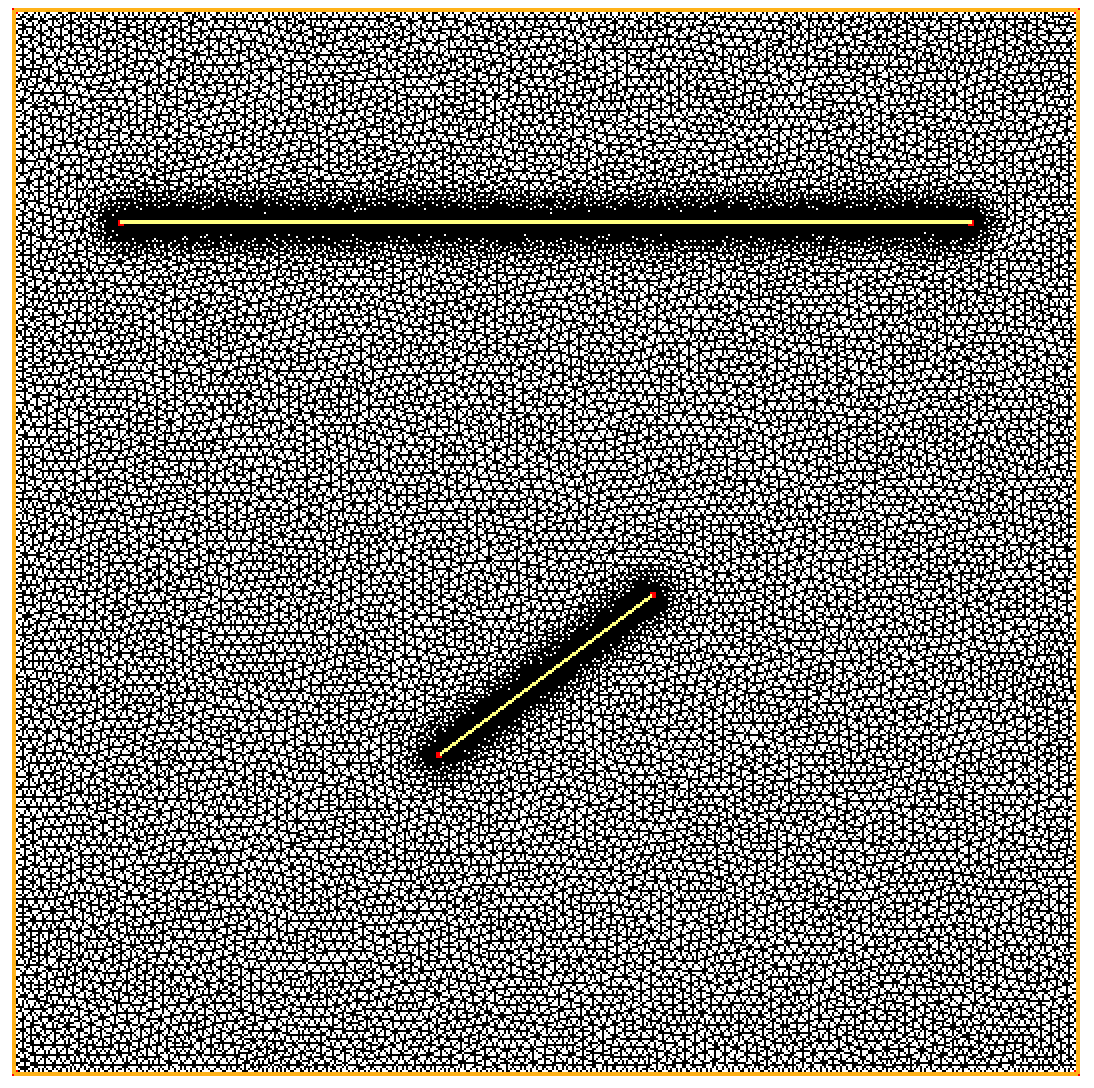}
\put(0,-3){\fcolorbox{black}{white}{$t=0$}}
\end{overpic}
\end{minipage} & 
\begin{minipage}{0.3\textwidth}
\centering
\begin{overpic}[width=1.0\textwidth]{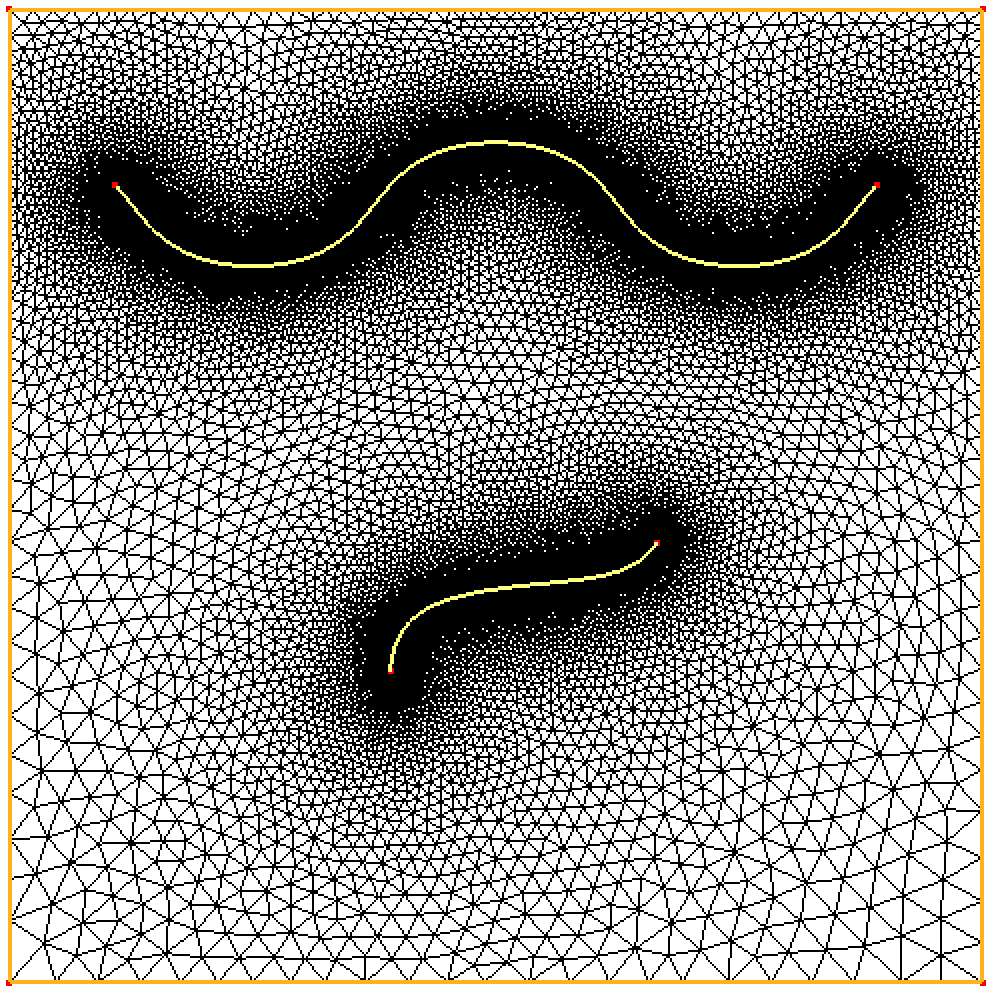}
\put(0,-3){\fcolorbox{black}{white}{$t=0.07$}}
\end{overpic}
\end{minipage}  
&
\begin{minipage}{0.3\textwidth}
\centering
\begin{overpic}[width=1.0\textwidth]{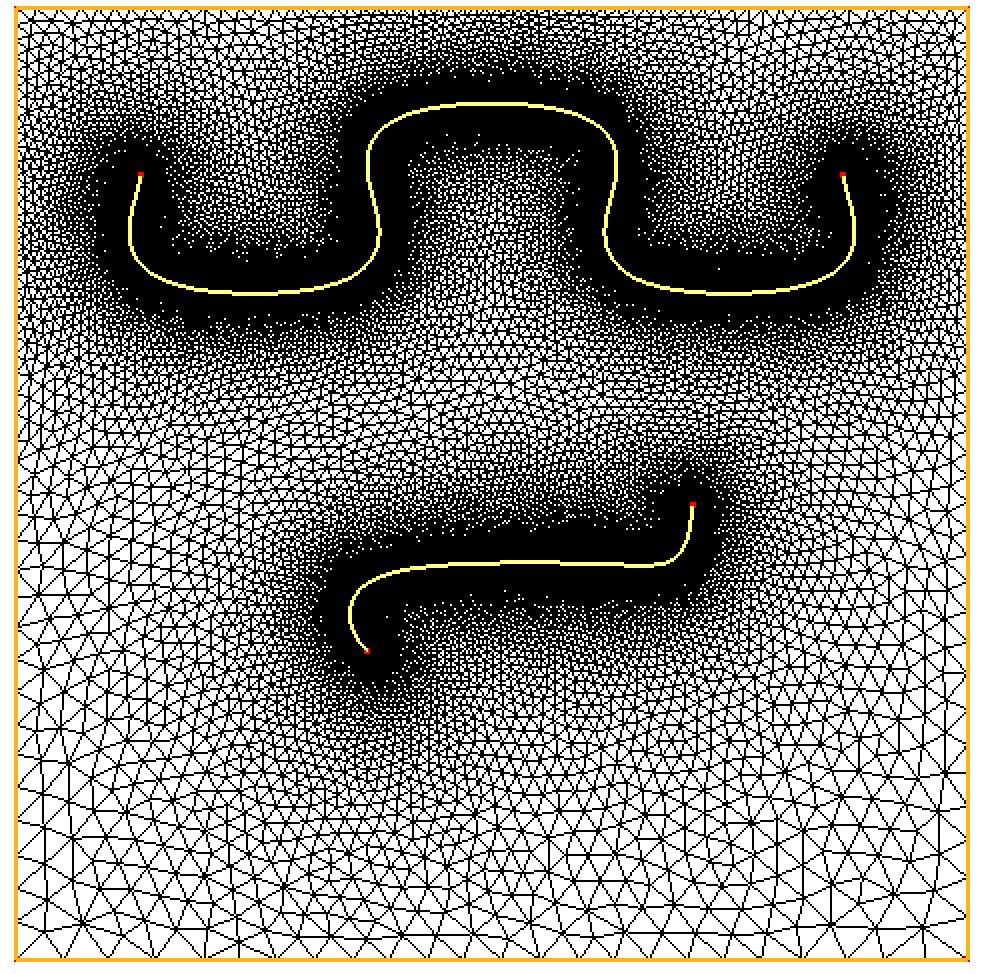}
\put(0,-3){\fcolorbox{black}{white}{$t=0.15$}}
\end{overpic}
\end{minipage}
\end{tabular} \par\bigskip
 \begin{tabular}{ccc}
 \centering
\begin{minipage}{0.3\textwidth}
\centering
\begin{overpic}[width=1.0\textwidth]{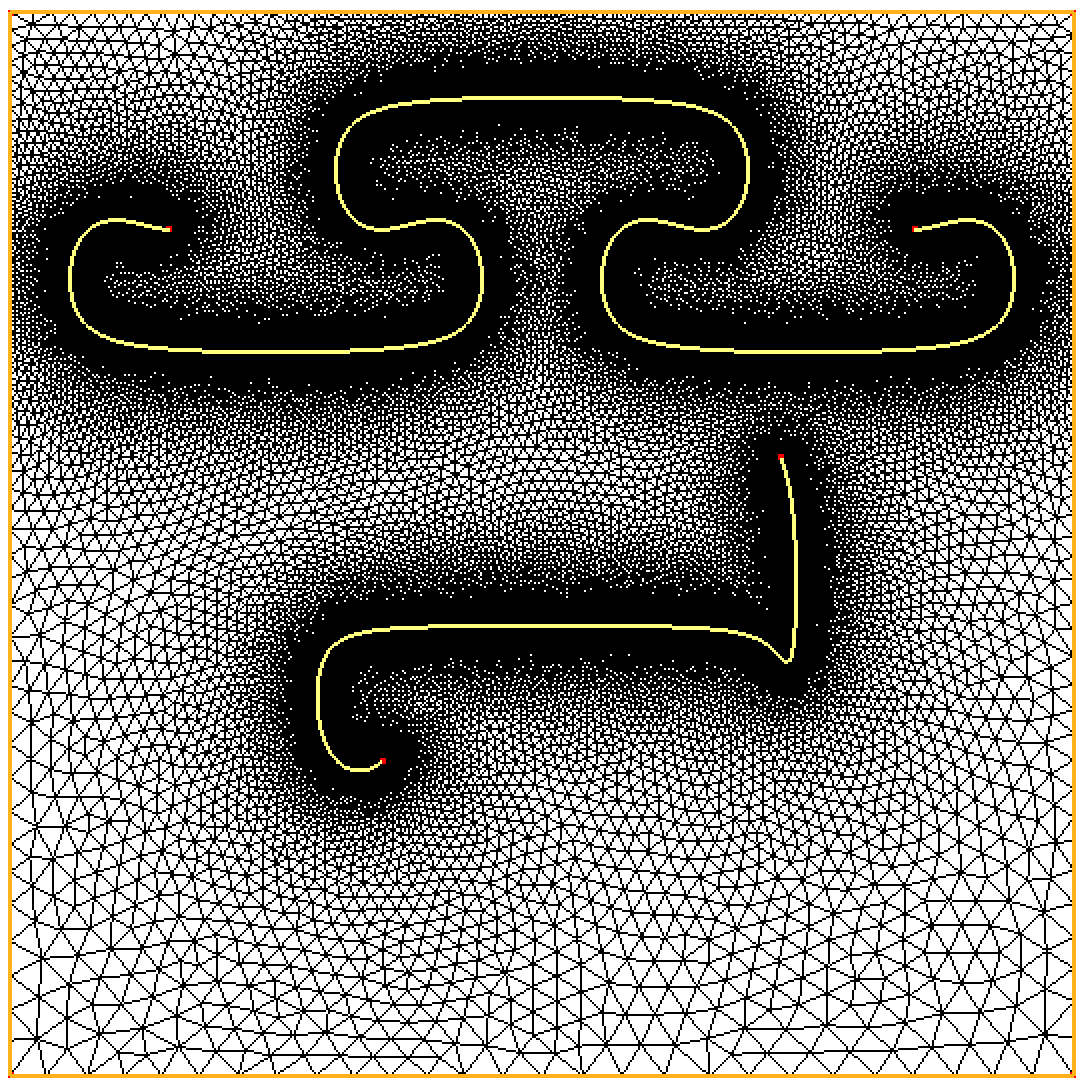}
\put(0,-3){\fcolorbox{black}{white}{$t=0.3$}}
\end{overpic}
\end{minipage} 
&
\begin{minipage}{0.3\textwidth}
\centering
\begin{overpic}[width=1.0\textwidth]{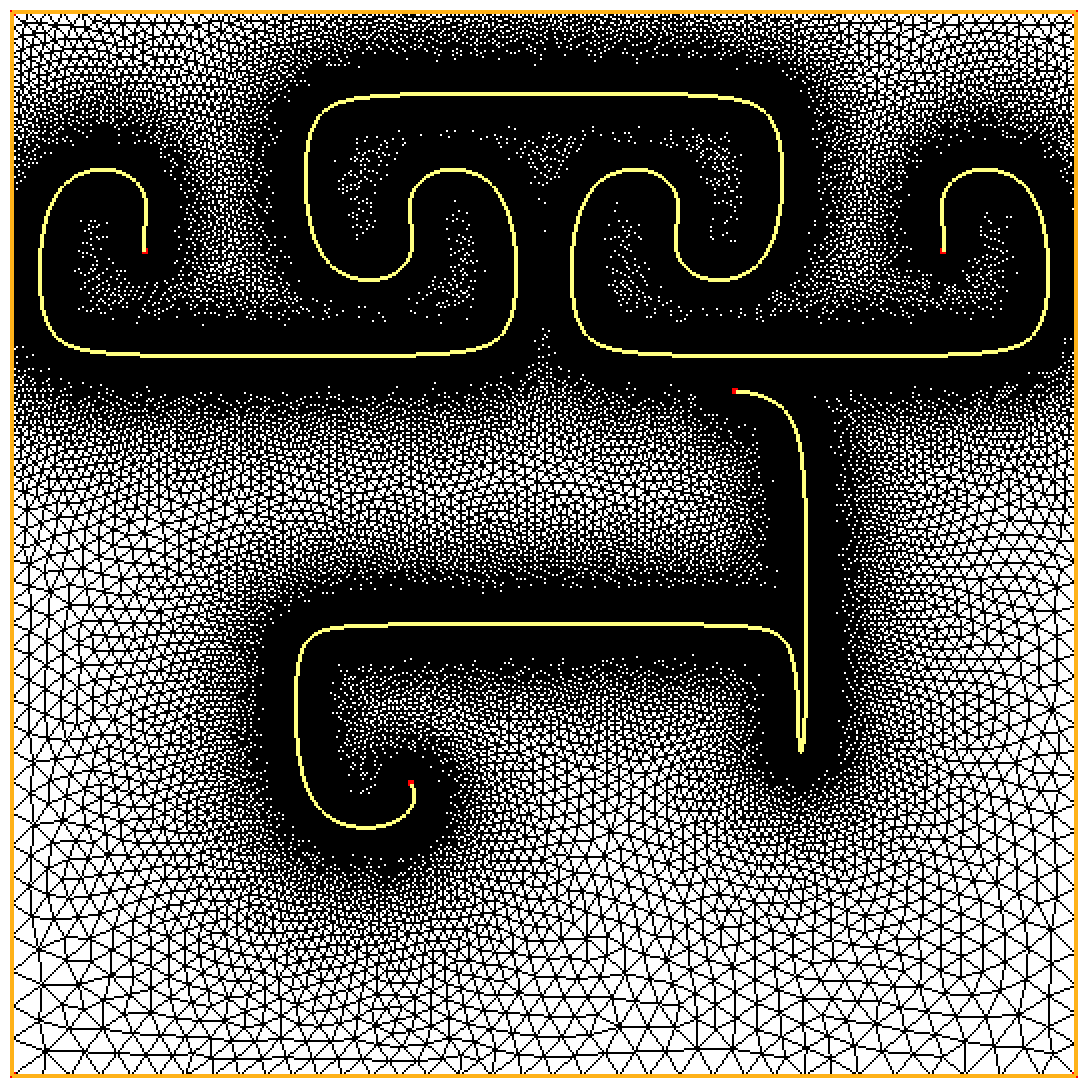}
\put(0,-3){\fcolorbox{black}{white}{$t=0.6$}}
\end{overpic}
\end{minipage} & 
\begin{minipage}{0.3\textwidth}
\centering
\begin{overpic}[width=1.0\textwidth]{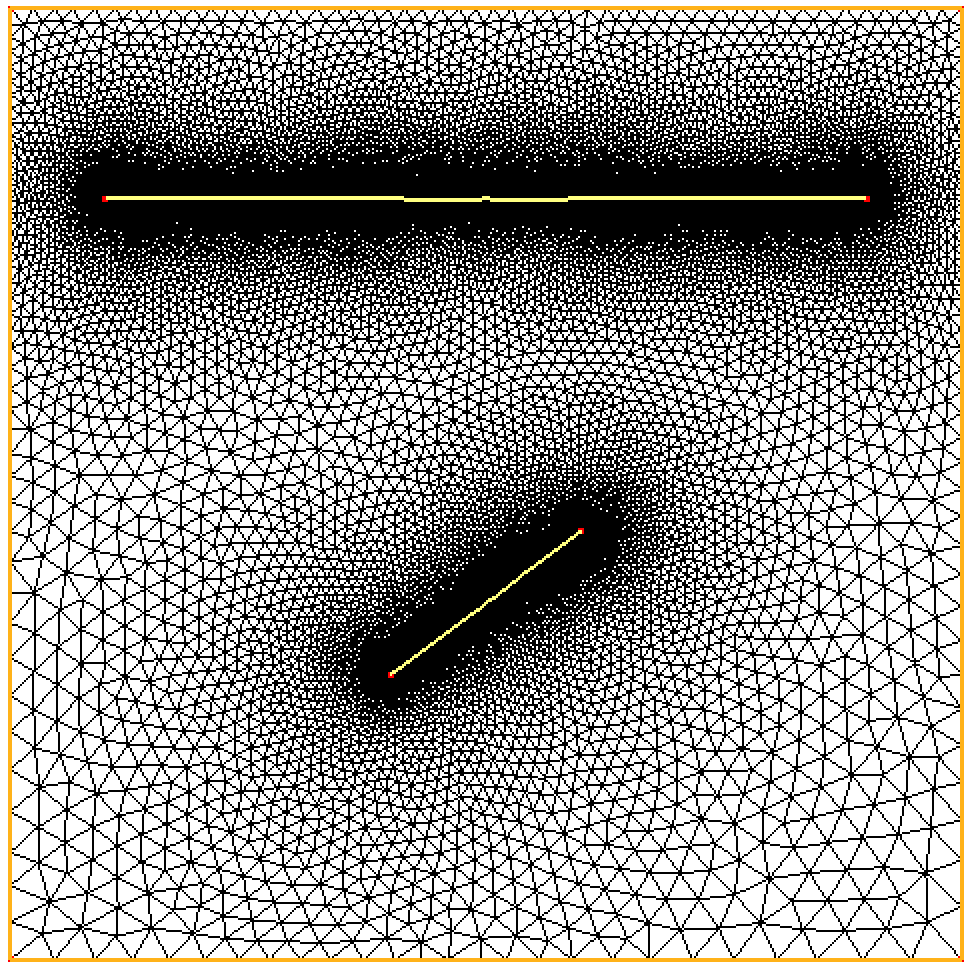}
\put(0,-3){\fcolorbox{black}{white}{$t=1.2$}}
\end{overpic}
\end{minipage} 
\end{tabular}  
 \caption{\it A few snapshots of the curve $\Gamma(t)$ evolving under the analytical velocity field considered in \cref{sec.numvor}.}
  \label{fig.vor}
\end{figure}

Since the considered motion satisfies $\Gamma(0)= \Gamma(T)$ at the continuous level, 
the accuracy of this numerical simulation can be measured in terms of the Hausdorff distance $d^H(\Gamma(0), \Gamma(T))$ between the initial and final configurations of $\Gamma(t)$: 
$$d^H(\Gamma(0), \Gamma(T)) = \max \Big( \rho(\Gamma(0),\Gamma(T)),  \rho(\Gamma(T),\Gamma(0)) \Big), \text{ where } \rho(K_1, K_2) := \sup\limits_{x \in K_1} d(x,K_2). $$
In the above situation, this error equals $d^H(\Gamma(0), \Gamma(T)) =1.034$e$^{-3}$, which is a much lower value than the minimum size of an edge in the mesh, in spite of the extreme stretching undergone by $\Gamma$ in the course of the evolution, 
which validates the accuracy of our method.

\subsection{Vortex sheet roll-up}\label{sec.vortex}

\noindent In this section, we apply our numerical tracking strategy to the simulation of the dynamics of a vortex sheet in a nearly inviscid 2d fluid; 
the velocity of such a one-dimensional structure is calculated as a line integral on the evolving curve.
The vortex sheet phenomenon has been the subject of extensive investigations: without any claim for exhaustivity, we refer to \cite{saffman1992vortex} and Chap. 6 of \cite{marchioro2012mathematical} about its physical modelling, to \cite{cottet2001vortex,majda2002vorticity,saffman1992vortex} about its mathematical analysis, and to \cite{krasny1986study,krasny1987computation} for numerical simulations by Lagrangian methods; see also \cite{harabetian1996eulerian} for Eulerian simulations of closed vortex sheets via the ``classical'' Level Set Method. 

The phenomenon under scrutiny originally takes place in a 3d medium filled with an incompressible and nearly inviscid fluid. A vortex sheet is a surface along which the velocity ``slips'', having discontinuous tangential component and continuous normal component. Such a pattern typically shows up in the wake of an aircraft, or at the limit between two immiscible fluids with different velocities, see \cref{fig.vorset} (a) for an illustration.
In particular configurations, taking advantage of symmetries allows to reduce this situation to that of an open curve $\Gamma(t)$ evolving within a 2d fluid medium $D \subset \R^2$, as we now consider.

For completeness, a few details about the 2d physical model are sketched in \cref{app.vortex}. For the purpose of this section, let us solely point out that
the dynamics of the fluid and of the discontinuity line $\Gamma(t)$ are governed by the vortex strength $\gamma(t,x)$, 
a scalar quantity defined for $x\in \Gamma(t)$ which is directly related to the jump in tangential velocity of the fluid. Precisely, the velocity $u(t,x)$ of the fluid surrounding $\Gamma(t)$ is given by the following formula: 
\begin{equation}\label{eq.velvor}
u(t,x) = \frac{1}{2\pi}\int_{\Gamma(t)} \gamma(t,y) \frac{(x-y)^\perp}{\lvert x - y \lvert^2}\:\d s(y), \quad t > 0, \:\: x \in D \setminus \overline{\Gamma(t)}. 
\end{equation}
Here, $v^\perp = (-v_2,v_1)$ stands for the $90^{\degree}$ counterclockwise rotate of a vector $v = (v_1,v_2)$. As throughout the article, we denote by $\d s$ the integration measure on a codimension $1$ subset of the plane $\R^2$ (i.e. a curve), out of consistency with the general case of a $d$-dimensional ambient space.

The velocity field driving the evolution of $\Gamma(t)$ itself has the same expression \cref{eq.velvor}, and with a small abuse of notation, it is also denoted by $u(t,x)$. In this latter case, however, 
the integrand in \cref{eq.velvor} is not absolutely integrable, and the formula is understood as a Cauchy principal value.
As is customary in the literature, the numerical evaluation of \cref{eq.velvor} relies on the so-called ``vortex-blob'' method, which alleviates the singularity of the kernel thanks to the following smooth approximation:
$$u(t,x) \approx \frac{1}{2\pi} \int_{\Gamma(t)} \gamma(t,y) \frac{(x-y)^\perp}{\lvert x - y \lvert^2 + \delta^2}\:\d s(y),$$
where $\delta>0$ is a small parameter, see \cite{krasny1986study}.

The description of the evolution of $\Gamma(t)$ is completed by an equation about the vortex strength $\gamma(t,x)$. 
The latter stems from the conservation of circulation within each small portion of $\Gamma(t)$ as it is driven by the fluid, which reads:
$$\lvert \com(\nabla X(t,0,x)) n \lvert \gamma(t,X(t,0,x)) = \gamma(0,x), \quad x \in \Gamma(0), $$
where $t \mapsto X_0(t,x)$ is the characteristic curve of $u(t,x)$ emerging from a point $x \in \Gamma(0)$, see \cref{eq.phicharac}. 

We apply the methodology of \cref{sec.algo} to track the motion of a vortex sheet, in a particular physical configuration considered in \cite{krasny1987computation}.  
The initial shape of the vortex sheet is the straight line segment $\Gamma(0) = (-1,1) \times \left\{0\right\}$, and the initial vortex strength is weakly singular at its endpoints: 
\begin{equation}\label{eq.inivorstren}
\gamma(0,x) = -\frac{x_1}{(1-x_1^2)^{1/2}}, \quad x= (x_1,x_2) \in \Gamma(0).
\end{equation}
The final time of the simulation is $T=2.2$, and we use the time step $\Delta t= 0.02$. The vortex-blob parameter is set to $\delta =0.05$.
A few snapshots of the evolution process are represented on \cref{fig.vorres}, which show good agreement with the results of \cite{krasny1987computation}. The maximum number of vertices in a mesh equals $81,272$, and the total computation takes about $90$ min.

 \begin{figure}[!ht]
\centering
    \begin{tabular}{cc}
\begin{minipage}{0.51\textwidth}
\centering
\begin{overpic}[width=1.0\textwidth]{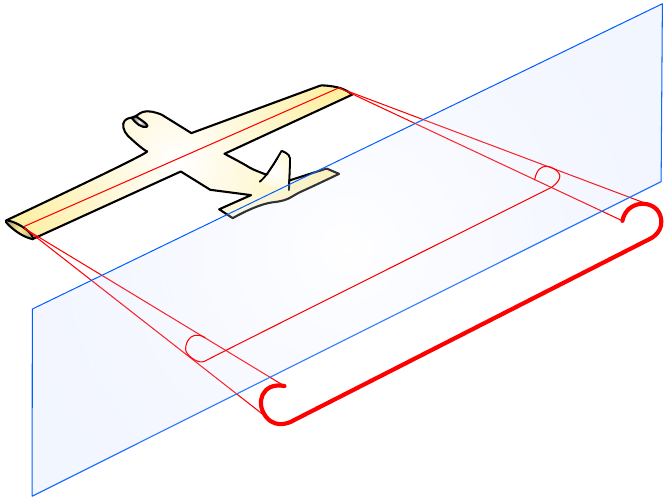}
\put(0,-3){\fcolorbox{black}{white}{a}}
\end{overpic}
\end{minipage} & 
\begin{minipage}{0.49\textwidth}
\centering
\begin{overpic}[width=1.0\textwidth]{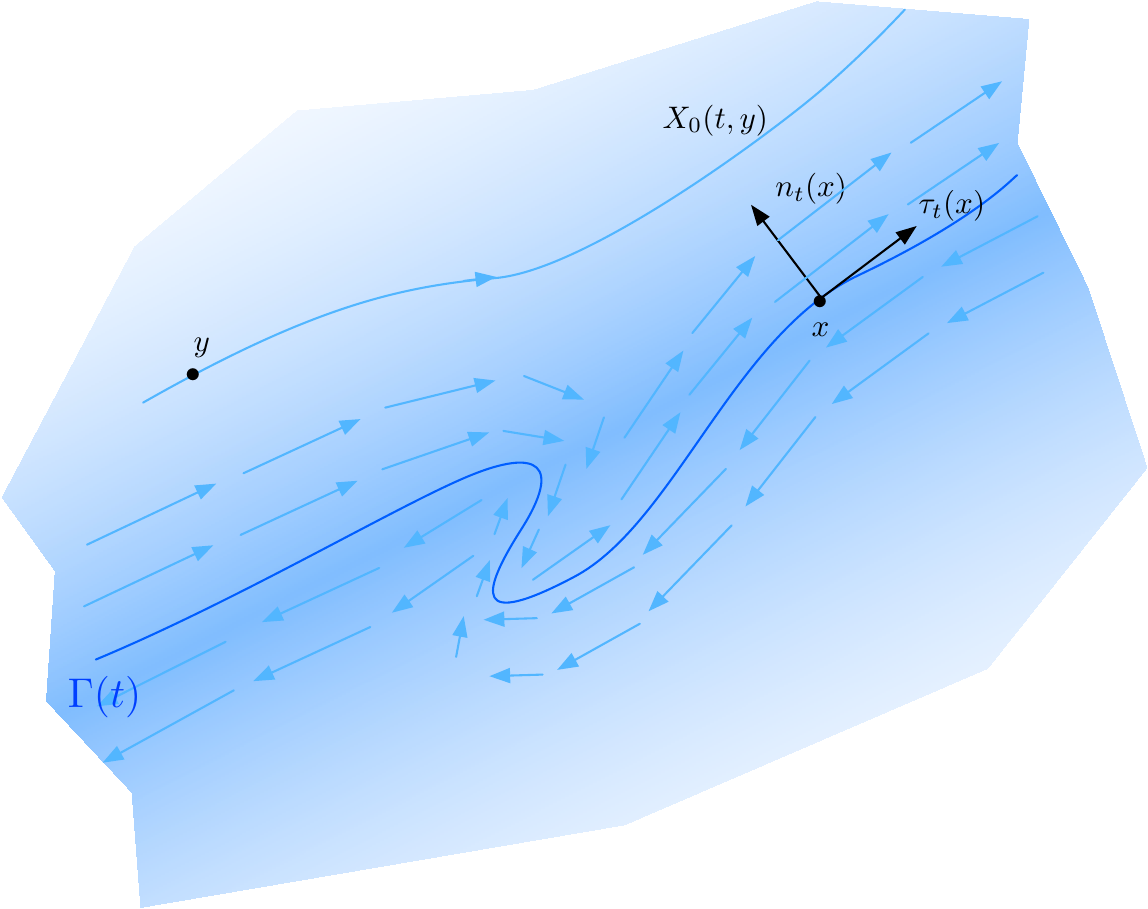}
\put(0,-3){\fcolorbox{black}{white}{b}}
\end{overpic}
\end{minipage}  
\end{tabular}  
\caption{\it (a) A 3d vortex sheet developing in the wake of an aircraft and its 2d sections; (b) Mathematical quantities associated to the description of a vortex sheet, as discussed in \cref{sec.vortex}.}
\label{fig.vorset} 
\end{figure}

\begin{figure}[!ht]
    \centering
    \begin{tabular}{cc}
\begin{minipage}{0.48\textwidth}
\centering
\begin{overpic}[width=1.0\textwidth]{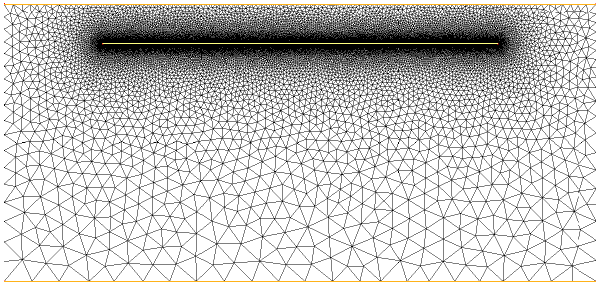}
\put(0,-3){\fcolorbox{black}{white}{$t=0$}}
\end{overpic}
\end{minipage} & 
\begin{minipage}{0.48\textwidth}
\centering
\begin{overpic}[width=1.0\textwidth]{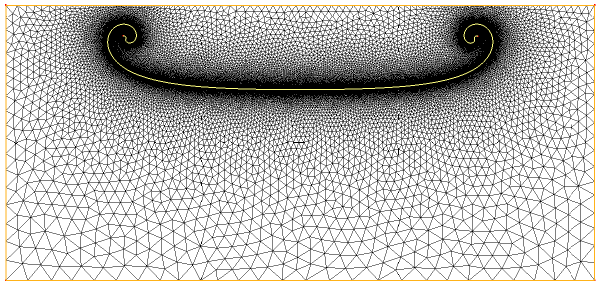}
\put(0,-3){\fcolorbox{black}{white}{$t=0.5$}}
\end{overpic}
\end{minipage}  
\end{tabular}  \par\bigskip 
    \begin{tabular}{cc}
\begin{minipage}{0.48\textwidth}
\centering
\begin{overpic}[width=1.0\textwidth]{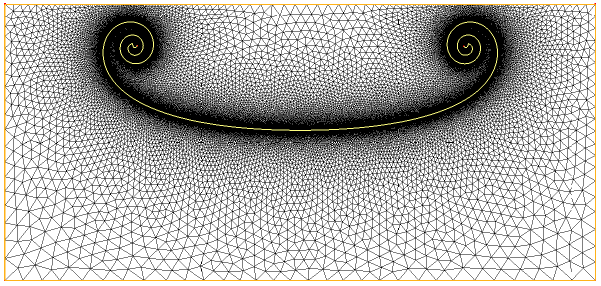}
\put(0,-3){\fcolorbox{black}{white}{$t=1$}}
\end{overpic}
\end{minipage} & 
\begin{minipage}{0.48\textwidth}
\centering
\begin{overpic}[width=1.0\textwidth]{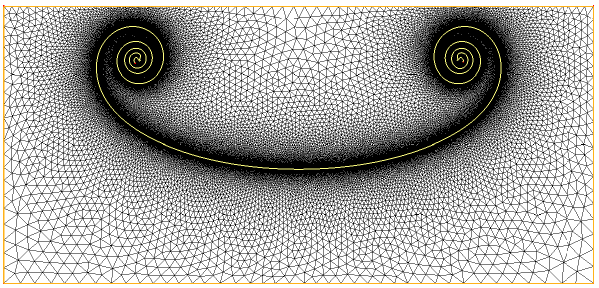}
\put(0,-3){\fcolorbox{black}{white}{$t=1.5$}}
\end{overpic}
\end{minipage}  
\end{tabular}  \par\bigskip
    \begin{tabular}{cc}
\begin{minipage}{0.48\textwidth}
\centering
\begin{overpic}[width=1.0\textwidth]{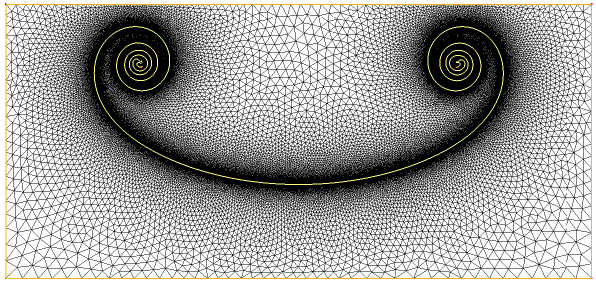}
\put(0,-3){\fcolorbox{black}{white}{$t=1.8$}}
\end{overpic}
\end{minipage} & 
\begin{minipage}{0.48\textwidth}
\centering
\begin{overpic}[width=1.0\textwidth]{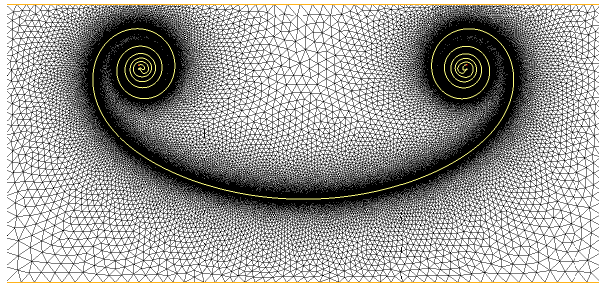}
\put(0,-3){\fcolorbox{black}{white}{$t=2$}}
\end{overpic}
\end{minipage}  
\end{tabular}  
\begin{center}
\begin{overpic}[width=1.0\textwidth]{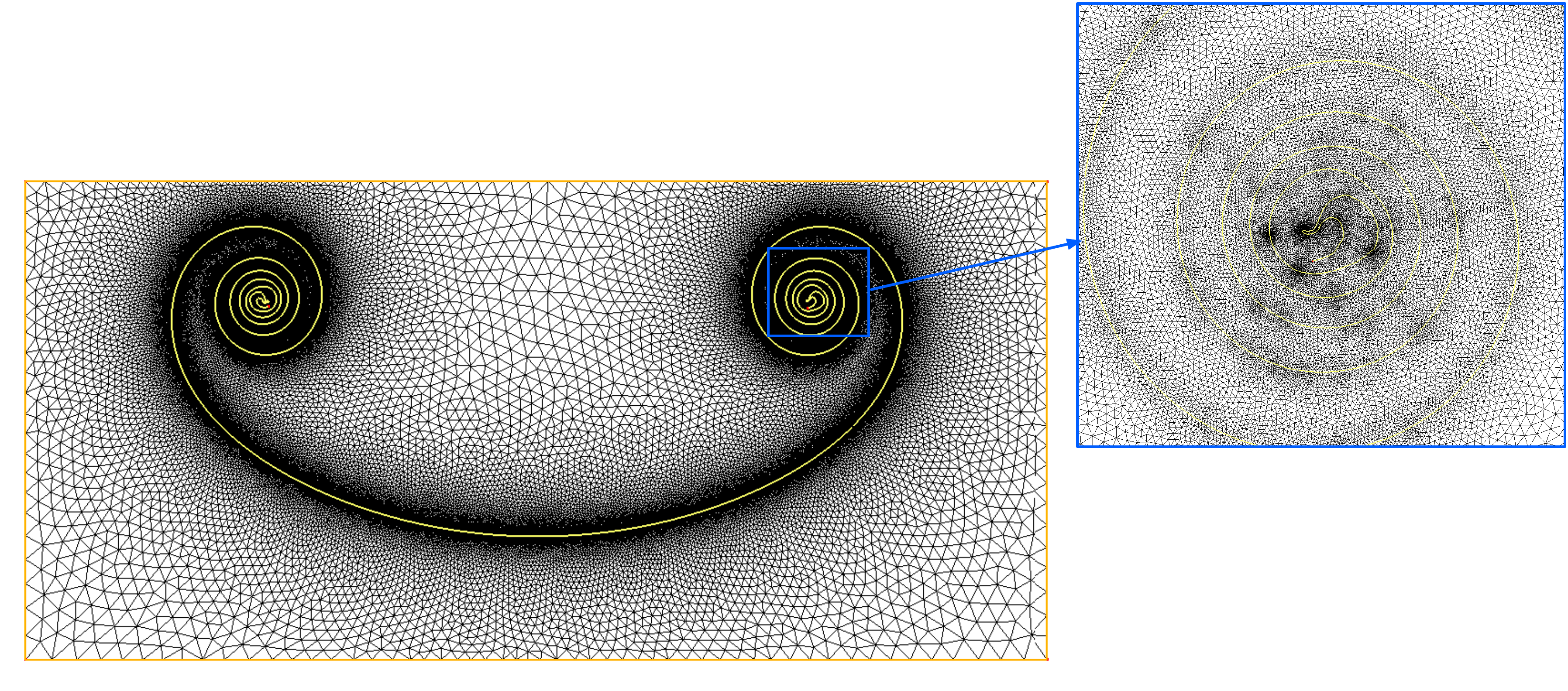}
\put(0,-3){\fcolorbox{black}{white}{$t=2.2$}}
\end{overpic}
\end{center}
 \caption{\it  A few intermediate configurations of the evolving vortex sheet considered in \cref{sec.vortex}.}
  \label{fig.vorres}
\end{figure}


\section{Applications in shape optimization}\label{sec.appso}


\noindent This section deals with the numerical resolution of optimization problems where the variable is the shape of an open curve $\Gamma$. 
These investigations raise the need for a little background about shape optimization, that we first present in \cref{sec.so}.
We then consider three different situations. In \cref{sec.aniper}, we optimize the length of an open curve with respect to a Riemannian metric of the plane; 
the next \cref{sec.am} deals with an application to the 3d printing of a structure: $\Gamma$ then represents the path of the laser used to raise the temperature of a target region of a metallic powder bed up to the fusion point. Finally, in \cref{sec.fault}, we use our methodology to solve an inverse problem of reconstruction of a fault in the underground from observational data.

\subsection{A primer about shape optimization}\label{sec.so}

\noindent The shape optimization problems under scrutiny are of the form
\begin{equation}\label{eq.sopb}
\min\limits_{\Gamma \subset D} \: J(\Gamma), 
\end{equation}
in which the variable $\Gamma$ is a 2d open curve (or a collection of such) in a fixed computational domain $D \subset \R^2$.
For simplicity, we omit constraints in this formulation, although their presence would not entail any additional conceptual difficulty, see e.g. \cite{feppon2020null} about this topic. 

The treatment of \cref{eq.sopb} calls for a notion of derivative for a function depending on an open curve,
and we rely on the boundary variation method of Hadamard \cite{allaire2020survey,allaire2007conception,henrot2018shape,murat1976controle,sokolowski1992introduction}.
In a nutshell, variations of a reference curve $\Gamma$ are considered of the form: 
$$\Gamma_\theta = (\Id + \theta) (\Gamma), \:\: \theta \in \calC^{1,\infty}(\R^2;\R^2), \:\: \lvert\lvert \theta \lvert\lvert_{\calC^{1,\infty}(\R^2;\R^2)} < 1, $$
i.e. each point $x$ of $\Gamma$ is perturbed according to a ``small'' vector field $\theta$ in the Banach space $\calC^{1,\infty}(\R^2;\R^2)$ of bounded vector fields with bounded derivatives. 
A function $J(\Gamma)$ is called shape differentiable at $\Gamma$ if the underlying mapping $\theta\mapsto J(\Gamma_\theta)$, from $\calC^{1,\infty}(\R^2;\R^2)$ into $\R$, is Fr\'echet differentiable at $\theta=0$. This gives rise to the following expansion: 
$$J(\Gamma_\theta) = J(\Gamma) + J^\prime(\Gamma)(\theta) + \o(\theta) , \text{ where } \frac{\o(\theta)}{\lvert\lvert \theta\lvert\lvert_{\Cinfty}} \xrightarrow{\theta\to 0} 0.$$
In practice, the knowledge of the shape derivative $J^\prime(\Gamma)(\theta)$ allows to identify a descent direction for the functional $J(\Gamma)$, that is, a vector field $\theta$ such that $J^\prime(\Gamma)(\theta) < 0$.
This property guarantees that for a small enough pseudo-time step $t>0$ the perturbed configuration $\Gamma_{t\theta}$ has better performance than $\Gamma$: 
$$J(\Gamma_{t\theta}) \:\: \approx \:\:  J(\Gamma) + t J^\prime(\Gamma)(\theta) \:\: <\:\: J(\Gamma).$$

The derivatives of the shape functionals considered in this article turn out to be of the form: 
\begin{equation}\label{eq.structJp}
J^\prime(\Gamma)(\theta) = \int_\Gamma v_\Gamma \theta\cdot n \:\d s + \int_\Sigma \Big( w_\Sigma \theta \cdot n_\Sigma  + z_\Sigma \theta \cdot n \Big) \:\d \ell,
\end{equation}
for some known scalar functions $v_\Gamma : \Gamma \to \R$ and $w_{\Sigma}, z_\Sigma: \Sigma \to \R$.
This structure reflects that only normal perturbations of the ``bulk'' of $\Gamma$ or perturbations of its endpoints may alter the value $J(\Gamma)$ at first order. 
A descent direction $\theta$ can be extracted from \cref{eq.structJp} by separating the treatments of its tangential and normal components:
\begin{itemize} 
\item A normal descent direction $\theta_n := -v_n n$ is obtained thanks to the so-called Hilbertian method, see \cite{allaire2020survey,azegami1996domain,burger2003framework,de2006velocity}: the scalar component $v_n$ is the solution to the variational problem
\begin{equation}\label{eq.Hilbert}
\text{Search for } v_n \in V \text{ s.t. } \forall w \in V, \quad a(v_n,w) =  \int_\Gamma v_\Gamma w \:\d s + \int_\Sigma  z_\Sigma  w \:\d \ell, 
\end{equation}
where $V$ is the Hilbert space
$$ V = \Big\{ v \in H^1(D) \text{ s.t. } v\lvert_\Gamma \in H^1(\Gamma) \Big\},$$
equipped with the following inner product: 
 $$ a(v,w) = \alpha^2 \int_D \nabla v \cdot \nabla w \:\d x + \int_D vw \:\d x  + \alpha^2 \int_\Gamma \nabla_\Gamma v \cdot \nabla_\Gamma w \:\d s.$$ 
This procedure is justified by the following simple calculation, which directly stems from \cref{eq.structJp,eq.Hilbert}: 
$$J^\prime(\Gamma)(\theta_n) = -\int_\Gamma v_\Gamma v_n  \:\d s - \int_\Sigma  z_\Sigma v_n \:\d \ell  = -a(v_n,v_n) < 0.$$
 \item A descent direction $\theta_\Sigma$ which is tangential to $\Gamma$ is simply given by: 
 $$ \theta_\Sigma = - w_\Sigma n_\Sigma.$$ 
 \end{itemize}
Combining both observations, the desired descent direction for $J(\Gamma)$ is then:
$$ \theta = -v_n n - w_\Sigma n_\Sigma.$$

Note that this separate search for the normal and tangential components of the descent direction $\theta$ agrees with the representation of  $\Gamma$ by two level set functions $\phi$ and $\psi$ described in \cref{sec.2LSM}.
Indeed, the normal vector field $\theta_n$ drives the motion of $\phi$ while leaving $\psi$ unaltered; on the contrary, advection by the tangential field $\theta_\Sigma$ does not modify $\phi$.

\subsection{Optimization of anisotropic perimeter functionals}\label{sec.aniper}

\noindent This section deals with the minimization of the Riemannian length of a path $\Gamma$ in the plane. 
More precisely, we consider the instance of \cref{eq.sopb}, where the objective function is an anisotropic perimeter functional of the form:
\begin{equation}\label{eq.anisoper}
J(\Gamma) = \int_\Gamma \varphi(x,n(x)) \:\d s(x), 
\end{equation}
made from a given smooth function $\varphi :\R^2_x \times \R^2_n \to \R$. 
The calculation of the shape derivative of such a functional is detailed in \cref{app.deranisoper}.

Despite the academic appearance, avatars of this problem show up in multiple applications, such as trajectory optimization \cite{kelly2017introduction}, car path planning \cite{ziegler2014trajectory} or models for protein folding \cite{lavalle2006planning}. 
Drawing inspiration from \cite{precioso2024hybrid}, we consider an instance of the so-called Zermelo's navigation problem \cite{zermelo1931navigationsproblem}: a traveler is strolling through a landscape whose topography is defined by a height function $h:D \to \R$ over the 2d unit square $D$, which is moreover subjected to wind conditions accounted for by the velocity field $w : D \to \R^2$. 
In this situation, we aim to optimize the trajectory of the traveller, which is represented by an open curve $\Gamma$ in the domain $D$, with respect to the length of its image on the graph of $h$, while taking into account the effort of moving against wind. More precisely, the optimization problem at stake is thus of the form \cref{eq.anisoper}, with $\varphi$ given by:
$$ \varphi(x,n) = - n \cdot w^\perp(x) + \sqrt{\calM(x)n \cdot n},$$
where $w^\perp = (-u_2,u_1)$ is the $90^{\degree}$ counterclockwise rotate of $w$, and 
$$ \calM(x) = \left(\begin{array}{cc}
1+ \left( \frac{\partial h}{\partial x_2}(x)\right)^2 & - \frac{\partial h}{\partial x_1} (x) \frac{\partial h}{\partial x_2 }(x) \\ 
  -\frac{\partial h}{\partial x_1}(x) \frac{\partial h}{\partial x_2 }(x) & 1+ \left( \frac{\partial h}{\partial x_1} (x)\right)^2 
\end{array} \right).$$
The first term in the expression of $\varphi(x,n)$ is minimal when the tangent vector to $\Gamma$ is aligned with the direction of wind (i.e. $n\cdot w^\perp$ is maximal) while the second term is the first fundamental form of the landscape expressed in terms of the normal vector to $\Gamma$. 

We consider two different physical situations. In a first example, depicted on \cref{fig.path4pks}, the landscape is made of four peaks and no wind is blowing ($w=0$). Using a straight line as initial guess, we optimize the path $\Gamma$ between two fixed endpoints by solving \cref{eq.sopb,eq.anisoper}.
Starting from a straight line, we apply the methodology of \cref{sec.algo} to this problem, and the outcome is depicted on \cref{fig.path4pks}. The convergence history, reported in \cref{fig.pathgraph} (a), shows that the objective function is smoothly minimized until a local minimum of the problem is attained.  

\begin{figure}[!ht]
    \centering
    \begin{tabular}{c}
\begin{minipage}{0.8\textwidth}
\centering
\begin{overpic}[width=1.0\textwidth]{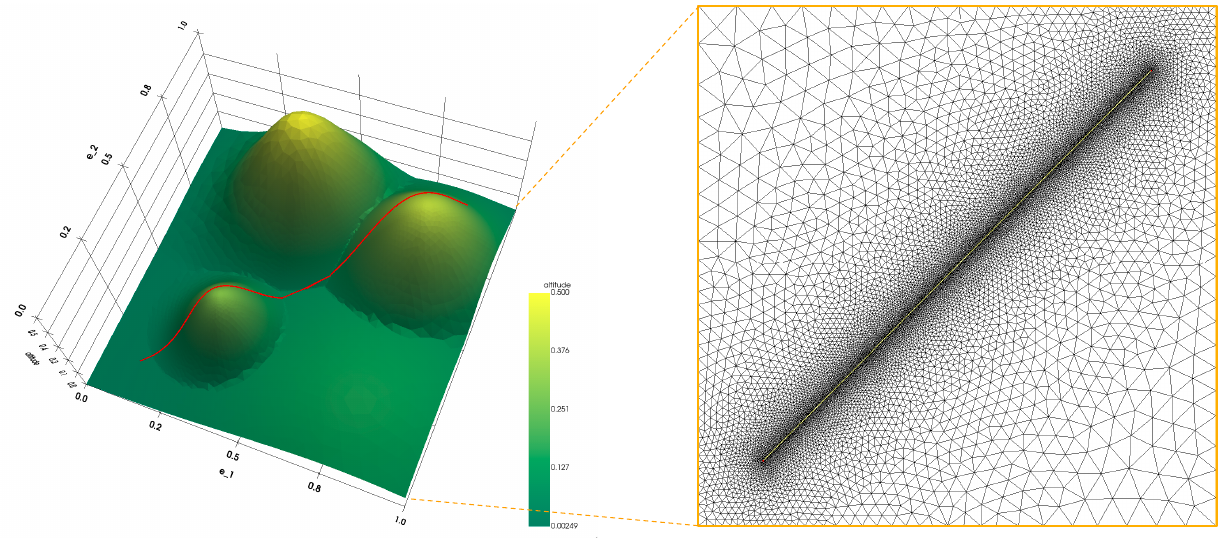}
\put(0,-3){\fcolorbox{black}{white}{a}}
\end{overpic}
\end{minipage} 
\end{tabular}\par\bigskip
    \begin{tabular}{c}
\begin{minipage}{0.8\textwidth}
\centering
\begin{overpic}[width=1.0\textwidth]{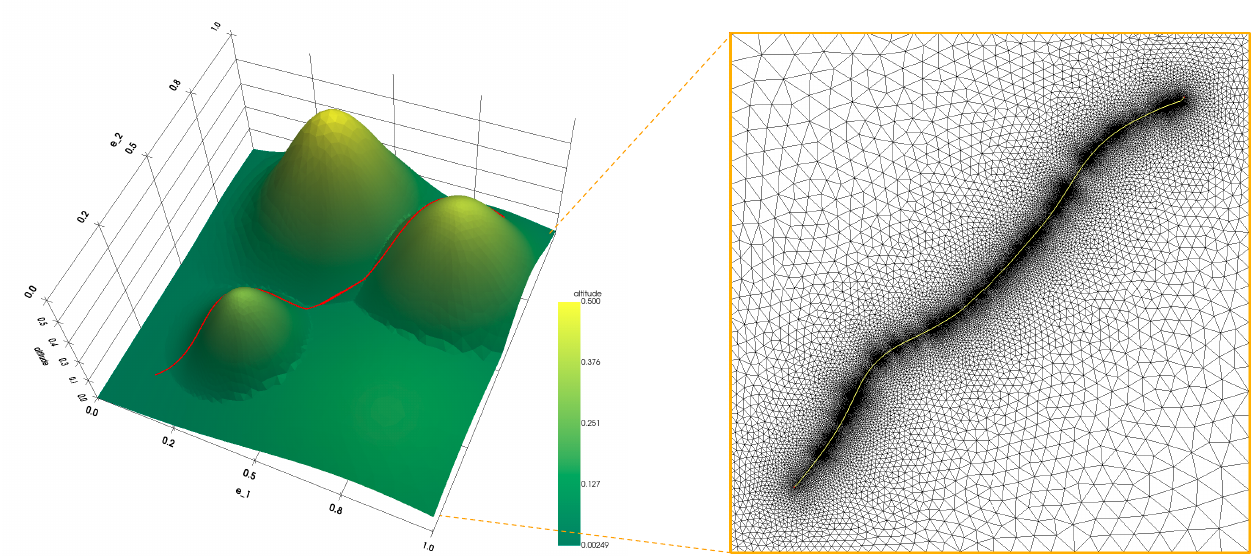}
\put(0,-3){\fcolorbox{black}{white}{b}}
\end{overpic}
\end{minipage}  
\end{tabular}  
\par\bigskip
    \begin{tabular}{c}
\begin{minipage}{0.8\textwidth}
\centering
\begin{overpic}[width=1.0\textwidth]{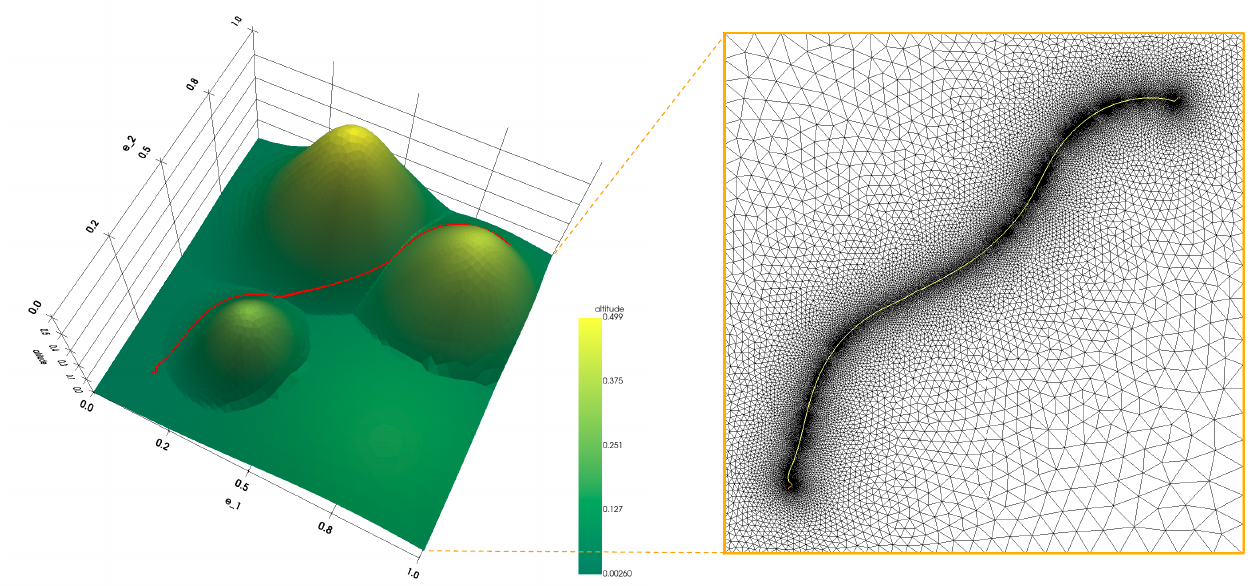}
\put(0,-3){\fcolorbox{black}{white}{c}}
\end{overpic}
\end{minipage}  
\end{tabular}  
 \caption{\it Optimization of the path between two points on a landscape featuring 4 peaks, without wind, as considered in \cref{sec.aniper}; (a) Initial guess for the path; (b) Intermediate shape ($n=6$); (c) Optimized path ($n=30$).}
  \label{fig.path4pks}
\end{figure}

In a second experiment, we consider the optimization of the path $\Gamma$ through a landscape made of three peaks, in the presence of a wind corridor blowing from North-East to South-West:
$$ w(x) = w_{\text{max}} e^{-\frac{d(x)^2}{\sigma_w^2}} \tau_w , \text{ where } d(x) = \Big(x-(0.1,0)\Big)\cdot n_w \text{ is the distance to the corridor},$$ $w_{\text{max}} = 2$ is the maximum intensity of wind, $\sigma_w = 0.15$, $\tau_w = (-0.9923,-0.1240)$ is the direction of wind and $n_w = \tau_w^\perp$ is the orthogonal direction.
 The results are depicted on \cref{fig.fig.path3pks}, see \cref{fig.pathgraph} (b) for the convergence history of the computation. Here, the path takes a detour, occasionally going uphill, in order to avoid moving against the wind.

\begin{figure}[!ht]
    \centering
    \begin{tabular}{c}
\begin{minipage}{0.9\textwidth}
\centering
\begin{overpic}[width=1.0\textwidth]{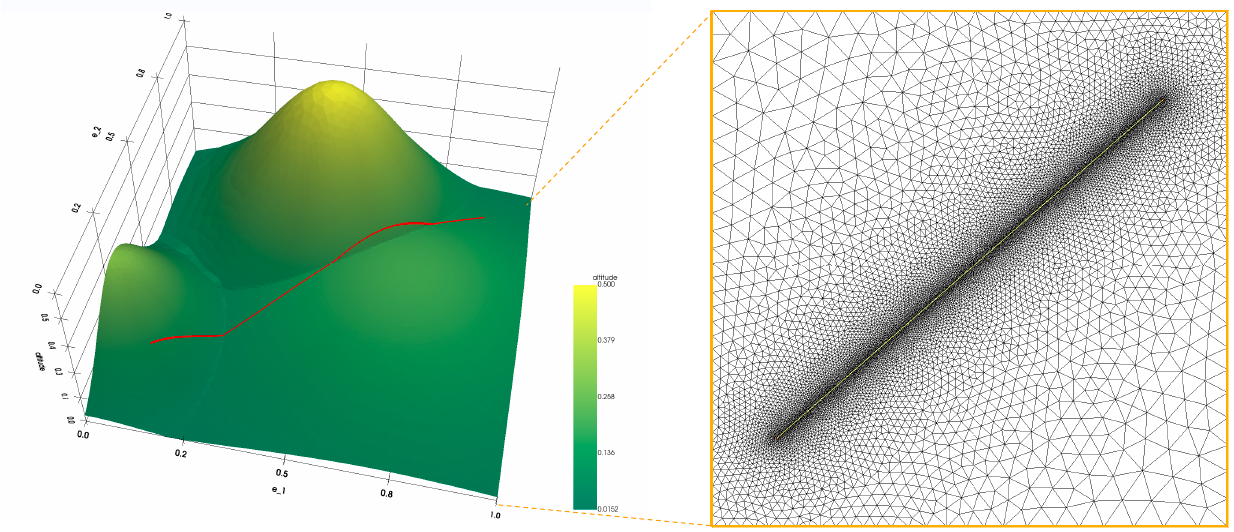}
\put(0,-3){\fcolorbox{black}{white}{a}}
\end{overpic}
\end{minipage} 
\end{tabular}\par\bigskip
    \begin{tabular}{c}
\begin{minipage}{0.9\textwidth}
\centering
\begin{overpic}[width=1.0\textwidth]{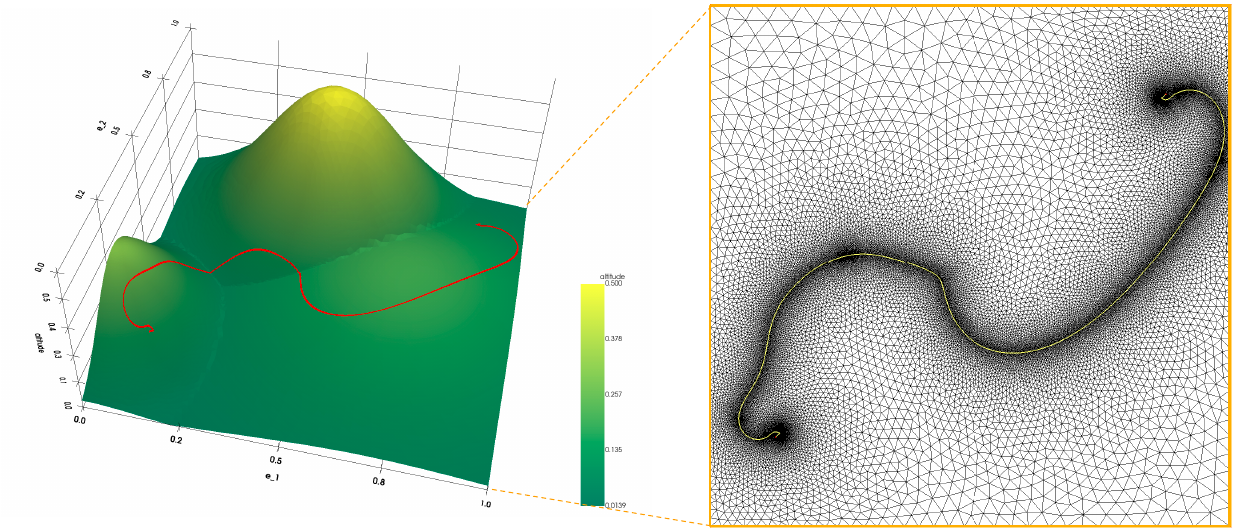}
\put(0,-3){\fcolorbox{black}{white}{b}}
\end{overpic}
\end{minipage}  
\end{tabular}  
 \caption{\it Optimization of the path between two points on a landscape featuring 3 peaks and wind, as considered in \cref{sec.aniper}; (a) Initial guess for the path; (b) Optimized path ($n=30$).}
  \label{fig.fig.path3pks}
\end{figure}

\begin{figure}[!ht]
    \centering
 \begin{tabular}{cc}
\begin{minipage}{0.45\textwidth}
\centering
\begin{overpic}[width=1.0\textwidth]{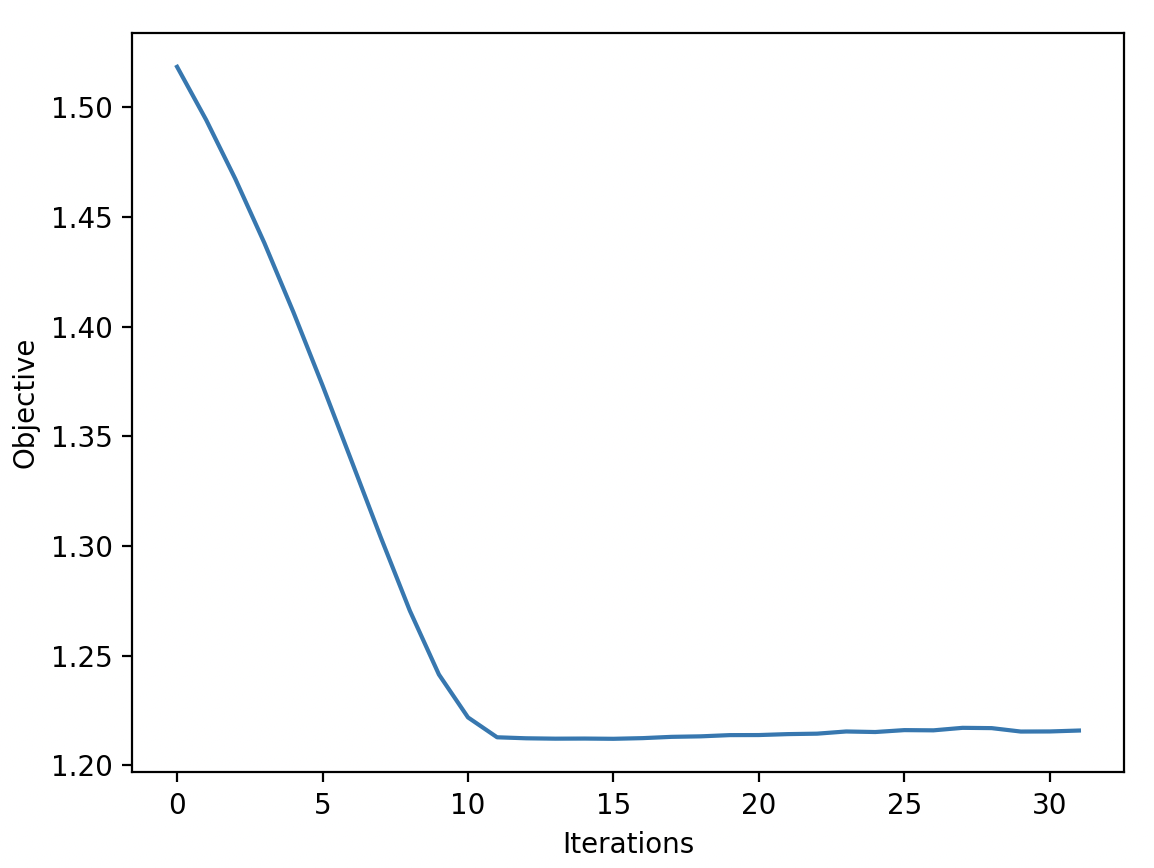}
\put(0,-3){\fcolorbox{black}{white}{$a$}}
\end{overpic}
\end{minipage} & 
\begin{minipage}{0.45\textwidth}
\centering
\begin{overpic}[width=1.0\textwidth]{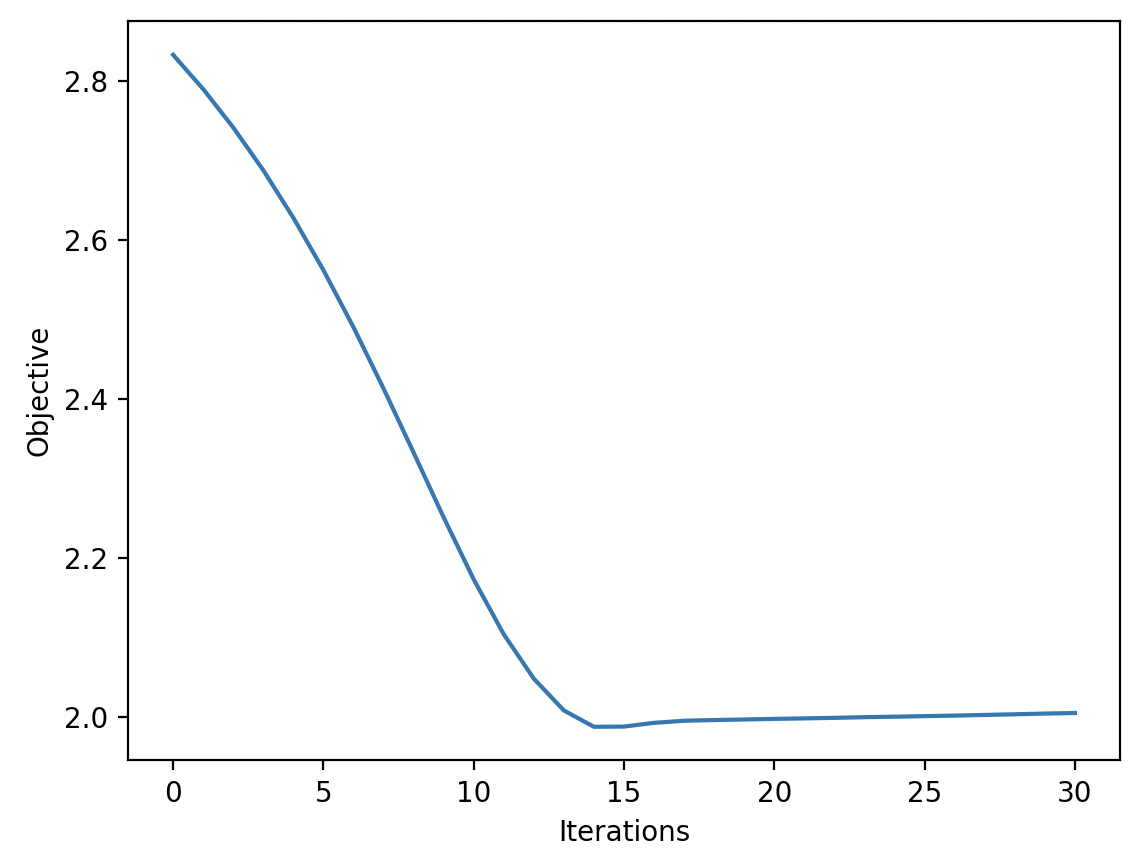}
\put(0,-3){\fcolorbox{black}{white}{$b$}}
\end{overpic}
\end{minipage} 
\end{tabular}  
 \caption{\it Convergence histories in the path planning examples of \cref{sec.aniper} when (a) The landscape has 4 peaks and wind is omitted; (b) The landscape has 3 peaks and wind is blowing.}
  \label{fig.pathgraph}
\end{figure}

%
%
%
%

\subsection{Optimization of the laser path for additive manufacturing}\label{sec.am}

\noindent This section arises in the context of the fabrication of a 3d shape with the Electron Beam Melting 3d printing technology \cite{gibson2021additive,korner2016additive}. 
This process starts by slicing the 3d shape to be produced into a series of horizontal layers, that will be assembled one on top of the other. The construction then takes place in a build chamber which is filled by metallic powder; each layer is assembled by selectively melting the powder in the region occupied by the considered pattern, see \cref{fig.amset}.
Drawing inspiration from the work \cite{boissier2020additive}, we aim to optimize the trajectory of the laser during the assembly of a given 2d layer so that powder is melted exactly in the desired region. 

 \begin{figure}[!ht]
\centering
\includegraphics[width=0.85\textwidth]{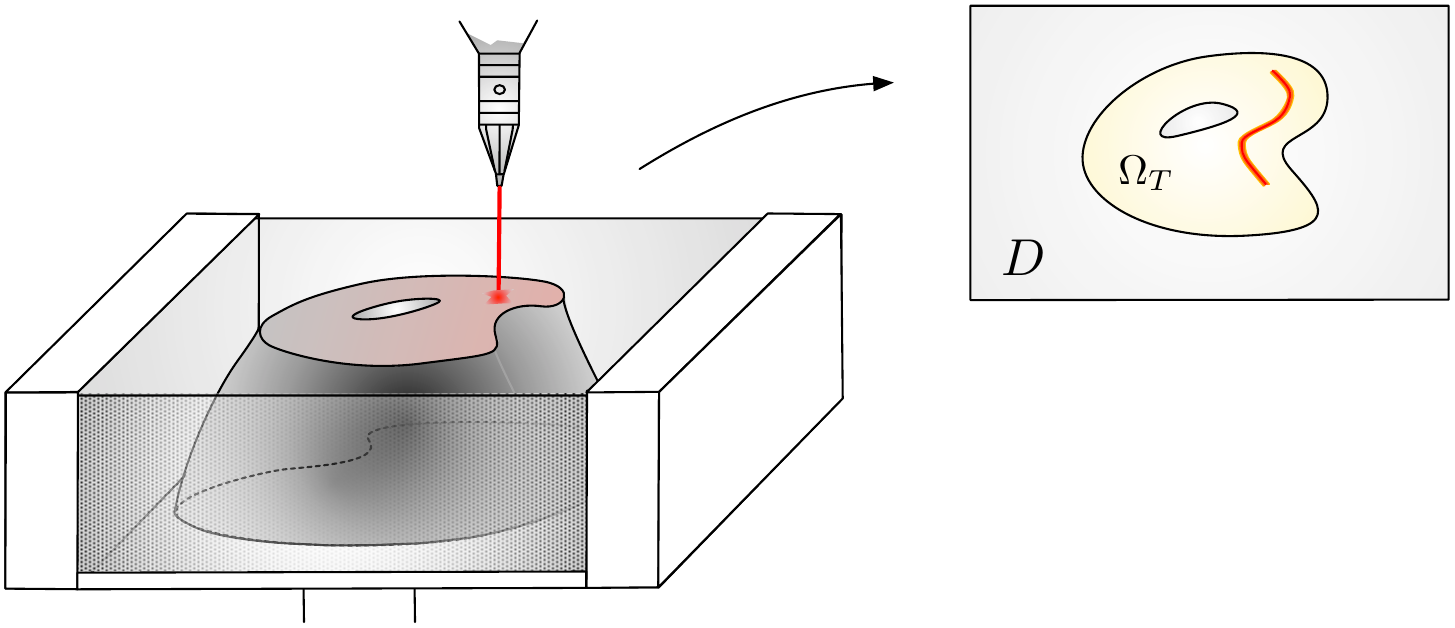}
\caption{\it Construction of a shape by Electron Beam Melting, as considered in \cref{sec.am}: the laser is required to melt the metallic powder specifically in the currently processed 2d slice of the shape.}
\label{fig.amset} 
\end{figure}

We consider the steady-state model proposed in \cite{boissier2020additive}, whose physical justification is provided in the appendix of that article. 
The hold-all domain $D$ represents the currently processed two-dimensional slice of the build chamber; the laser acts as an instantaneous source of heat along the laser path, which is represented by an open curve $\Gamma \subset D$. The temperature $u_\Gamma$ in these circumstances is the $H^1(D)$ solution to the conductivity equation: 
\begin{equation}\label{eq.heateq}
\left\{
\begin{array}{cl}
-\dv(\gamma \nabla u_\Gamma) + \beta (u_\Gamma - u_0) = q \delta_\Gamma& \text{in } D, \\
\gamma \frac{\partial u_\Gamma}{\partial n} = 0 & \text{on } \partial D,
\end{array}
\right.
\end{equation}
where $\delta_\Gamma$ stands for integration on $\Gamma$. Precisely, $u_\Gamma$ is the solution to the following variational problem: 
$$ \text{Search for } u_\Gamma \in H^1(D)\: \text{ s.t. } \forall v \in H^1(D), \quad \int_D \gamma \nabla u_\Gamma \cdot \nabla v \:\d x + \int_D \beta u_\Gamma v \:\d x = \int_{D} \beta u_0 v \:\d x +\int_\Gamma q v \:\d s.$$
In this model, $\gamma$ is the thermal conductivity of the powder, the parameter $\beta$ accounts for the heat transmitted from the current 2d layer to the outer medium whose temperature equals $u_0$, and $q$ is the power of the laser per unit of length.
For simplicity, $\gamma$, $\beta$, $u_0$ and $q$ are assumed to be constant. 

Denoting by $\upc$ the phase change temperature of the powder and by $\Omega_T$ the target shape to be melted within the slice $D$, we consider the minimization of the following objective function:
\begin{equation}\label{eq.Jam}
J(\Gamma) = \int_{\Omega_T}{\left[u_\Gamma - \upc \right]_-^2 \:\d x} +  \int_{D\setminus \overline{\Omega_T}}{\left[u_\Gamma - \upc \right]_+^2 \:\d x} ,
\end{equation}
where we have set $[t]_+ = \max(0,t)$ and $[t]_- = \min(0,t)$. 
Intuitively, we wish to optimize the laser path $\Gamma$ so that the induced temperature should be larger than $\upc$ in $\Omega_T$, thus making the first integral in \cref{eq.Jam} small, and lower than $\upc$ in $D \setminus \overline{\Omega_T}$, thus making the second integral small.

The shape derivative $J^\prime(\Gamma)(\theta)$ of this functional is calculated in \cite{boissier2020additive} thanks to the formal C\'ea's method \cite{cea1986conception}; for completeness, the computation is presented in \cref{app.sdam} by means of the rigorous technique exposed in e.g. \cite{henrot2018shape,murat1976controle}.


As such, the minimization of $J(\Gamma)$ tends to produce curves $\Gamma$ presenting self-intersections -- a pattern that is often undesirable in practice. To alleviate the onset of such features, we follow \cite{allaire2016thickness} and bring into play an additional functional:
$$C(\Gamma) = \int_\Gamma \Big(\left[d_\Omega(x + \dmin n(x))\right]_-^2  + \left[d_\Omega(x - \dmin n(x))\right]_+^2 \Big)\:\d s,$$
where $\Omega$ denotes any smooth bounded domain whose boundary encloses $\Gamma$ and $d_\Omega$ stands for the signed distance function to $\Omega$, see \cref{eq.sdf}. 
Loosely speaking, the minimization of $C(\Gamma)$ imposes that for each parameter value $t \in (-\frac{\dmin}{2},\frac{\dmin}{2})$, the offset curve $\Gamma_{tn} = (\Id + tn)(\Gamma)$ should be diffeomorphic to $\Gamma$: $C(\Gamma)$ takes large values if at some point $x \in \Gamma$, one of the normal rays $(x,x + \dmin n(x))$ or $(x,x - \dmin n(x))$ crosses the boundary of $\Omega$.
The calculation of the shape derivative of $C(\Gamma)$ and its numerical integration into a shape optimization algorithm are achieved exactly along the lines of our previous work \cite{feppon2018var}, which is devoted to a variational method for enforcing geometric constraints, expressed in terms of the signed distance function to the optimized shape.

All things considered, the shape optimization problem under scrutiny reads: 
\begin{equation}\label{eq.sopbam}
\min\limits_{\Gamma} \:\: \Big( J(\Gamma) + \ell C(\Gamma) \Big), 
\end{equation}
where $\ell$ is a fixed weight.

We apply this shape optimization framework to the setting exemplified in \cref{fig.amset1} (a): the design domain $D$ is the unit square $(0,1)^2$ and the shape $\Omega_T$ of the 2d slice to be assembled is a slightly smaller square. Starting from the simple initial curve $\Gamma^0$ depicted on \cref{fig.amres} (a), we apply our numerical \cref{algo.openls} to the resolution of the optimization problem \cref{eq.sopbam}. The parameters of the computation are $\gamma =0.01$, $\beta = 0.5$, $q = 100$, $u_0 = 0$, $\upc = 350$ and $\ell =1\text{e}9$. The maximum number of vertices in a mesh is $48,245$ (for about twice as many elements) and the total computational time is about 20 mn. A few snapshots of the optimization process are reported on \cref{fig.amres} and the convergence history in \cref{fig.amset1} (b) shows the smooth decrease of the objective function until a local minimum is attained.

\begin{figure}[!ht]
    \centering
 \begin{tabular}{cc}
\begin{minipage}{0.45\textwidth}
\centering
\begin{overpic}[width=0.75\textwidth]{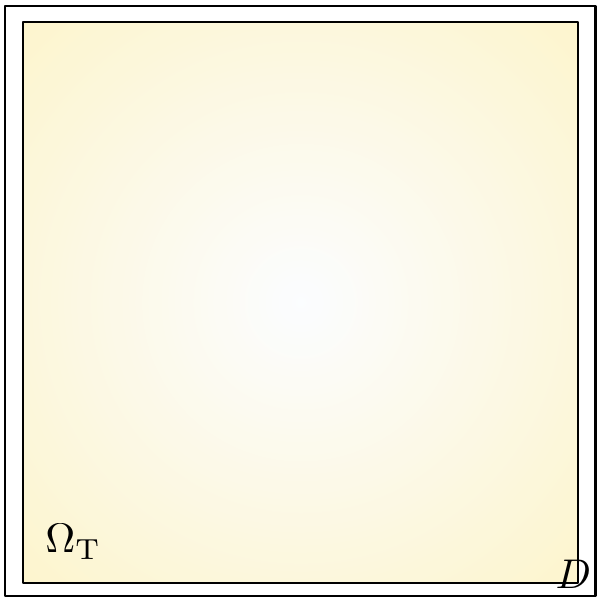}
\put(0,-3){\fcolorbox{black}{white}{$a$}}
\end{overpic}
\end{minipage} & 
\begin{minipage}{0.45\textwidth}
\centering
\begin{overpic}[width=1.0\textwidth]{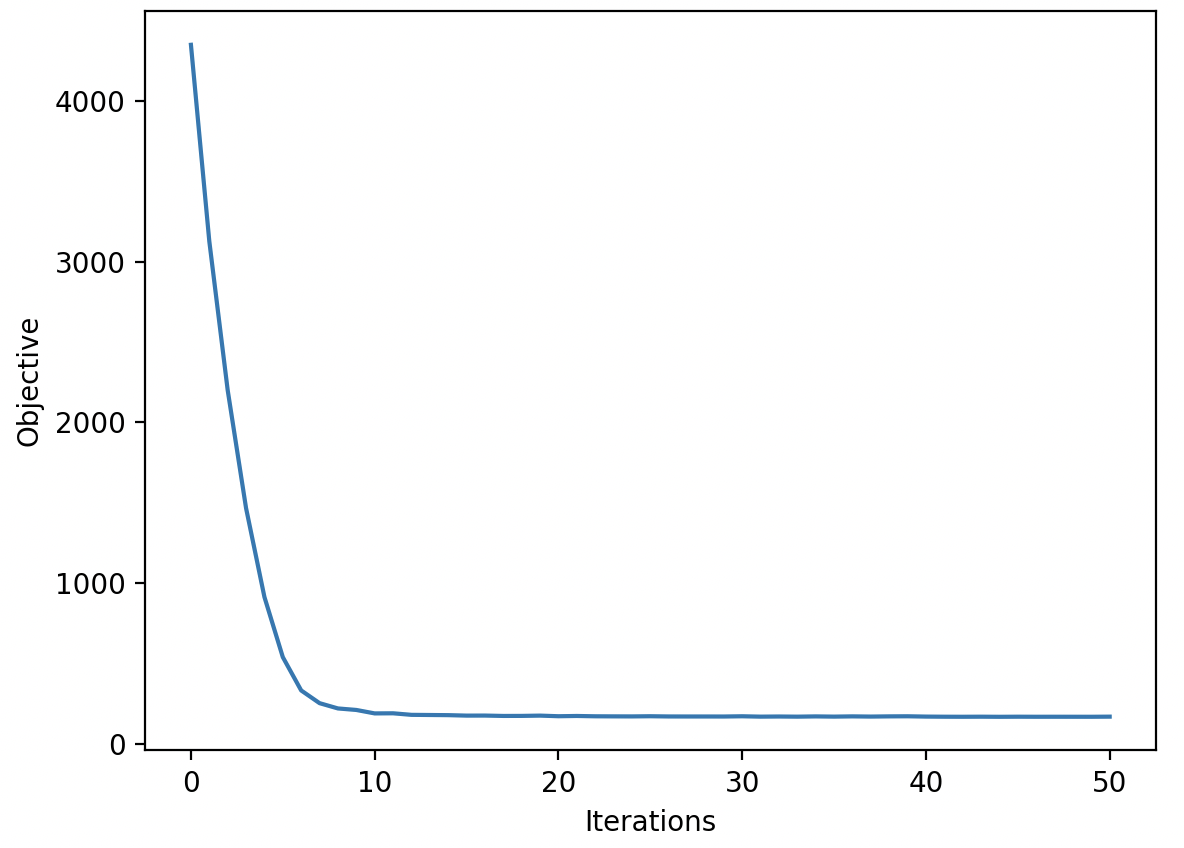}
\put(0,-3){\fcolorbox{black}{white}{$b$}}
\end{overpic}
\end{minipage} 
\end{tabular}  
 \caption{\it (a) Target shape in the first numerical optimization example of the laser path used in the Electron Beam Melting method considered in \cref{sec.am}; (b) Convergence history.}
  \label{fig.amset1}
\end{figure}

\begin{figure}[!ht]
    \centering
    \begin{tabular}{cc}
\begin{minipage}{0.48\textwidth}
\centering
\begin{overpic}[width=1.0\textwidth]{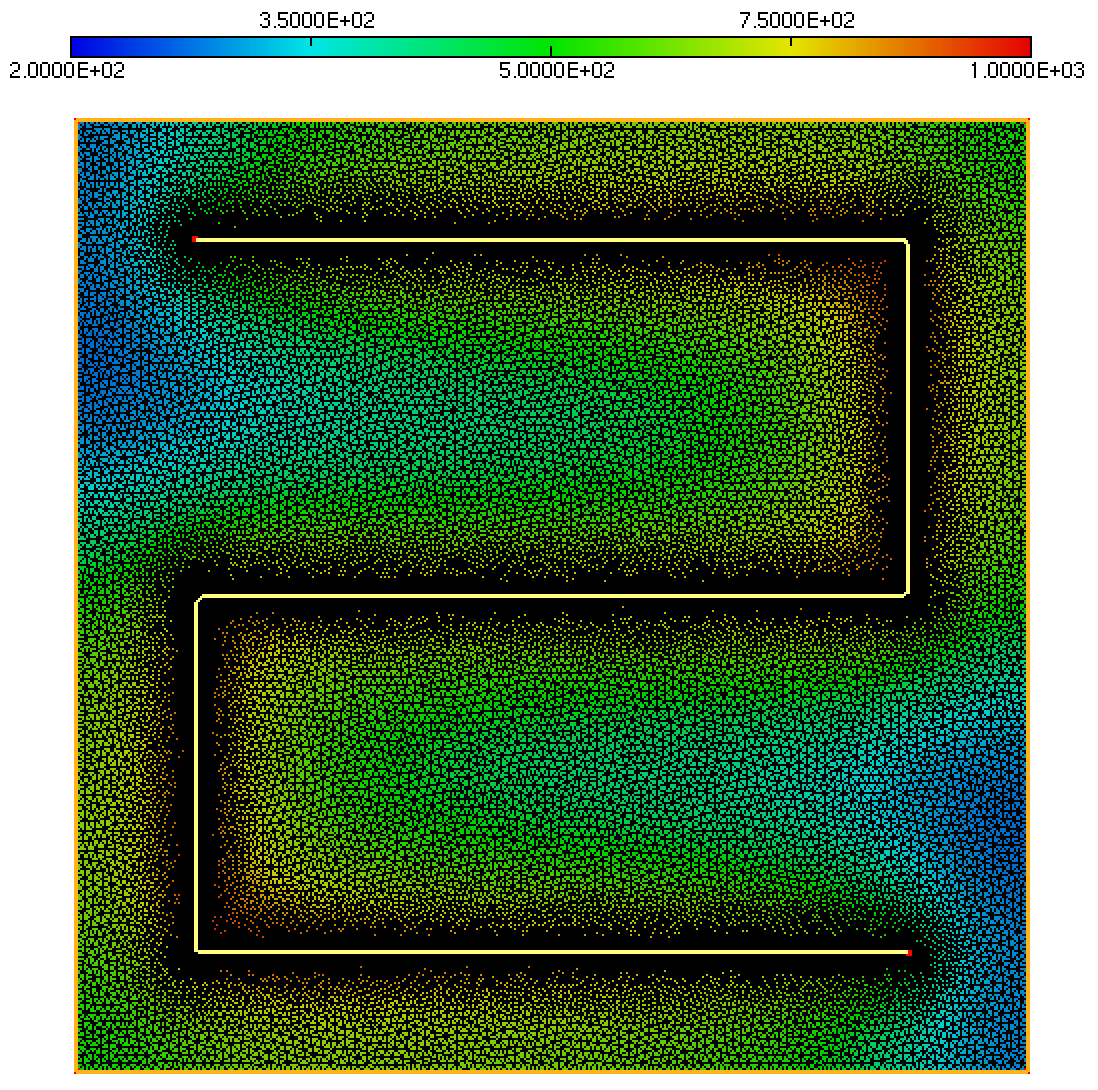}
\put(0,-3){\fcolorbox{black}{white}{$n=0$}}
\end{overpic}
\end{minipage} & 
\begin{minipage}{0.48\textwidth}
\centering
\begin{overpic}[width=1.0\textwidth]{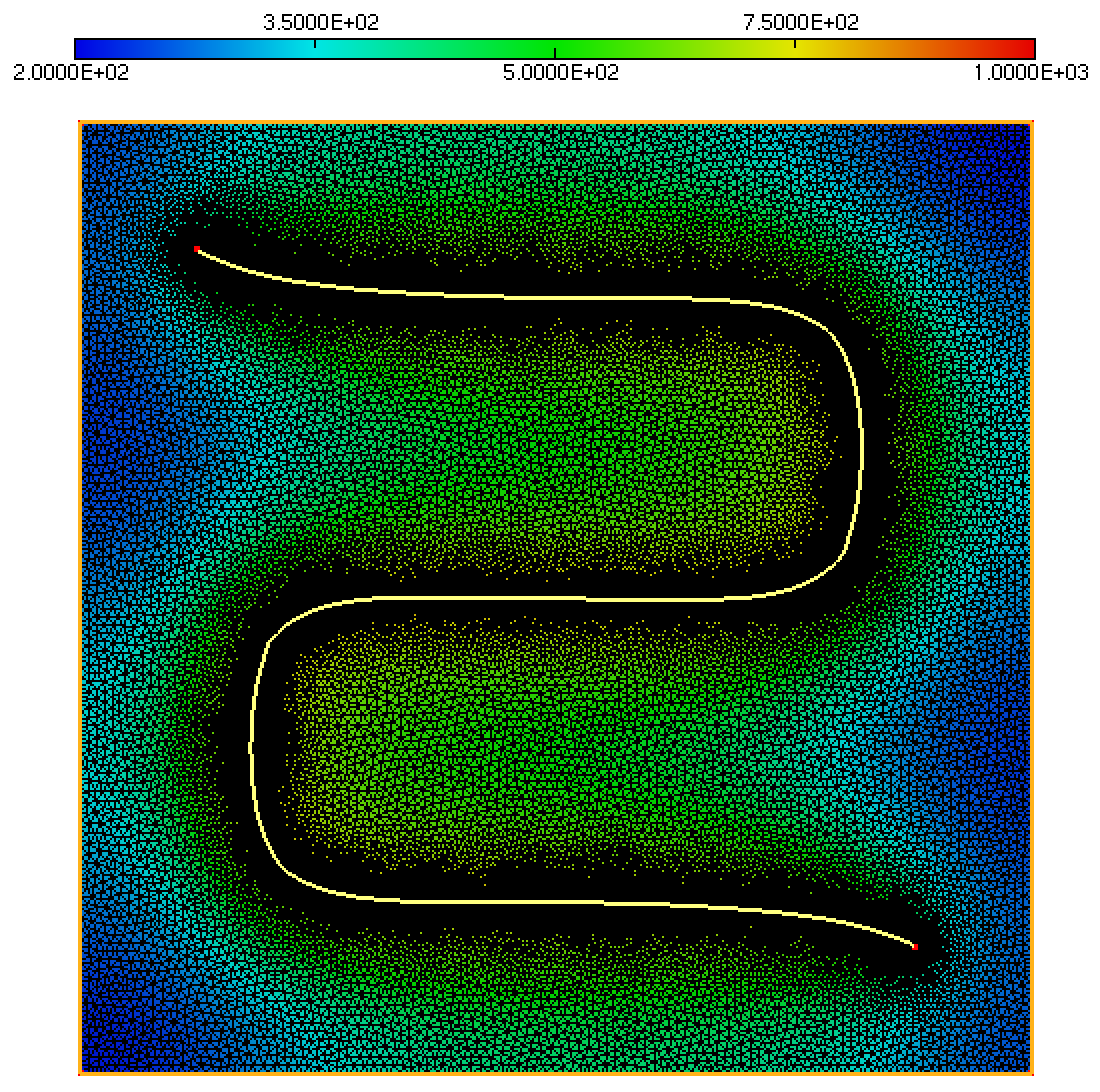}
\put(0,-3){\fcolorbox{black}{white}{$n=5$}}
\end{overpic}
\end{minipage}  
\end{tabular}  \par\bigskip 
    \begin{tabular}{cc}
\begin{minipage}{0.48\textwidth}
\centering
\begin{overpic}[width=1.0\textwidth]{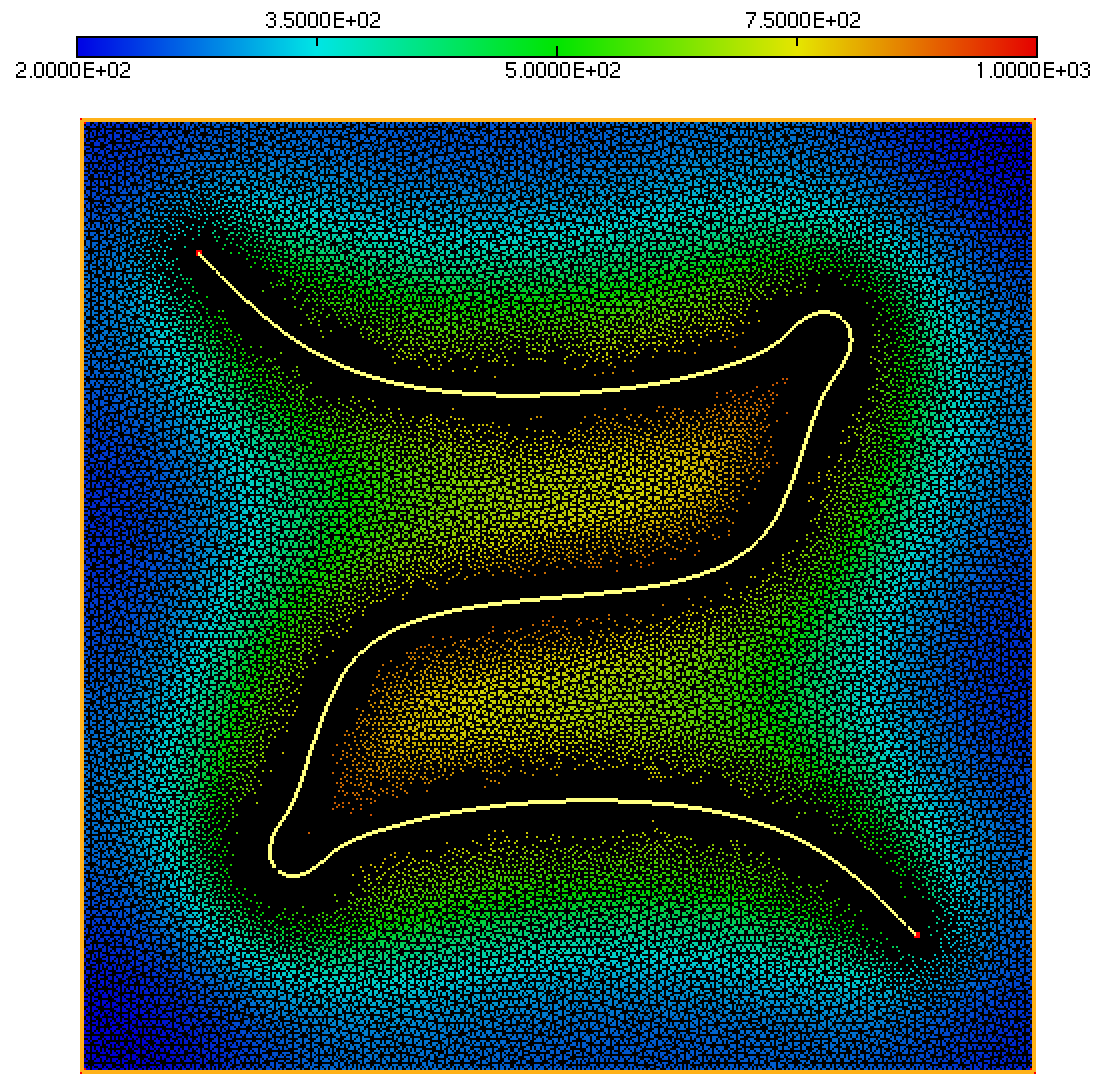}
\put(0,-3){\fcolorbox{black}{white}{$n=20$}}
\end{overpic}
\end{minipage} & 
\begin{minipage}{0.48\textwidth}
\centering
\begin{overpic}[width=1.0\textwidth]{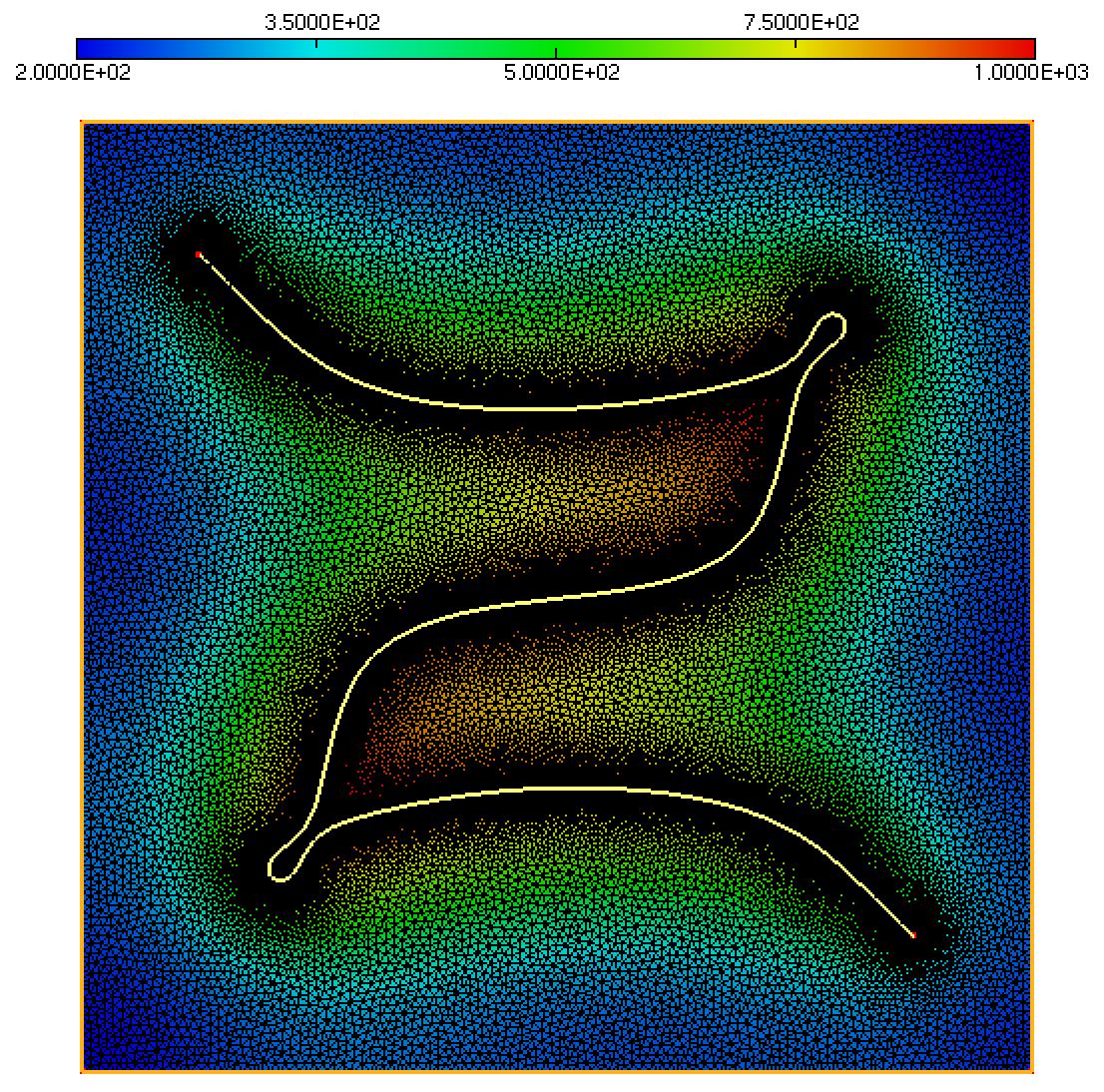}
\put(0,-3){\fcolorbox{black}{white}{$n=50$}}
\end{overpic}
\end{minipage}  
\end{tabular}  
 \caption{\it Intermediate shapes of the path of the laser in view of the assembly of the 2d slice of \cref{fig.amset1} (a) in \cref{sec.am}; in each situation, the colors scale refers to the values of the temperature $u_\Gamma$.}
  \label{fig.amres}
\end{figure}

We next turn to a second example, associated to the more complex layer pattern depicted on \cref{fig.amset2} (a), using the same physical parameters as for the previous example; a few intermediate configurations of the optimized curve are presented on \cref{fig.amcres} and the convergence history of the computation is presented on \cref{fig.amset2} (b).  \par\bigskip

\begin{figure}[!ht]
    \centering
 \begin{tabular}{cc}
\begin{minipage}{0.5\textwidth}
\centering
\begin{overpic}[width=1.0\textwidth]{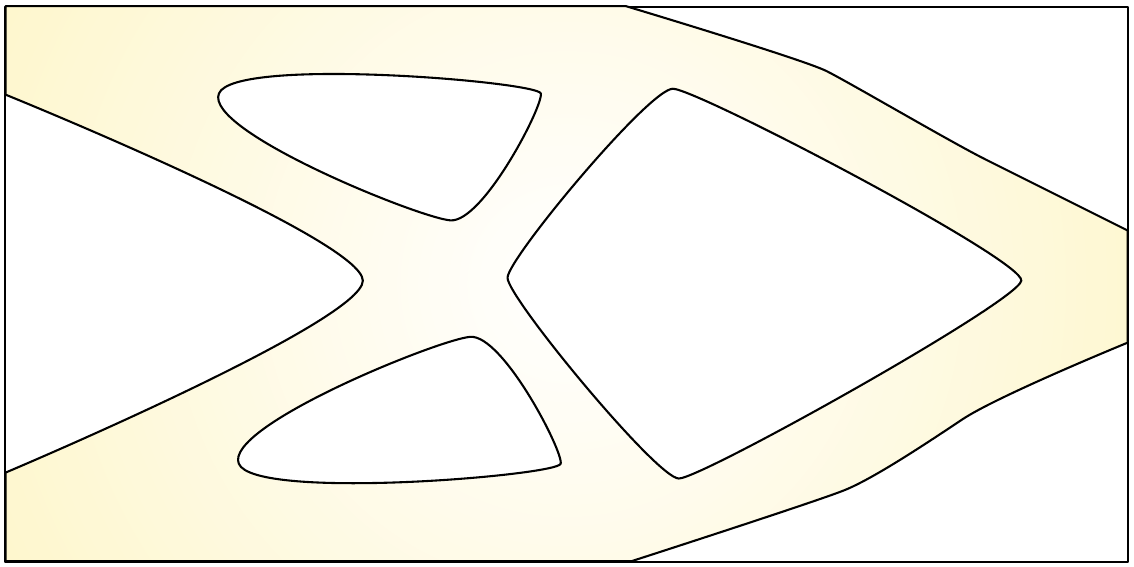}
\put(0,-3){\fcolorbox{black}{white}{$a$}}
\end{overpic}
\end{minipage} & 
\begin{minipage}{0.4\textwidth}
\centering
\begin{overpic}[width=0.9\textwidth]{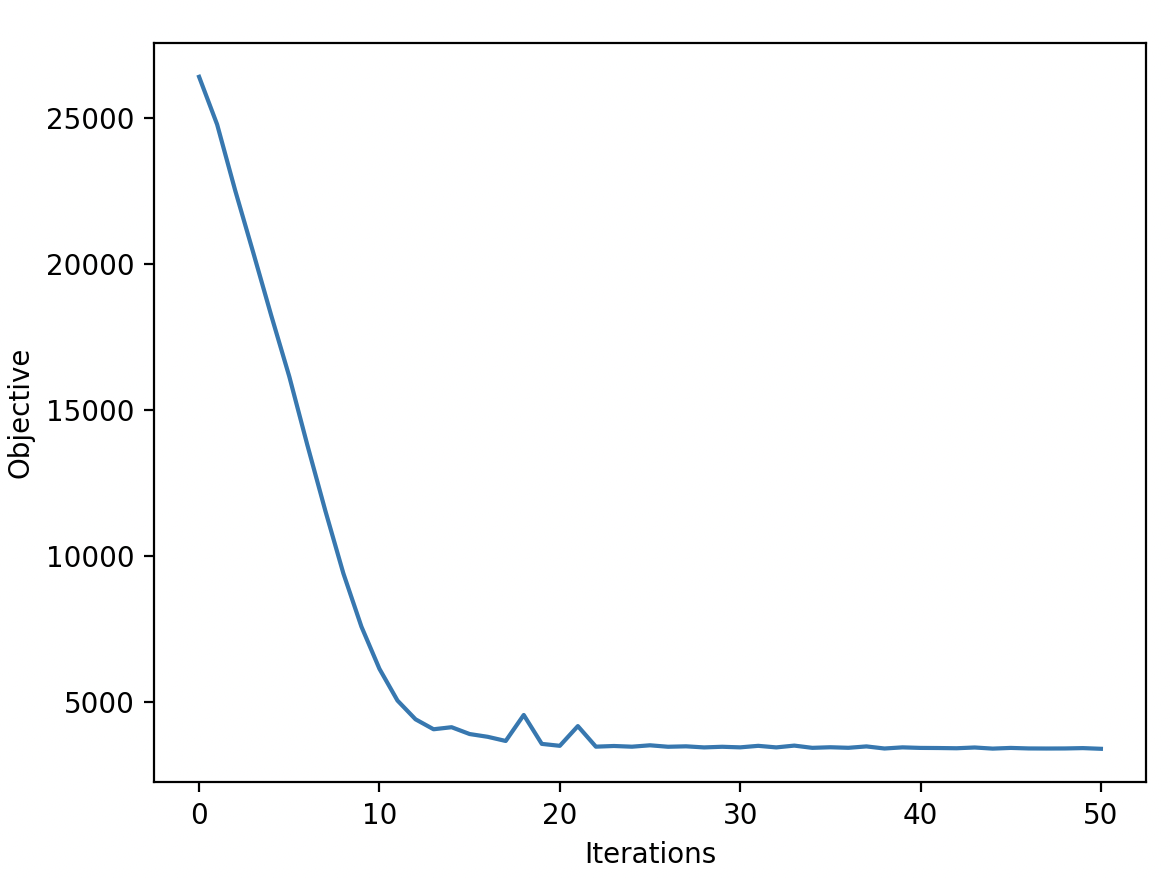}
\put(0,-3){\fcolorbox{black}{white}{$b$}}
\end{overpic}
\end{minipage} 
\end{tabular}  
 \caption{\it (a) Target shape in the second numerical optimization example of the laser path used in the Electron Beam Melting method considered in \cref{sec.am}; (b) Convergence history.}
  \label{fig.amset2}
\end{figure}

\begin{figure}[!ht]
    \centering
    \begin{tabular}{cc}
\begin{minipage}{0.48\textwidth}
\centering
\begin{overpic}[width=1.0\textwidth]{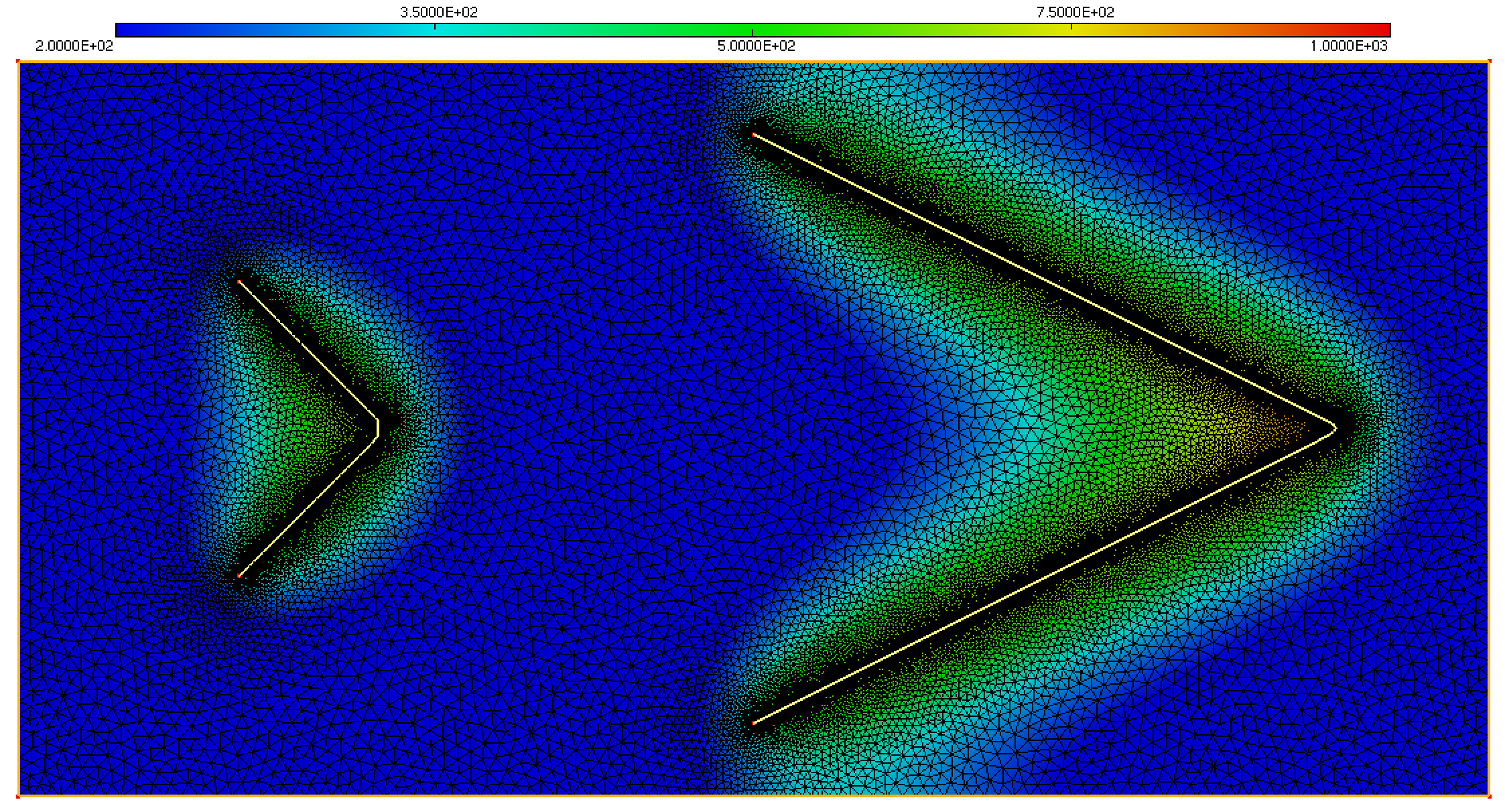}
\put(0,-3){\fcolorbox{black}{white}{$n=0$}}
\end{overpic}
\end{minipage} & 
\begin{minipage}{0.48\textwidth}
\centering
\begin{overpic}[width=1.0\textwidth]{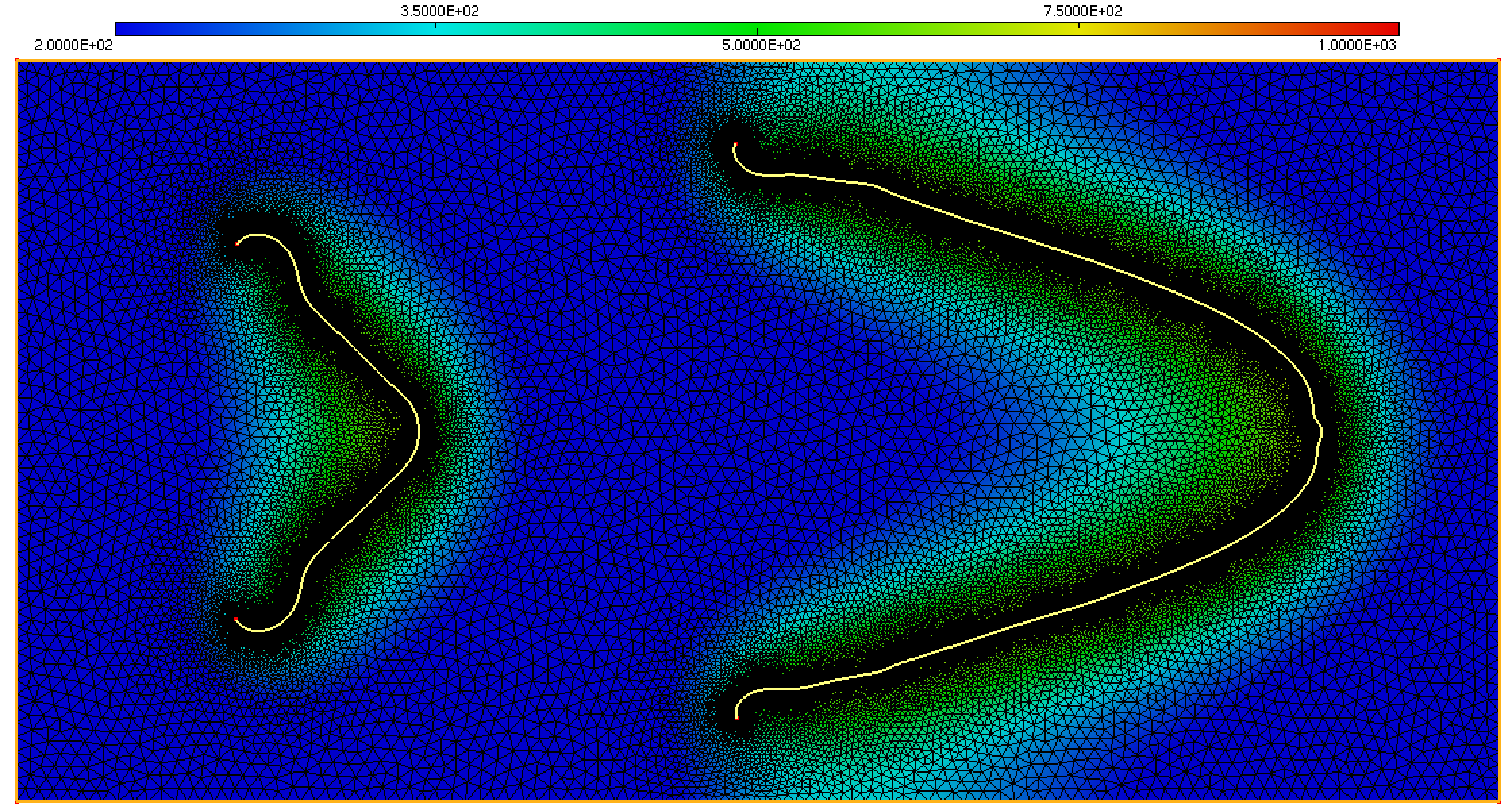}
\put(0,-3){\fcolorbox{black}{white}{$n=5$}}
\end{overpic}
\end{minipage}  
\end{tabular}  \par\bigskip 
    \begin{tabular}{cc}
\begin{minipage}{0.48\textwidth}
\centering
\begin{overpic}[width=1.0\textwidth]{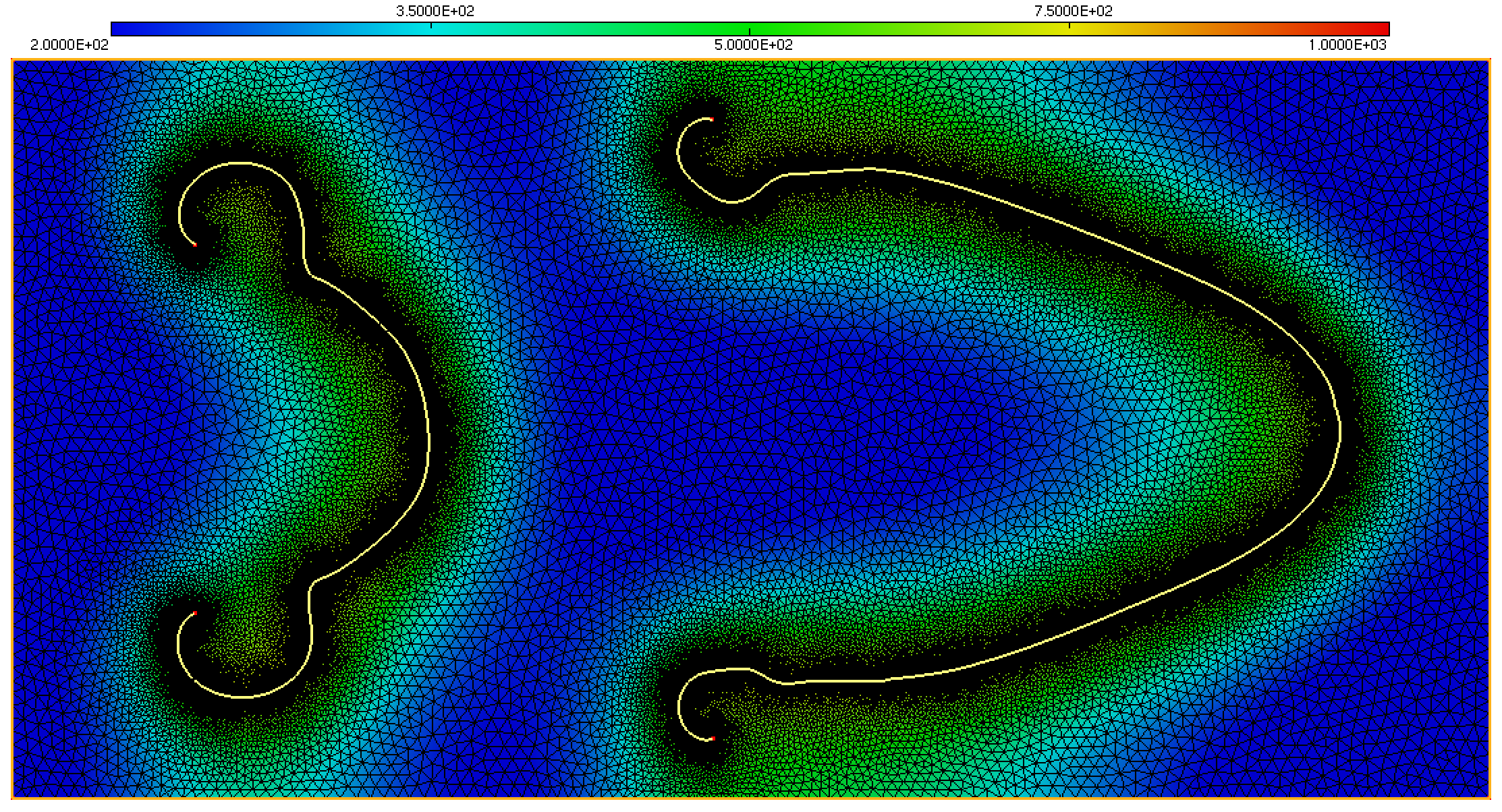}
\put(0,-3){\fcolorbox{black}{white}{$n=10$}}
\end{overpic}
\end{minipage} & 
\begin{minipage}{0.48\textwidth}
\centering
\begin{overpic}[width=1.0\textwidth]{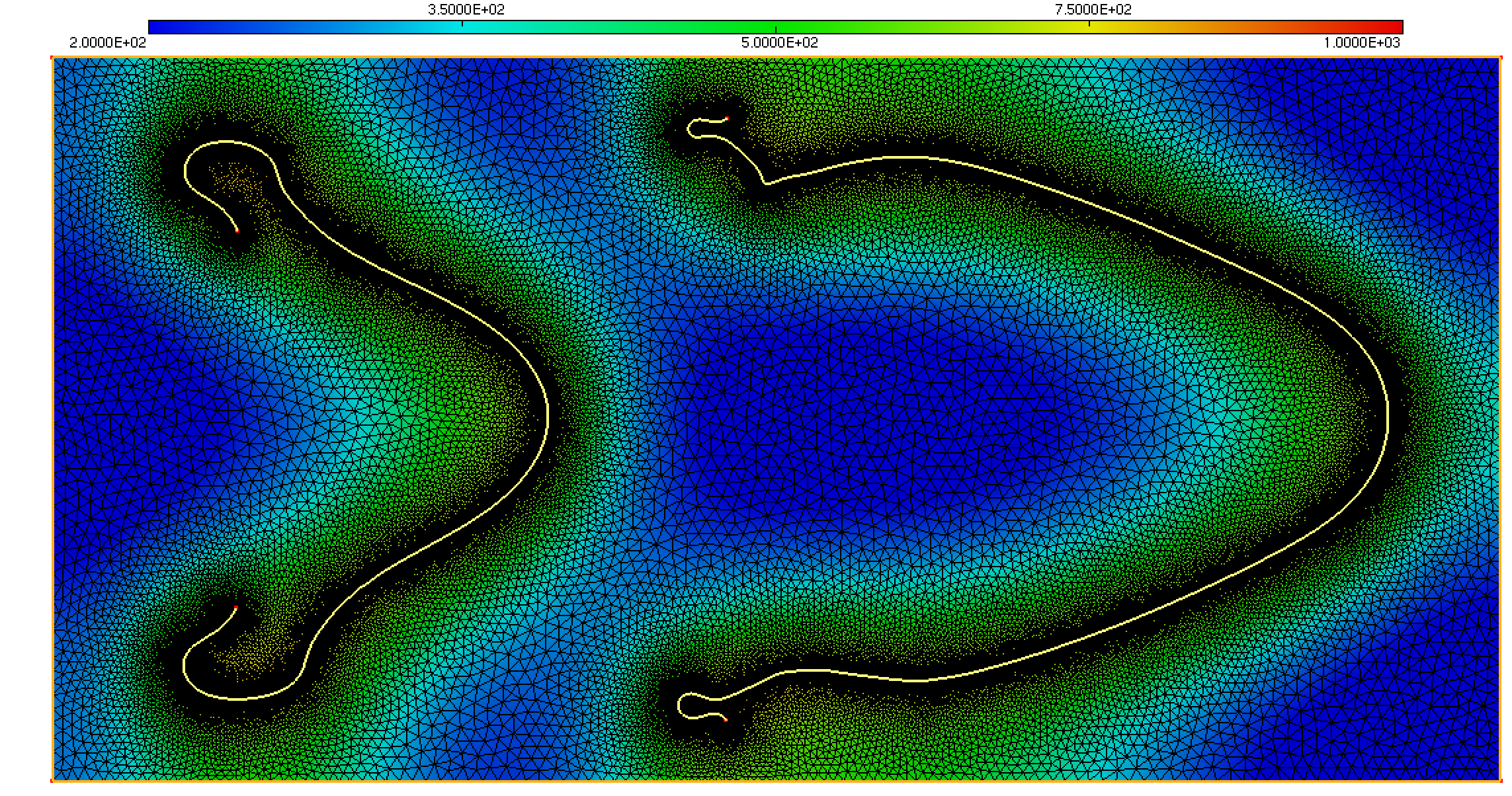}
\put(0,-3){\fcolorbox{black}{white}{$n=50$}}
\end{overpic}
\end{minipage}  
\end{tabular}  
 \caption{\it Intermediate shapes of the path of the laser in view of the assembly of the 2d slice of \cref{fig.amset2} (a) in \cref{sec.am}.}
  \label{fig.amcres}
\end{figure}

\begin{remark}
\noindent \begin{itemize}
\item In principle, one may wish to add a constraint about the length of the path $\Gamma$ to the optimization problem \cref{eq.sopbam}; although it is perfectly doable, we did not feel the need to do so to observe interesting optimized paths in our examples. 
\item The optimization of the path $\Gamma$ via smooth perturbations of its shape conducted in this section could be complemented with the use of ``topological derivatives'', allowing to add small connected components to $\Gamma$ in an optimal fashion, see Chap.8 of \cite{boissier2020coupling} for a discussion about this subject.
\end{itemize}
\end{remark}

\subsection{Least-square reconstruction of a fault line in the underground}\label{sec.fault}

\noindent This section deals with an application of our framework in geophysics. 
The considered open curve $\Gamma$ accounts for a fault line in the underground Earth, where the displacement is discontinuous. 
We aim to reconstruct the shape of $\Gamma$ by minimizing a least-square function of the difference between the displacement predicted by a ``forward'' mathematical model and that observed in the course of a geological event. 

We first present the exact mathematical model predicting the displacement of the underground from the knowledge of the fracture $\Gamma$ in \cref{sec.modfault}, before introducing an approximate version in \cref{sec.faultdirect} which simplifies the numerical resolution. The inverse reconstruction problem of $\Gamma$ from observations of the displacement is eventually addressed in \cref{sec.IPfracture}. 

\subsubsection{Description of the physical model}\label{sec.modfault}

\noindent Let $D\subset \R^2$ be a ``hold-all'' domain, accounting for a two-dimensional section of the geographical region of interest, which is assumed to be made of a linearly elastic material. 
This region contains a fault, modeled as an open curve $\Gamma \Subset D$, whose unknown shape is to be determined, see \cref{fig.fracset} (a).
Along the fault, the elastic displacement is discontinuous (i.e. the underground may slip differently on both sides), but the normal traction is continuous. 
The bottom side $\partial D_D$ of $D$ is assumed to be fixed and the remaining boundary $\partial D \setminus \overline{\partial D_D}$, which is composed of the upper surface $\Gobs$ and the lateral sides, is traction-free, i.e. homogeneous Neumann boundary conditions are applied.
The displacement of the crust is the unique solution $u_\Gamma \in H^1(D \setminus \Gamma)^2$ to the linear elasticity system: 
\begin{equation}\label{eq.elasfrac}
\left\{
\begin{array}{cl}
-\dv(Ae(u_\Gamma)) = 0 & \text{in } D \setminus \overline\Gamma, \\
u_\Gamma = 0 &\text{on } \partial D_D, \\
Ae(u_\Gamma)n = 0 & \text{on } \partial D \setminus \overline{\partial D_D},\\
\left[u_\Gamma\right] = g_\Gamma \text{ and } \left[Ae(u_\Gamma)n\right] = 0 & \text{on } \Gamma.
\end{array}
\right. 
\end{equation}
Here, $e(u) = \frac12(\nabla u+ \nabla u^T)$ is the strain tensor associated to a displacement field $u : D \to \R^2$ and $A$ is the Hooke's law of the elastic material in the crust: 
$$\text{For all } 2\times 2 \text{ symmetric matrix } e, \quad Ae = 2\mu e + \lambda \tr(e) \I, $$
where $\lambda = 0.5769$ and $\mu = 0.3846$ are the Lam\'e coefficients.
In \cref{eq.elasfrac}, $n$ stands for the unit normal vector to $\Gamma$; we also denote by:
\begin{equation}\label{eq.defjump}
\alpha^+(x) = \lim\limits_{t \to 0 \atop t > 0} \alpha(x+tn(x)), \:\: \alpha^-(x) = \lim\limits_{t \to 0 \atop t > 0} \alpha(x-tn(x)), \text{ and } \left[ \alpha \right] (x) = \alpha^+(x) - \alpha^-(x)
\end{equation}
the one-sided limits and the jump of a quantity $\alpha$ which is smooth enough from either side of $\Gamma$. 
The  so-called slip vector $g_\Gamma$ belongs to the functional space 
\begin{equation}\label{eq.H12t}
\widetilde{H}^{1/2}(\Gamma)^2 := \left\{ u \in L^2(\Gamma)^2 \text{ s.t. } \widetilde{u} \in H^1(\widetilde{\Gamma})^2 \right\}, 
\end{equation}
where $\widetilde{\Gamma}$ is an arbitrary extension of $\Gamma$ into a closed curve and $\widetilde{u}$ is the extension of $u$ to $\widetilde\Gamma$ by $0$.
The expression of $g_\Gamma$ depends on the shape of $\Gamma$. Let $c_0, c_1$ be the endpoints of $\Gamma$ and, for any points $x_0, x_1 \in \Gamma$, let $\Gamma_{x_0,x_1}$ denote the region of $\Gamma$ comprised between $x_0$ and $x_1$; $g_\Gamma(x)$ is of the form:
\begin{equation}\label{eq.slipvec}
\forall x \in \Gamma, \quad g_\Gamma(x) = g(s_\Gamma(x)), 
\end{equation}
where $g : [0,1] \to \R^2$ is a given smooth function vanishing at $0$ and $1$, and 
$$s_\Gamma(x) := \frac{\lvert \Gamma_{c_0,x}\lvert}{\lvert\Gamma\lvert}, \text{ where } \lvert \Gamma \lvert = \int_\Gamma \:\d s,$$ 
is the normalized arc length of the region $\Gamma_{c_0,x}$ of $\Gamma$ delimited by $c_0$ and $x$.

\begin{figure}[!ht]
    \centering
    \begin{tabular}{cc}
\begin{minipage}{0.48\textwidth}
\centering
\begin{overpic}[width=1.0\textwidth]{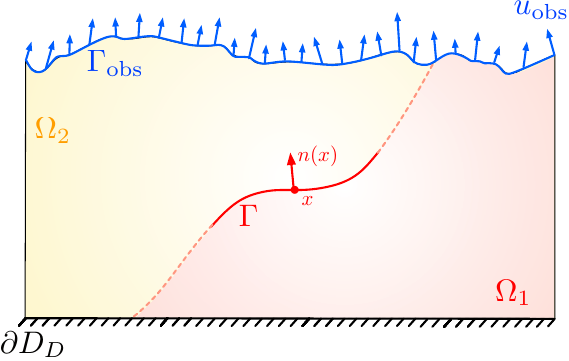}
\put(0,-3){\fcolorbox{black}{white}{a}}
\end{overpic}
\end{minipage} & 
\begin{minipage}{0.5\textwidth}
\centering
\begin{overpic}[width=1.0\textwidth]{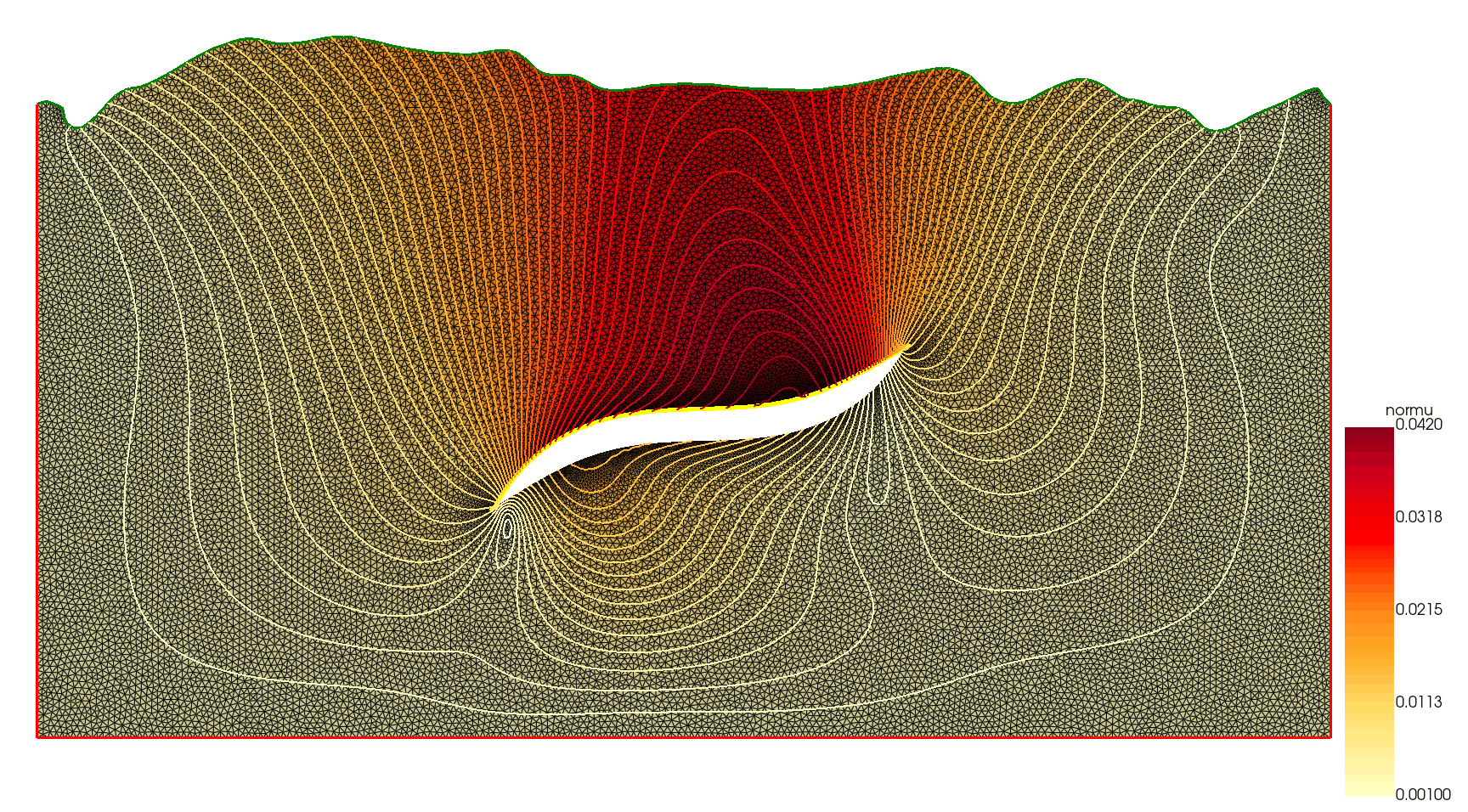}
\put(0,-3){\fcolorbox{black}{white}{b}}
\end{overpic}
\end{minipage}  
\end{tabular}  
 \caption{\it (a) Setting of the fracture identification example of \cref{sec.fault}; (b) Deformed configuration of the underground by application of the fields $u_{1,\e}$, $u_{2,\e}$ resulting from the approximate penalization and domain decomposition method of \cref{sec.faultdirect}; the color scale corresponds to the norm of the displacement.}
  \label{fig.fracset}
\end{figure}

\begin{remark}
The assumption that $g(t)$ vanishes at $t=0$ and $t=1$, ensuring that the slip vector $g_\Gamma$ belongs to the space $\widetilde{H}^{1/2}(\Gamma)$ in \cref{eq.H12t}, can be dispensed with, at the expense of dealing with a more intricate functional setting for the problem \cref{eq.elasfrac} \cite{aspri2025shape}.
\end{remark}

\subsubsection{Numerical simulation}\label{sec.faultdirect}

\noindent The numerical simulation of \cref{eq.elasfrac} is cumbersome, even when a body-fitted description of the fracture $\Gamma$ is available as a collection $\calL$ of edges of the mesh $\calT$ of the computational domain $D$.  Indeed, ``classical'' continuous finite element methods cannot easily represent discontinuous functions: one would have to duplicate the nodes on $\Gamma$ into different copies, associated to the triangles located on either side of the latter, and then to modifiy the element-to-vertex relations accordingly. This procedure is quite intrusive with respect to the implementation of the finite element solver, see nevertheless \cite{feppon2026fractured} for a numerical algorithm allowing to do so. 

Here, we rather rely on the elegant method introduced in \cite{allaire2021shape,allaire2022topology}, based on a duplication of the sought function. 
Briefly, let us introduce two subdomains $\Omega_1$, $\Omega_2 \subset D$ dividing $D$ according to $\Gamma$, that is:
$$\Omega_1 \cap \Omega_2 = \emptyset , \:\: \overline{D} = \overline{\Omega_1} \cup \overline{\Omega_2}, \text{ and } \Gamma \subset \partial \Omega_1\cap \partial \Omega_2. $$
The subdomains $\Omega_1$ and $\Omega_2$ lie respectively ``below'' and ``above'' the fracture set $\Gamma$: the normal vector $n$ is oriented from $\Omega_1$ to $\Omega_2$, see \cref{fig.fracset} (a). 
For instance, in the numerical context where $\Gamma$ is accounted for by two level set functions $\phi,\psi : D \to \R$ as in \cref{eq.2LS}, $\Omega_1$ and $\Omega_2$ may be defined as the negative and positive subdomains of $\phi$, respectively. 

Let us then introduce the functional space
$$H^1_{\partial D_D}(D) = \left\{ u \in H^1(D) \text{ s.t. } u = 0 \text{ on } \partial D_D \right\}, $$
and consider the following variational problem: 
\begin{multline}\label{eq.appvarffract}
\text{Search for } (u_{1,\e}, u_{2,\e}) \in H^1_{\partial D_D}(D)^4 \text{ s.t. } \forall (v_1, v_2) \in H^1_{\partial D_D}(D)^4, \\
 \int_D A_{1,\e} e(u_{1,\e}) : e( v_1) \:\d x +  \int_D A_{2,\e} e(u_{2,\e}) : e(v_2) \:\d x + \frac{1}{\e} \int_\Gamma ( u_{2,\e} - u_{1,\e} )  \cdot (v_2-v_1) \:\d s =\\
  \frac{1}{\e} \int_\Gamma g_\Gamma \cdot (v_2-v_1)  \:\d s,
\end{multline}
where we have defined the extensions of the Hooke's law $A$ outside $\Omega_1$ and $\Omega_2$ by:
$$A_{i,\e}(x) = \left\{
\begin{array}{cl}
A & \text{if } x \in \Omega_i, \\
\e A & \text{otherwise}.
\end{array}
\right. $$
Loosely speaking, \cref{eq.appvarffract} is obtained from \cref{eq.elasfrac} by duplicating the sought displacement $u$ (and the test function $v$) into the pair $(u_{1,\e}, u_{2,\e})$, where $u_{i,\e}$ is meant as an approximation of the restriction of $u$ to $\Omega_i$, $i=1,2$. Both $u_{1,\e}$ and $u_{2,\e}$ are continuous on $D$ but only their respective values on $\Omega_1$ and $\Omega_2$ are relevant. Furthermore, their traces do not agree on $\Gamma$: the difference $u_{2,\e} - u_{1,\e}$ is imposed to be an approximation of $g_\Gamma$ by penalization. 
The justification of this approximation is relatively classical; 
for completeness, the argument is sketched in \cref{app.justifbroken}, in the model setting of the conductivity equation, as a variation of the analysis conducted in \cite{allaire2021shape}. 

An application example of this approximation strategy is presented on \cref{fig.fracset} (b), where the solution $(u_{1,\e}, u_{2,\e})$ to \cref{eq.appvarffract} is computed in the physical situation of \cref{fig.fracset} (a). 
The penalization parameter is set to $\e = 1\text{e}^{-15}$ and the slip profile $g$ is defined by:
$$\forall s \in [0,1], \quad g(s) = \left\{
\begin{array}{cl}
\left(0.02\left(1-\left(\frac{s-0.5}{0.49}\right)^4 \right),0.05 \left(1-\left(\frac{s-0.5}{0.49}\right)^4 \right) \right) & \text{if } s \in [0.01,0.99],\\
0 & \text{otherwise} 
\end{array} 
\right.  $$

\subsubsection{Reconstruction of the fracture $\Gamma$ from observational data}\label{sec.IPfracture}

\noindent This section deals with the inverse problem associated to the ``forward'' model of \cref{sec.modfault,sec.faultdirect} for the elastic displacement of the underground in the presence of a fracture. We aim to reconstruct the shape of a fracture $\Gamma$ within the 2d elastic medium $D$ from the observation of the displacement $\uobs: \Gobs \to \R^2$ of the surface $\Gobs$. 

We follow the approach in \cite{aspri2025shape}.
The fracture $\Gamma$ is sought as a minimizer of the average on $\Gobs$ of the least-square difference between the observed displacement $\uobs$ and the model prediction $u_\Gamma$, solution to \cref{eq.elasfrac}. Precisely, we consider the optimization problem: 
\begin{equation}\label{eq.sopbfrac}
\min\limits_{\Gamma \subset D} J(\Gamma), \text{ where } J(\Gamma) = \frac12 \int_{\Gobs} \lvert u_\Gamma -\uobs \lvert^2 \:\d s. 
\end{equation}
The shape derivative of $J(\Gamma)$ is calculated in \cref{prop.sdfaultelas}.

The physical situation at stake is depicted on \cref{fig.fracinvdata} (a); the observed displacement $\uobs: \Gobs \to \R^2$ of the surface is that caused by the known fracture pattern $\Gamma_T$ in there. 
For a given shape $\Gamma$ of the fracture, the displacement $u_\Gamma$ is approximated via the couple $(u_{1,\e},u_{2,\e})$, solution to \cref{eq.appvarffract}, along the lines of \cref{sec.faultdirect}. 
The reconstruction problem \cref{eq.sopbfrac} is known to be very instable, especially when the slip vector vanishes at the endpoints of the fracture and only one measurement is used, see the discussion in \cite{aspri2025shape} where this phenomenon is observed even though the reconstructed fractured is constrained to be a line segment.
Here, we content ourselves with observing that the reconstructed fracture roughly has a similar inclination as the exact one, and that the objective function $J(\Gamma)$ is dutifully minimized: its value is decreased by 500 times within 100 iterations.

\begin{figure}[!ht]
    \centering
 \begin{tabular}{cc}
\begin{minipage}{0.45\textwidth}
\centering
\begin{overpic}[width=1.0\textwidth]{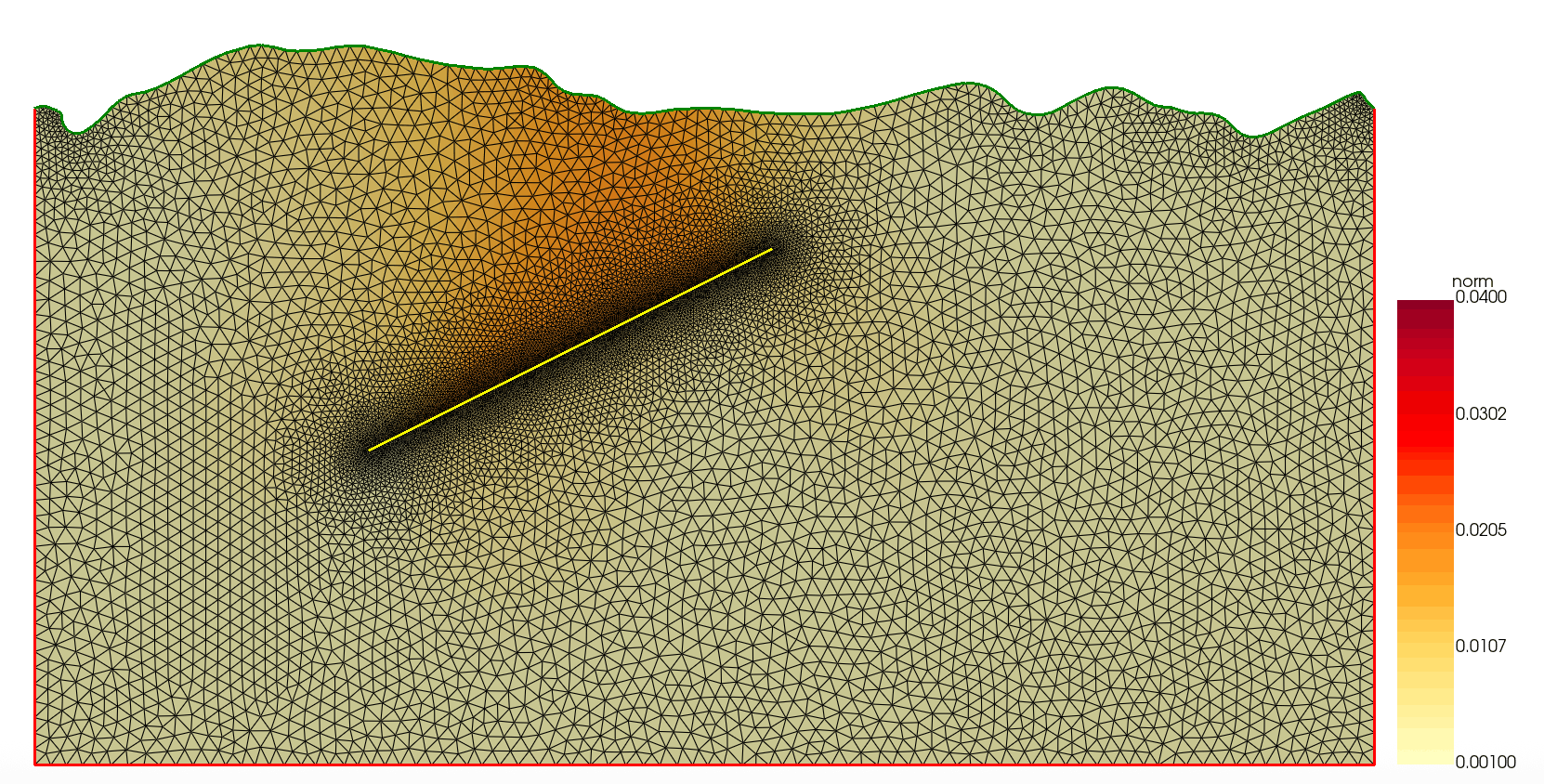}
\put(0,-3){\fcolorbox{black}{white}{$a$}}
\end{overpic}
\end{minipage} & 
\begin{minipage}{0.45\textwidth}
\centering
\begin{overpic}[width=0.7\textwidth]{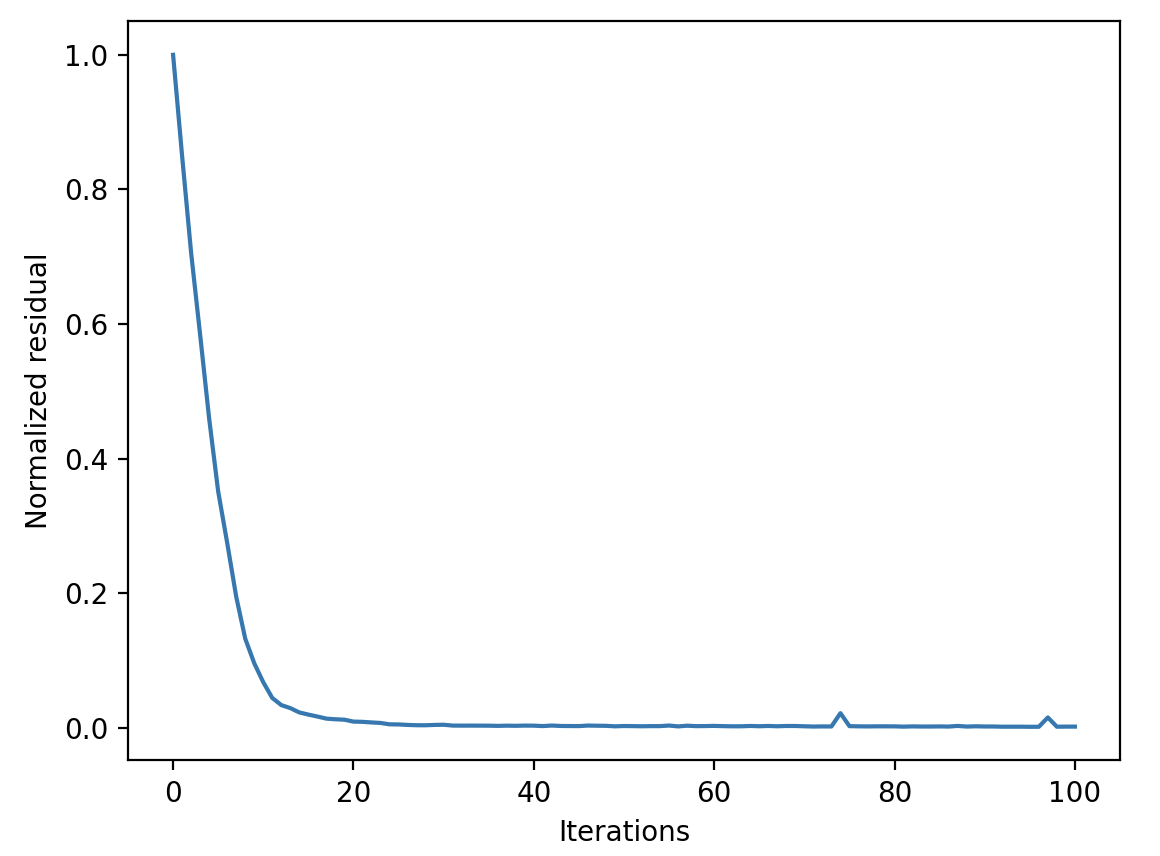}
\put(0,-3){\fcolorbox{black}{white}{$b$}}
\end{overpic}
\end{minipage} 
\end{tabular}  
 \caption{\it (a) Target fracture pattern $\Gamma_T$ in the example of \cref{sec.IPfracture}; (b) Convergence history of the normalized values $J(\Gamma^n)/J(\Gamma^0)$ of the objective function.}
  \label{fig.fracinvdata}
\end{figure}

\begin{figure}[!ht]
    \centering
    \begin{tabular}{cc}
\begin{minipage}{0.45\textwidth}
\centering
\begin{overpic}[width=1.0\textwidth]{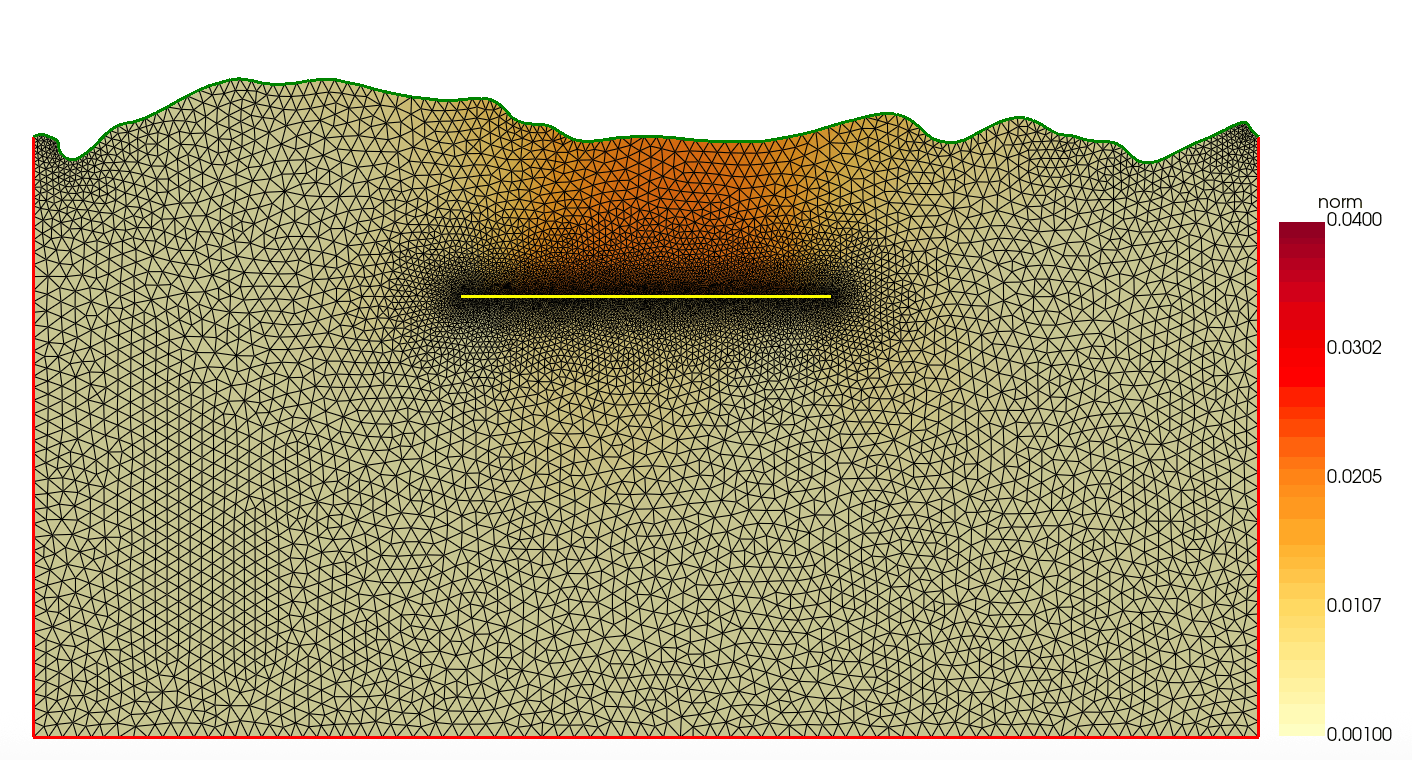}
\put(0,-3){\fcolorbox{black}{white}{$n=0$}}
\end{overpic}
\end{minipage}
&
\begin{minipage}{0.45\textwidth}
\centering
\begin{overpic}[width=1.0\textwidth]{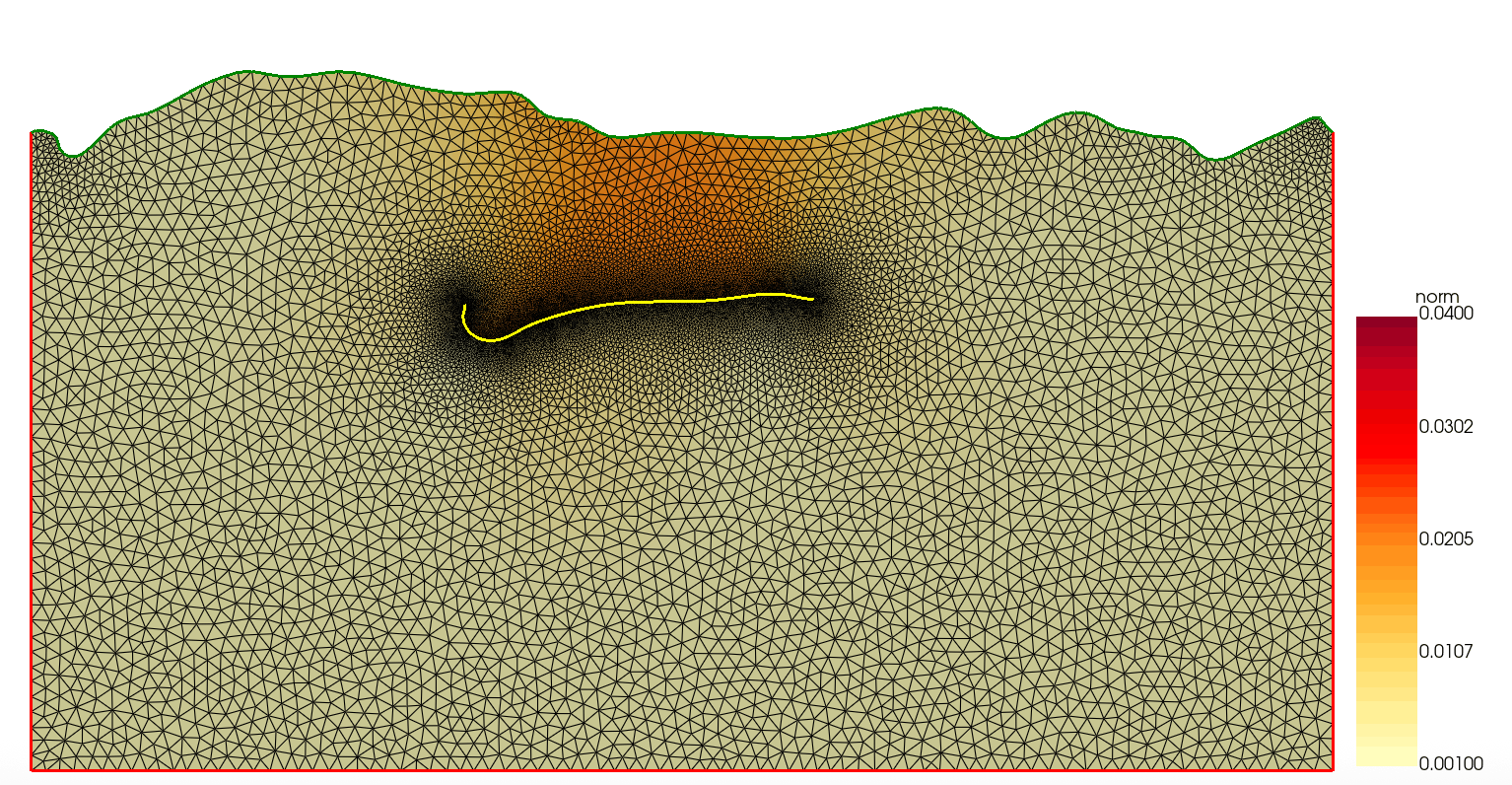}
\put(0,-3){\fcolorbox{black}{white}{$n=5$}}
\end{overpic}
\end{minipage} 
\end{tabular}  
\par\bigskip
 \begin{tabular}{cc}
\begin{minipage}{0.45\textwidth}
\centering
\begin{overpic}[width=1.0\textwidth]{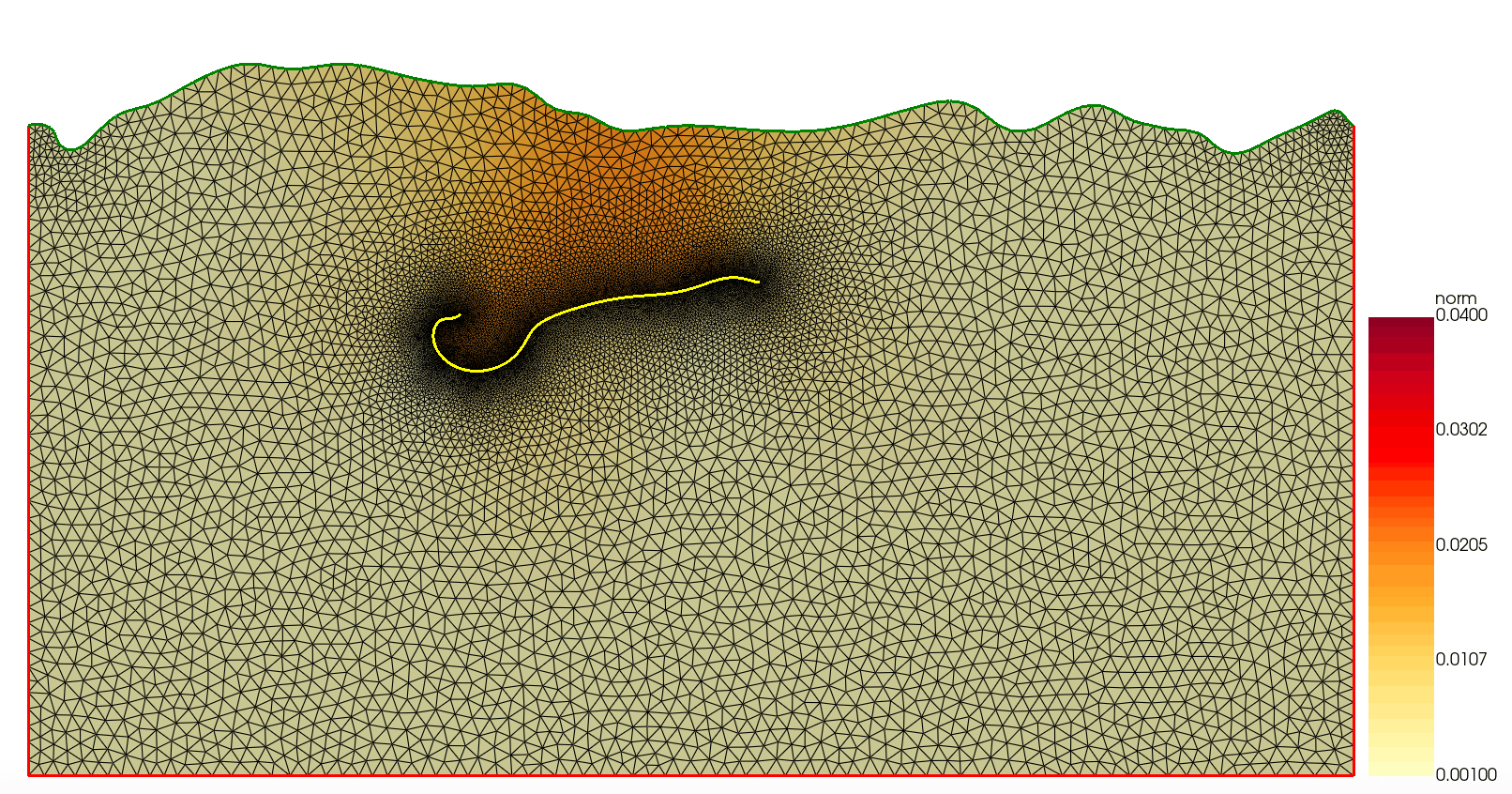}
\put(0,-3){\fcolorbox{black}{white}{$n=10$}}
\end{overpic}
\end{minipage} & 
\begin{minipage}{0.45\textwidth}
\centering
\begin{overpic}[width=1.0\textwidth]{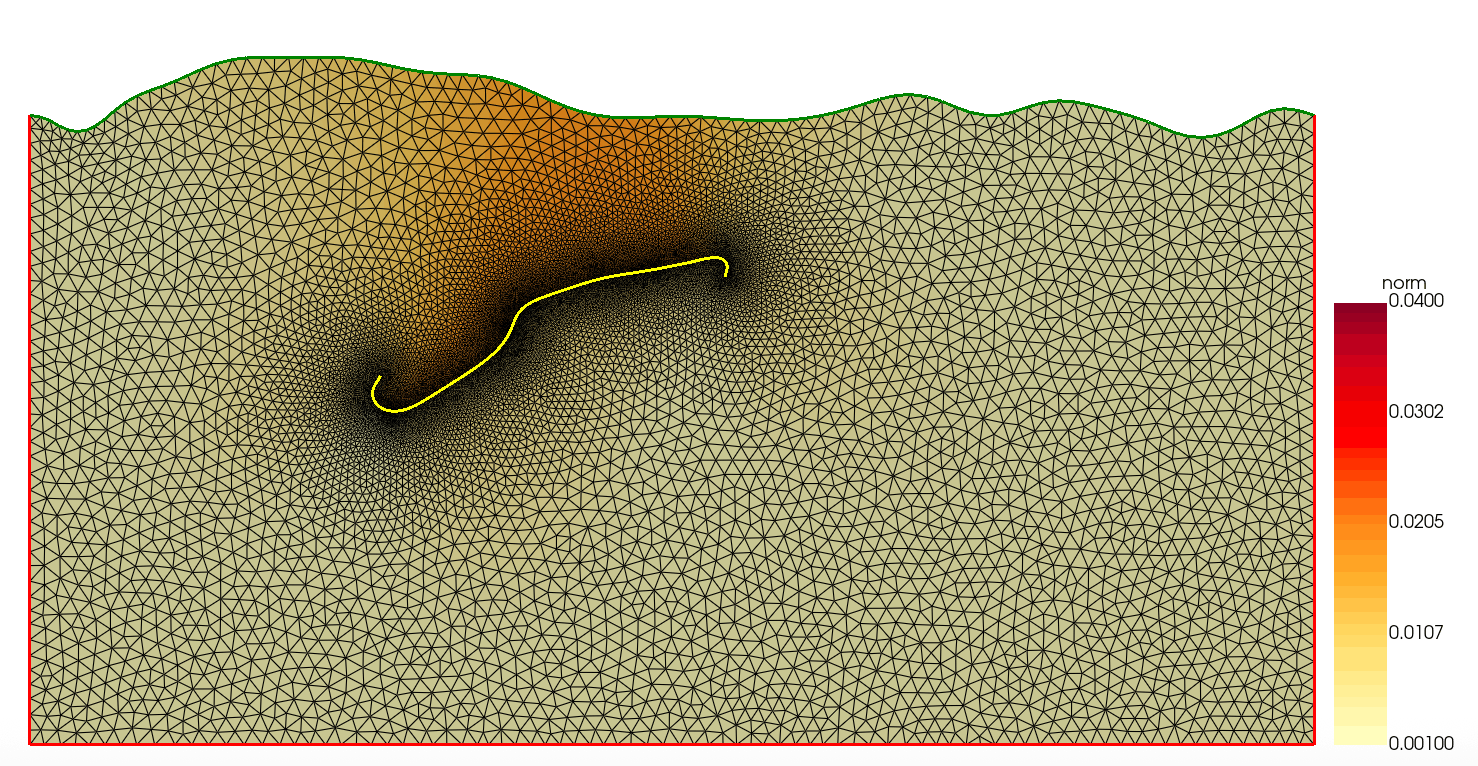}
\put(0,-3){\fcolorbox{black}{white}{$n=100$}}
\end{overpic}
\end{minipage} 
\end{tabular}  
 \caption{\it A few iterates in the fracture detection example of \cref{sec.IPfracture}.}
  \label{fig.fracinv}
\end{figure}

\begin{remark}
Note that, in principle, the slip vector should be made part of the reconstruction, which entails no conceptual change in the above method. 
\end{remark}

\section{Conclusion and perspectives}\label{sec.concl}

\noindent In this article, we have proposed a novel numerical framework for tracking the motion of an evolving open curve, or a collection of such.
This strategy combines two representations of the curve at each iteration of the evolution process: (i) an explicit discretization, as a subset of the edges of the computational mesh, which allows to perform accurate evaluations of its geometric features and precise finite element computations associated to its physical behavior; (ii) an implicit representation, via a variant of the Level Set Method, which allows to capture its evolution in a robust and efficient manner, however dramatic.
We have used our framework to address various problems, such as the simulation of a vortex sheet roll-up, the optimization of the path of the laser ensuring the 3d printing of a shape by the Electron Beam Melting technology, or the reconstruction of a fault in the underground from observational data.

This work is preliminary in many respects and calls for multiple perspectives.
At first, the proposed methodology can be applied to various physical problems, beyond those exemplified in this article:
\begin{itemize}
\item Open curves are ubiquitous in image segmentation, where they can be optimized to best indicate sharp variations of intensity or color, see e.g. \cite{basu2007implicit,leung2009grid}.
\item A similar formalism to that of the path optimization problem discussed in \cref{sec.am} could be applied to the optimal graft of a thin ligament to a structure in order to improve its robustness: our previous works \cite{dapogny2020connection,dapogny2020topolig} indeed reduce this task to the minimization of an anisotropic perimeter functional after calculation of a ``topological ligament expansion'' of a mechanical shape functional of interest. 
\item We aim to apply the methodology in the physical setting of electromagnetism, to optimize the placement of current lines within the ferromagnetic region of an electric motor. 
\end{itemize}

Another lead for improving the proposed framework concerns the handling of one-dimensional structures showing branching or multiple junctions:
its coupling with a strategy for decomposing a branched structure into a collection of manifold curves would allow to simulate complex fracture propagation patterns.

Finally, and perhaps most importantly, as we have mentioned in the introduction, 
the original motivation and natural perspective of this work is to develop an efficient and accurate, body-fitted tracking strategy for the motion of open surfaces in three space dimensions, with applications including the optimization of shells or electromagnetic screens.
Although the implementation effort is then expected to be more intense, the methodology of this article was developed with this ambition in mind, and its extension to this setting does not present any theoretical obstruction. 

\par\bigskip

\noindent \textbf{Acknowledgements.} This work is partially supported by the projects ANR-24-CE40-2216 STOIQUES, ANR-22-CE46-0006
StableProxies and CNRS-MITI DOLMEN.
Part of it was realized while the author was visiting the Department of Mathematics of the University of Trento, whose hospitality is gratefully acknowledged.

\appendix
\section{A few facts from tangential calculus}\label{app.remtgtcalc}

\noindent This appendix gathers some basic facts from differential calculus on a smooth hypersurface of $\R^d$, $d \geq 2$. We mainly follow \cite{henrot2018shape}, 
see also classical books such as \cite{chavel2006riemannian,lang2012fundamentals} for differential calculus on manifolds. \par \medskip

Let $\Gamma$ be an oriented open surface in $\R^d$, $d \geq 2$, whose contour $\Sigma := \partial \Gamma$ is then a $(d-2)$-dimensional closed submanifold of $\R^d$. 
We denote by $n : \Gamma \to \R^d$ the unit normal vector, whose orientation is consistent with that of $\Gamma$. Moreover,

\begin{itemize}
\item The tangential gradient of a function $u : \Gamma \to \R$ of class $\calC^1$ is defined by $\nabla_\Gamma u := \nabla \widetilde u - (\nabla \widetilde u \cdot n ) n$, 
where $\widetilde{u}$ is an arbitrary extension of $u$ to an open neighborhood of $\Gamma$ in $\R^d$.
\item The tangential divergence of a vector field $v : \Gamma \to \R^d$ is defined by $\dv_\Gamma(v) = \dv(\widetilde{v}) - \nabla \widetilde v n \cdot n $,
where $\widetilde v$ is an arbitrary extension of $v$ to an open neighborhood of $\Gamma$ in $\R^d$.
\end{itemize} 

Let us now recall the following formula for changing variables in surface integrals.

\begin{proposition}\label{prop.chgvarintsurf}
Let $\Gamma$ be a smooth open surface of $\R^d$ and let $T : \R^d \to \R^d$ be a diffeomorphism of class $\calC^1$. Then a function $f$ belongs to $L^1(T(\Gamma))$ if and only if $f \circ T$ belongs to $L^1(\Gamma)$, and the following equality holds true:
$$\int_{T(\Gamma)} f \:\d s = \int_\Gamma \lvert \com(\nabla T)n \lvert (f \circ T) \:\d s , $$
where $\com(M)$ is the cofactor matrix of a $d \times d$ matrix and the term $\lvert \com(\nabla T)n \lvert$ is called the tangential Jacobian of $T$.
\end{proposition}

We now state a well-known integration by parts formula on a surface. 

\begin{proposition}\label{prop.IPPsurf}
Let $\theta : \R^d \to \R^d$ and $f: \R^d \to \R$ be a vector field and a function of class $\calC^1$, respectively. Then, the following integration by parts formula holds true: 
$$ \int_\Gamma \dv_\Gamma (\theta) f \:\d s = \int_\Gamma \kappa f  \theta \cdot n \:\d s -\int_\Gamma \theta \cdot \nabla_\Gamma f \:\d s + \int_\Sigma f \theta\cdot n_\Sigma \:\d\ell,$$
where $\kappa := \dv_\Gamma (n)$ is the mean curvature of $\Gamma$.
\end{proposition}

\section{Calculation of the shape derivative of anisotropic perimeter functionals}\label{app.deranisoper}

\noindent This appendix details the calculation of the shape derivatives of the anisotropic perimeter functionals considered in \cref{sec.aniper}.
Placing ourselves in the general context of a $d$-dimensional ambient space, we consider a functional of the form 
$$ J(\Gamma ) =\int_\Gamma \varphi(x, n_\Gamma(x)) \:\d s,$$
depending on an open smooth hypersurface $\Gamma$ through its unit normal vector $n_\Gamma : \Gamma \to \R^d$, that we simply denote by $n$ when the concerned hypersurface is clear. Here, $\varphi : \R^d_x \times \R^d_n \times \R$ is a given smooth function;
the gradients of the partial mappings $x \mapsto \varphi(x,n)$ and $n \mapsto \varphi(x,n)$ are respectively denoted by $\nabla_x \varphi (x,n)$ and $\nabla_n \varphi(x,n)$.

\begin{proposition}
The functional $J(\Gamma)$ is shape differentiable at any smooth, open hypersurface $\Gamma$, and its derivative reads:
$$J^\prime(\Gamma)(\theta) = \int_\Gamma \kappa \varphi(x,n(x)) \:\theta\cdot n \:\d s +  \int_\Sigma  \varphi(x,n(x)) \theta\cdot n_\Sigma \:\d s  +  \int_\Gamma \frac{\partial \varphi}{\partial n_x}(x,n(x)) \theta\cdot n \:\d s  \\
 - \int_\Gamma \nabla_n\varphi(x,n(x)) \cdot \nabla_\Gamma(\theta\cdot n)\:\d s. $$
\end{proposition}
\begin{proof}
For a given perturbation $\theta$, it holds:
$$ J(\Gamma_\theta) =\int_{\Gamma_\theta} \varphi(x, n_{\Gamma_\theta}(x)) \:\d s,$$
and so a change of variables in terms of the mapping $(\Id + \theta)$ yields: 
\begin{equation}\label{eq.Jthetaanisoper}
J(\Gamma_\theta) =\int_{\Gamma} \lvert \com (\I + \nabla \theta) n \lvert \varphi(x + \theta(x), n_{\Gamma_\theta}(x + \theta(x))) \:\d s, \text{ where } n \equiv n_\Gamma.
\end{equation}
We now rely on the following formula for the transported version of the normal vector field: 
$$n_{\Gamma_\theta}(x+ \theta(x)) = \frac{\com(\I + \nabla \theta) n(x)}{\lvert \com(\I + \nabla \theta) n(x) \lvert} , \quad x \in \Gamma. $$
On a different note, the well-known formula $\com(M) = \det(M) M^{-T}$ easily implies the following expansion:
\begin{equation}\label{eq.dertgtJac}
\lvert \com (\I + \nabla \theta) n \lvert = \dv_\Gamma (\theta) + \o(\theta),
\end{equation}
where $ \dv_\Gamma (\theta) :=  \dv(\theta) - \nabla \theta n \cdot n$ is the tangential divergence of $\theta$. It follows that:
$$ n_{\Gamma_\theta}(x+ \theta(x))  = n -\nabla\theta^Tn + (\nabla \theta n \cdot n )n + \o(\theta).$$ 
Hence, taking derivatives in \cref{eq.Jthetaanisoper}, we obtain: 
$$J^\prime(\Gamma)(\theta) = \int_\Gamma \dv_\Gamma(\theta) \varphi(x,n(x)) \:\d s  +  \int_\Gamma \nabla_x\varphi(x,n(x)) \cdot \theta \:\d s   + \int_\Gamma \nabla_n\varphi(x,n(x)) \cdot \left( -\nabla\theta^T n+ (\nabla \theta n \cdot n) n\right)\:\d s.  $$
We now integrate by parts on the boundary in the first integral of the above right-hand side thanks to \cref{prop.IPPsurf}. Denoting by $\theta_\Gamma = \theta - (\theta\cdot n)n$ the tangential component of $\theta$, this yields:
\begin{multline}\label{eq.T1appB}
J^\prime(\Gamma)(\theta) = \int_\Gamma \kappa \varphi(x,n(x)) \:\theta\cdot n \:\d s +  \int_\Sigma  \varphi(x,n(x)) \theta\cdot n_\Sigma \:\d s  +  \int_\Gamma \nabla_x\varphi(x,n(x)) \cdot (\theta - \theta_\Gamma) \:\d s  \\
 + \int_\Gamma \nabla_n\varphi(x,n(x)) \cdot \left( -\nabla\theta^T n+ (\nabla \theta n \cdot n) n - \nabla n \theta_\Gamma \right)\:\d s. 
\end{multline}
In order to treat the final term, we remark that $\nabla n \theta_\Gamma = \nabla n \theta$, since $\nabla n n = 0$. 
Hence, we obtain: 
\begin{equation}\label{eq.T2appB}
 \begin{array}{>{\displaystyle}cc>{\displaystyle}l} 
 -\nabla\theta^T n+ (\nabla \theta n \cdot n) n - \nabla n \theta_\Gamma &=&  (\nabla \theta n \cdot n) n - \nabla (\theta\cdot n)\\
 &=&( \nabla(\theta\cdot n)\cdot n ) n - \nabla (\theta\cdot n) \\
 &=& - \nabla_\Gamma(\theta\cdot n).
 \end{array}
 \end{equation}
 Combining \cref{eq.T1appB,eq.T2appB}, we arrive at the desired formula.
 \end{proof} 
%
%

\section{Additional details about the 2d vortex sheet evolution problem}\label{app.vortex}

\noindent This appendix expands a little on the vortex sheet roll-up example considered in \cref{sec.vortex} and its numerical treatment. 
The mathematical model is described in a formal way in \cref{app.vormod} and a few additional details about its practical implementation are provided in \cref{app.vorimp}. 

\subsection{Mathematical model of the vortex sheet evolution in 2d}\label{app.vormod}

\noindent A vortex sheet is a phenomenon that may develop within an incompressible and nearly inviscid fluid, i.e. when the Reynolds number is very high. 
As we have mentioned in \cref{sec.vortex}, it takes the form of a curve (in 2d) or a surface (in 3d) $\Gamma(t)$ across which the fluid velocity ``slips'', i.e. its normal component is continuous, but its tangential component is discontinuous. 
Such a structure has its own dynamics: it expands and rolls up under the effect of the complex velocity patterns that it itself creates.
We refer to \cite{chorin1990mathematical,kundu2024fluid} for general references about fluid mechanics, to \cite{girault2012finite,temam2001navier} about its mathematical framework, and to \cite{cottet2001vortex,majda2002vorticity,marchioro2012mathematical,saffman1992vortex} about the specific subject of vortex patterns.

Formally, let us assume that the fluid under scrutiny is contained in the infinite plane $\R^2$, i.e. we neglect the boundary effects induced by the use of a fixed computational domain $D$.
We denote by $u : \R_t \times \R_x^2 \to \R^2$ and $p:\R_t \times \R_x^2 \to \R$ the fluid velocity and pressure, respectively. 
Denoting by $T$ the final time of the study, the couple $(u,p)$ is solution to the inviscid, incompressible Navier-Stokes equations:
\begin{equation}\label{eq.NS} 
\left\{
\begin{array}{cl}
\frac{\partial u}{\partial t} + \nabla u u + \nabla p = 0 & \text{for } (t,x) \in (0,T) \times \R^2, \\
\dv(u) = 0 & \text{for } (t,x) \in (0,T) \times \R^2, \\
u(0,x) = u_0(x) & \text{for } x \in \R^2.
\end{array}
\right.
\end{equation}
The first equation in \cref{eq.NS} accounts for the balance of momentum within the fluid, while the second one accounts for its incompressibility;
the initial velocity profile $u_0: \R^2 \to \R^2$ is assumed to be known. 
 
A handful quantity in the description of this situation is the vorticity of the fluid $\omega : \R_t \times \R_x^2 \to \R$, which is defined by:
\begin{equation}\label{eq.omcurl}
\omega(t,x) = \curl(u)(t,x) := \frac{\partial u_2}{\partial x_1}(t,x) - \frac{\partial u_1}{\partial x_2} (t,x).
\end{equation}
Intuitively, at each time $t$, $\omega(t,x)$ accounts for the infinitesimal rotating motion of the fluid around $x$. 
By taking the curl of the balance of momentum relation and using the incompressibility condition $\dv(u) = 0$, an elementary calculation yields the following equation about $\omega$:
$$\frac{\partial \omega}{\partial t} + u\cdot \nabla \omega =0, $$
which expresses that the vorticity is transported along with the fluid particles. 

Another important quantity in the description of vortex sheets is the stream function $\psi: \R_t \times \R_x^2 \to \R$ of the fluid. 
Since $\dv(u) = 0$, the Helmholtz decomposition implies that there exists a function $\psi$, which is unique up to a constant, such that:
\begin{equation}\label{eq.defpsi}
u = \nabla^\perp \psi := \left(-\frac{\partial \psi}{\partial x_2},\frac{\partial\psi}{\partial x_1}\right).
\end{equation}
Combining this identity with the definition \cref{eq.omcurl} of $\omega$, we arrive at:
\begin{equation}\label{eq.Deltapsi}
\Delta \psi = \omega. 
\end{equation}\par\medskip 

Let us now assume that a vortex sheet is present within the fluid, that takes the form of an oriented open curve $\Gamma(t)$. We denote by $n_t$ the unit normal vector to $\Gamma(t)$
and by $\tau_t$ its unit tangent vector, such that $(\tau_t,n_t)$ is a direct orthonormal frame of the plane, i.e. $n_t = \tau_t^\perp := (-\tau_{t,2,}\tau_{t,1})$, see \cref{fig.vorset} (b).
Let us recall from \cref{sec.fault} that, when $\alpha$ is a quantity which is smooth on either side of $\Gamma(t)$, but possibly discontinuous across $\Gamma(t)$, we denote by
$$ \alpha^\pm(x) = \lim\limits_{s \to 0 \atop s > 0} \alpha(x \pm s n_t(x)) \text{ and } \left[ \alpha \right] (x) = \alpha^+(x) - \alpha^-(x)$$
the one-sided limits and the jump of $\alpha$ at $x \in \Gamma(t)$, respectively.
Mathematically, the vortex sheet pattern is characterized by the fact that, at each time $t$, the vorticity is concentrated on $\Gamma(t)$, i.e.
$$\omega(t,x) = \gamma(t,x) \delta_{\Gamma(t)}, $$
where $\delta_{\Gamma}$ is the distribution accounting for integration over $\Gamma(t)$, and the scalar factor $\gamma(t,x)$ is called the strength of the vortex. 
This relation has to be understood in the sense of distributions in $\R^2$: for any open subset $U \subset \R^2$, it holds:
\begin{equation}\label{eq.locvor}
\int_U \omega(t,x) \:\d x = \int_{\Gamma(t) \cap U} \gamma(t,x) \:\d s.
\end{equation}
Substituting \cref{eq.Deltapsi} for $\omega(t,x)$ in \cref{eq.locvor} and integrating by parts, we arrive at the following boundary-value problem for the stream function $\psi(t,\cdot)$ at each time $t>0$: 
\begin{equation}\label{eq.bvppsi}
\left\{
\begin{array}{cl}
\Delta \psi = 0 & \text{in } \R^2 \setminus \Gamma(t), \\
-\left[ \frac{\partial \psi}{\partial n}\right] = \gamma(t,\cdot) & \text{on } \Gamma(t),
\end{array}
\right.
\end{equation}
where the sign comes from the chosen orientation in the definition of the jump across $\Gamma(t)$. 
It follows from \cref{eq.bvppsi} that $\psi(t,\cdot)$ can be expressed as a single layer potential, 
$$\psi(t,x) = -\int_{\Gamma(t)} \gamma(t,y) G(x,y) \:\d s(y), $$
where $G(x,y) = \frac{1}{2\pi} \log\lvert x - y \lvert$ is the fundamental solution of the Laplace operator in 2d, see for instance \cite{folland1995introduction,mclean2000strongly} about the topic of potential theory. 
Using once again the relation \cref{eq.defpsi} between $u$ and $\psi$, we arrive at the following representation formula for $u$ in terms of $\gamma$: 
$$u(t,x) = \int_{\Gamma(t)} \gamma(t,y) K(x,y) \:\d s(y), \text{ where } K(x,y) := -\frac{1}{2\pi} \frac{ (x-y)^\perp}{\lvert x - y\lvert^2}. $$
The classical jump relations for the single layer potential readily imply that $u$ has continuous normal component across $\Gamma(t)$, but that its tangential component is discontinuous: 
\begin{equation}\label{eq.jump}
\left[u\cdot n_t \right] = 0 ,\text{ and } (u \cdot \tau_t)^\pm= \pm \frac{1}{2} \gamma(t,x) + \pv \int_{\Gamma(t)} \gamma(t,y) K(x,y) \cdot \tau_t(x) \:\d s(y), 
\end{equation}
where the last integral is understood in the sense of a Cauchy principal value, see again \cite{folland1995introduction,mclean2000strongly}. In particular, this shows that the vortex strength $\gamma(t,x)$ coincides with the jump $-[u\cdot \tau_t]$ in tangential velocity. 
\par\medskip 

Let us now describe the evolution of $\Gamma(t)$. It can be shown that the consistency of the latter with the inviscid Navier-Stokes equations \cref{eq.NS} for the surrounding fluid imposes that the velocity field driving its motion should be equal to the average of the one-sided values of the fluid velocity, see Chap. 6 of \cite{marchioro2012mathematical}. Using the jump relations \cref{eq.jump}, this means that the velocity of $\Gamma(t)$, that we still denote by $u(t,x)$ with a small abuse of notation, is given by the so-called Birkhoff-Rott equation: 
\begin{equation}\label{eq.pvuvortex}
 u(t,x) =  \pv  \int_{\Gamma(t)} \gamma(t,y) K(x,y) \:\d s(y).
 \end{equation}
To complete this model, we need to characterize the evolution of the vortex strength $\gamma(t,x)$. This task involves another conservation principle: the total vortex strength contained in any portion of $\Gamma(t)$ is conserved during the motion. 
Mathematically, for any $s >0$, let us denote by $t\mapsto X_s(t,x)$ the characteristic curve of the velocity field $u(t,x)$ emerging from $x$ at time $s$, which is the unique solution to the following ordinary differential equation, see \cref{eq.characcurveexpl}:
\begin{equation}\label{eq.odecharacbis} 
\left\{
\begin{array}{cl}
\frac{\d X_s}{\d t}(t,x) = u(t,X_s(t,x)) & \text{for } t \in (0,T), \\
X_s(s,x) = x.&
\end{array}
\right.
\end{equation}
Mathematically, the above conservation principle means that, for any region $G \subset \Gamma(0)$, it holds:
$$\frac{\d }{\d t} \left( \int_{X_0(t,G)} \gamma(t,y) \:\d s(y) \right) = 0.$$
Using the change of variables in surface integrals recalled in \cref{prop.chgvarintsurf}, this rewrites: 
$$\frac{\d }{\d t} \left( \int_{G} \lvert \com\nabla X_0(t,y) n_t(y) \lvert  \gamma(t,X_0(t,y)) \:\d s(y) \right) = 0.$$
Finally, since this property holds for any subset $G \subset \Gamma(0)$, the quantity $ \lvert \com\nabla X_0(t,y) n_t(y) \lvert  \gamma(t,X_0(t,y)) $ is conserved, and so:
\begin{equation}\label{eq.circconserv}
\text{For all } x \in \Gamma(0), \quad \lvert \nabla X_0(t,x) \tau_0(x) \lvert \gamma(t,X_0(t,x)) = \gamma(0,x),
\end{equation}
where we have used the following elementary calculation: 
$$\lvert \com\nabla X_0(t,x) n_0(x) \lvert  = \lvert \nabla X_0(t,x) \tau_0(x) \lvert. $$
Equivalently, \cref{eq.circconserv} has the following equivalent backward expression: 
\begin{equation}\label{eq.circconservback}
\text{For all } y \in \Gamma(t), \quad \gamma(t,y) = \frac{\gamma(0,X_t(0,y))}{\lvert \nabla X_0(t,X_t(0,y)) \tau_0(X_t(0,y))\lvert} = \frac{\gamma(0,X_t(0,y))}{\lvert \nabla X_t(0,y)^{-1} \tau_0(X_t(0,y))\lvert}. 
\end{equation}
Finally, the equations \cref{eq.pvuvortex,eq.circconserv} (or \cref{eq.circconservback}) completely characterize the dynamics of the vortex sheet $\Gamma(t)$.
\par\medskip 

\subsection{A few implementation details}\label{app.vorimp}

\noindent The numerical implementation of the above model for the simulation of vortex sheets raises a few practical issues, that we broach in this section.

\subsubsection{Update of the vortex strength}

\noindent At each iteration $n=0,\ldots$ of the evolution, the vortex strength $\gamma(t^n,x)$ is discretized at the vertices of the line mesh $\calL^n$ of $\Gamma(t^n)$, whose edges explicitly appear in the mesh $\calT^n$ of $D$, as one key feature of our body-fitted tracking method, see \cref{sec.algo}. The calculation of the updated values $\gamma(t^{n+1},\cdot)$ at the vertices of the mesh $\calL^{n+1}$ for $\Gamma(t^{n+1})$ from those of $\gamma(t^n,\cdot)$ on $\Gamma(t^n)$ is based on the formula \cref{eq.circconservback}, which can  be given an iterative flavor: 
\begin{equation}\label{eq.vorstrenupdate}
\text{For each } y \in \Gamma(t^{n+1}), \quad \gamma(t^{n+1},y) = \frac{\gamma(t^n,(T^{-n}(y)))}{\lvert \nabla T^n (T^{-n}(y)) \tau(T^{-n}(y))\lvert},
\end{equation}
where we have used the shortcuts 
$$T^n = X_{t^n}(t^{n+1},\cdot) : \Gamma(t^n) \to \Gamma(t^{n+1}), \text{ and } T^{-n} = (T^n)^{-1}:  \Gamma(t^{n+1}) \to \Gamma(t^{n}),$$
for the flow of $u(t,x)$ between $t^n$ and $t^{n+1}$ and its inverse, respectively. 
The update of $\gamma(t^n,\cdot)$ into $\gamma(t^{n+1},\cdot)$ from this formula thus requires to: 
\begin{enumerate}
\item Calculate the position $T^{-n}(y)$ on $\Gamma(t^n)$ associated to each vertex $y$ on $\calL^{n+1}$; 
\item Calculate the tangential derivative of the flow mapping $T^n$ on $\Gamma(t^n)$;
\item Assemble the expression \cref{eq.vorstrenupdate} from these data.  
\end{enumerate} 

In numerical practice, the flow mapping $T^n$ (resp. its inverse $T^{-n}$) can be calculated by solving the ordinary differential equation \cref{eq.odecharacbis} (resp. its inverse, forward expression) starting from each vertex of $\Gamma(t^{n+1}$) (resp. of $\Gamma(t^n)$). However, when evaluating \cref{eq.vorstrenupdate}, it is crucial to ensure that the particle positions $T^{-n}(y) \in \Gamma(t^n)$ attached to the vertices $y$ of $\calL^{n+1}$ belong exactly to the mesh $\calL^n$ of $\Gamma(t^n)$, since $\gamma(t^n,\cdot)$ is solely defined on this mesh. 
To achieve this, we rely on the following remark about the transformation of the arc length through a diffeomorphism. Let $\Gamma$ be a smooth 2d curve with endpoints $c_0$, $c_1$;  
the normalized arc length function $s_\Gamma: \Gamma \to [0,1]$ reads: 
$$\forall x \in \Gamma, \quad s_\Gamma(x) = \frac{1}{\lvert \Gamma \lvert }\int_{\Gamma_{c_0,x}} \:\d s, $$
where $\lvert \Gamma \lvert $ is the length of $\Gamma$ and we recall that $\Gamma_{c_0,x}$ stands for the portion of curve comprised between $c_0$ and $x$. 
Let now $T : \R^2 \to \R^2$ be a smooth diffeomorphism. This function then transforms as follows under the effect of $T$: 
$$s_{T(\Gamma)}(T(x)) = \frac{1}{\lvert T(\Gamma) \lvert} \int_{T(\Gamma)_{T(c_0),T(x)}} \:\d s = \frac{1}{\lvert T(\Gamma) \lvert} \int_{\Gamma_{c_0,x}} \lvert \com(\nabla T) n \lvert \:\d s  $$
Applying this observation to the mapping $T^{-n}$, it is possible to identify the arc length on $\Gamma(t^n)$ associated to the particle position $T^{-n}(y)$ on the mesh $\calL^n$ of $\Gamma(t^n)$ associated to each vertex $y$ of $\calL^{n+1}$. 

\subsubsection{Calculation of the velocity field}

\noindent At each iteration $n=0,\ldots$ of the evolution, the curve $\Gamma(t^n)$ is discretized by a line mesh $\calL^n$; once the vortex strength $\gamma(t^n,x)$ is computed at the vertices of $\calL^n$, 
the formula \cref{eq.pvuvortex} for the velocity $u(t^n,\cdot)$ is evaluated by quadrature.
This task is a little delicate for two reasons: 
\begin{enumerate}
\item The kernel $K(x,y)$ featured in \cref{eq.pvuvortex} blows up when $x = y$, and this formula makes sense as a Cauchy principal value.
As suggested in \cite{krasny1986study}, to ease the numerical computation, we rely on the so-called ``vortex-blob'' approximate formula:
\begin{equation}\label{eq.blobpv}
u(t,x) \approx \frac{1}{2\pi} \int_{\Gamma(t)} \gamma(t,y) \frac{(x-y)^\perp}{\lvert x - y \lvert^2 + \delta^2}\:\d s(y),
\end{equation}
where the presence of the ``small'' parameter $\delta>0$ at the denominator of the above integrand removes the singularity of the kernel. 
\item Often in practice, the vortex strength $\gamma(t,x)$ takes infinite values at the endpoints of $\Gamma(t)$ while still being an integrable function on $\Gamma(t)$, see for instance the expression \cref{eq.inivorstren} used in the numerical example of \cref{sec.vortex}. To alleviate  this issue, we use a quadrature formula such as the simple midpoint rule, that evaluates the integrand of \cref{eq.blobpv} using quadrature points located inside the edges of $\Gamma(t)$ and not at its endpoints.
\end{enumerate} 

\section{Details about the treatment of the laser path optimization problem}\label{app.sdam}

\noindent This appendix details the calculation of the shape derivative of a slightly more general version of the functional \cref{eq.Jam}, associated to the laser path optimization example of \cref{sec.am}. We consider the quantity 
$$J(\Gamma) = \int_D j(x,u_\Gamma(x)) \:\d x, $$
depending on the open curve $\Gamma$ contained in the fixed computational domain $D$ via the
solution $u_\Gamma \in H^1(D)$ to the boundary-value problem \cref{eq.heateq} and $j : \R^2_x \times \R_u \to \R$ is a smooth function, satisfying suitable growth conditions. For any $x \in D$, $u \in \R$, we denote by $\nabla_x j(x,u)$ and $\frac{\partial j}{\partial u}(x,u)$ the gradient and derivative of the partial mappings $x \mapsto j(x,u)$ and $u \mapsto j(x,u)$, respectively. 

The shape derivative of $J(\Gamma)$ is the subject of the next proposition.

\begin{proposition}\label{prop.sd3dprint}
The functional $J(\Gamma)$ is shape differentiable at any smooth open curve $\Gamma \Subset D$, and its shape derivative reads: 
$$J^\prime(\Gamma)(\theta) = - q \int_\Gamma  \left(\kappa p_\Gamma +  \frac{\partial p_\Gamma}{\partial n} \right) \: \theta\cdot n \:\d s -  q\int_\Sigma p_\Gamma \theta \cdot n_\Sigma  \:\d \ell, $$
\end{proposition}
where the adjoint state $p_\Gamma$ is the $H^1(D)$ solution to the following boundary-value problem: 
\begin{equation}\label{eq.pGammam}
\left\{
\begin{array}{cl}
-\dv(\gamma \nabla p_\Gamma) + \beta p_\Gamma = - \frac{\partial j}{\partial u}(x,u_\Gamma)& \text{in } D, \\
\gamma \frac{\partial p_\Gamma}{\partial n} = 0 & \text{on } \partial D. 
\end{array}
\right.
\end{equation}
\begin{proof}[Sketch of the proof]
We proceed along the lines of \cite{allaire2020survey}, introducing the transported mapping $\overline{u_\Gamma}(\theta) := u_{\Gamma_\theta} \circ (\Id + \theta)$. \par\medskip

\noindent \textit{Step 1: We prove the differentiability of the mapping $\theta \mapsto \overline{u_\Gamma}(\theta)$ and we characterize its derivative.}

To achieve this, we write the variational problem satisfied by the state function $u_{\Gamma_\theta}$ attached to the perturbed curve $\Gamma_\theta$: 
$$\forall v \in H^1(D), \quad \int_D \gamma \nabla u_{\Gamma_\theta}\cdot \nabla v \:\d x + \beta \int_D u_{\Gamma_\theta} v \:\d x = \beta \int_D u_0 v \:\d x + \int_{\Gamma_\theta} q v \:\d s. $$
Applying the change of variables of \cref{prop.chgvarintsurf} via the mapping $(\Id + \theta)$ and using the change of test functions $w = v \circ (\Id + \theta)$ to transport this problem back to the reference configuration featuring the curve $\Gamma \subset D$, we obtain the following variational characterization for $\overline{u_\Gamma}(\theta)$:
\begin{equation}\label{eq.varfuGth}
\forall w \in H^1(D), \quad \int_D \gamma A (\theta) \nabla \overline{u_{\Gamma}}(\theta) \cdot \nabla w \:\d x + \beta \int_{D} m(\theta)  \overline{u_{\Gamma}}(\theta) w \:\d x = \\
  \beta \int_{D} m(\theta) u_0 w \:\d x + \int_\Gamma m_{\text{t}}(\theta) q w \:\d s,
\end{equation}
where we have set 
$$m(\theta) = \lvert \det(\I + \nabla \theta) \lvert, \:  m_{\text{t}}(\theta) = \lvert \com(\I + \nabla \theta) n \lvert , \text{ and }  A(\theta ) =  m(\theta) (\I + \nabla\theta )^{-1} (\I + \nabla \theta)^{-T}. $$
A  classical application of the implicit function theorem now allows to prove that the mapping $\theta \mapsto \overline{u_\Gamma}(\theta)$ is Fr\'echet differentiable in the neighborhood of $\theta = 0$, see for instance \cite{murat1976controle}. Taking derivatives in \cref{eq.varfuGth}, we arrive at the following boundary-value problem for $\mathring{u_\Gamma}(\theta)$:
\begin{multline}\label{eq.bvpderlag} 
\forall w \in H^1(D), \quad  \int_D \gamma \nabla \mathring{u_\Gamma}(\theta) \cdot \nabla v \:\d x + \beta \int_D \mathring{u_\Gamma}(\theta) v \:\d x  = - \int_D \gamma \Big(\dv(\theta) \I - \nabla \theta - \nabla \theta^{-T} \Big) \nabla u_\Gamma \cdot \nabla w \:\d x \\
- \beta \int_D \dv(\theta) u_{\Gamma} w \:\d x + \beta \int_D \dv(\theta) u_0 w \:\d x + \int_\Gamma \dv_\Gamma(\theta) q w \:\d s .
\end{multline}\par\medskip

\noindent \textit{Step 2: We calculate the shape derivative of $J(\Gamma)$ in terms of $\mathring{u_\Gamma}(\theta)$.}

This follows from a simple change of variables in the definition of $J(\Gamma)$:
$$J(\Gamma_\theta) = \int_D m(\theta) j(x + \theta(x), \overline{u_\Gamma}(\theta)(x)) \:\d x , $$
which yields, by application of the chain rule:
$$J^\prime(\Gamma)(\theta) = \int_D  \Big(\dv(\theta) j(x,u_\Gamma) + \nabla_x j(x,u_\Gamma) \cdot \theta \Big)\:\d x + \int_D \frac{\partial j}{\partial u}(x,u_\Gamma)  \mathring{u_\Gamma}(\theta)\:\d x. $$
\par\medskip 

\noindent \textit{Step 3: We infer the volume form of the shape derivative $J^\prime(\Gamma)(\theta)$ by using the adjoint method.}

The defining problem \cref{eq.pGammam} for the adjoint state $p_\Gamma$ reads, under variational form: 
$$\forall w \in H^1(D), \quad \int_D \gamma \nabla p_{\Gamma} \cdot \nabla w \:\d x + \beta \int_D p_{\Gamma} w \:\d x = - \int_D \frac{\partial j}{\partial u}(x,u_\Gamma) w \:\d x.$$ 
Hence, by a classical adjoint-based computation, we obtain: 
\begin{equation}\label{eq.volformam}
\begin{array}{>{\displaystyle}cc>{\displaystyle}l}
J^\prime(\Gamma)(\theta) &=& \int_D  \Big(\dv(\theta) j(x,u_\Gamma) + \nabla_x j(x,u_\Gamma) \cdot \theta \Big)\:\d x  - \int_D \gamma \nabla p_{\Gamma} \cdot \nabla   \mathring{u_\Gamma}(\theta) \:\d x - \beta \int_D p_{\Gamma}   \mathring{u_\Gamma}(\theta) \:\d x \\[1em]
&=&  \int_D  \Big(\dv(\theta) j(x,u_\Gamma) + \nabla_x j(x,u_\Gamma) \cdot \theta \Big)\:\d x  + \int_D \gamma \Big(\dv(\theta) \I - \nabla \theta - \nabla \theta^{-T} \Big) \nabla u_\Gamma \cdot \nabla p_\Gamma \:\d x \\[1em]
&& \hspace{3.5cm} + \beta \int_D \dv(\theta) u_{\Gamma} p_\Gamma \:\d x - \beta \int_D \dv(\theta) u_0 p_\Gamma \:\d x - \int_\Gamma \dv_\Gamma(\theta) q p_\Gamma \:\d s .
\end{array}
\end{equation}
This is the desired volume form of $J^\prime(\Gamma)(\theta)$. 
\par\medskip
\noindent \textit{Step 4: We inspect the regularity of $u_\Gamma$ and $p_\Gamma$.}

Classical considerations from elliptic regularity theory -- about which we refer to e.g. \S 9.6 in \cite{brezis2010functional} -- allow to see that $p_\Gamma$ belongs to $H^2(D)$, while $u_\Gamma \in H^1(D)$ has jumping normal derivative across $\Gamma$:
$$\left[ \gamma \frac{\partial u_\Gamma}{\partial n} \right] = -q \text{ on } \Gamma, $$
where we recall from \cref{sec.fault} the notation $\left[ \cdot \right]$ for the jump of a discontinuous quantity across $\Gamma$. 

\par\medskip
\noindent \textit{Step 5: We infer the surface form of the derivative $J^\prime(\Gamma)(\theta)$.}

To achieve this, we perform integration by parts to eliminate all the derivatives of $\theta$ in the foregoing expression \cref{eq.volformam} of the volume form of $J^\prime(\Gamma)(\theta)$. This rests on the following avatars of the Green's formula, which are valid for all sufficiently smooth functions $u : D \to \R$ and vector fields $a,b :D \to \R^2$:
\begin{equation}\label{eq.Greenav1}
\int_D \dv(\theta) u \:\d x = \int_{\partial D} u \theta \cdot n \:\d s - \int_D \theta \cdot \nabla u \:\d x, 
\end{equation}
\begin{equation}\label{eq.Greenav2}
 \int_\Gamma \dv_\Gamma \theta u \:\d s =  \int_\Gamma \kappa u \theta \cdot n \:\d s - \int_\Gamma \theta_\Gamma \cdot \nabla_\Gamma u \:\d s + \int_\Sigma u \theta \cdot n_\Sigma  \:\d \ell,
\end{equation}
and
\begin{equation}\label{eq.Greenav3}
\int_D \nabla \theta a \cdot b \:\d x = \int_{\partial D} (\theta \cdot b ) (a \cdot n) \:\d s - \int_D \dv(a) \theta \cdot b \:\d x - \int_D \nabla b a \cdot \theta \:\d x.
\end{equation}
Hence, we arrive at:
\begin{multline*}
J^\prime(\Gamma)(\theta) = -\int_\Gamma \left[ \gamma \nabla u_\Gamma \cdot \nabla p_\Gamma \right] \theta\cdot n \:\d s  + \int_\Gamma \left( \left[ \gamma \nabla u_\Gamma \cdot \theta \right] \frac{\partial p_\Gamma}{\partial n} +  \gamma (\nabla p_\Gamma \cdot \theta) \left[ \frac{\partial u_\Gamma}{\partial n} \right] \right) \:\d s \\
-  \int_\Gamma q p_\Gamma \kappa \theta \cdot n \:\d s - \int_\Sigma q p_\Gamma \theta \cdot n_\Sigma  \:\d \ell + R(\theta),
\end{multline*}
where $R(\theta)$ is a sum of integrals on $D$ featuring only $\theta$ (not its derivatives) and of integrals on $\Gamma$ involving the tangential part of $\theta$ (not its normal component), that may change from one line to the other. 

Decomposing the inner product featured in the first integral in the above right-hand side into tangential and normal components, and using the fact that only the normal derivative of $u_\Gamma$ jumps through $\Gamma$ (see Step 4), we arrive at: 
$$J^\prime(\Gamma)(\theta) =  \int_\Gamma \left[ \gamma \frac{\partial u_\Gamma}{\partial n} \right]  \frac{\partial p_\Gamma}{\partial n}\: \theta\cdot n \:\d s  -  \int_\Gamma q p_\Gamma \kappa \theta \cdot n \:\d s - \int_\Sigma q p_\Gamma \theta \cdot n_\Sigma  \:\d \ell  + R(\theta). $$
Eventually, elementary albeit tedious computations allow to see that the remainder $R(\theta)$ vanishes, which reveals the desired surface form of the derivative of $J(\Gamma)$:
$$J^\prime(\Gamma)(\theta) = - q \int_\Gamma  \left(\kappa p_\Gamma +  \frac{\partial p_\Gamma}{\partial n} \right) \: \theta\cdot n \:\d s -  \int_\Sigma q p_\Gamma \theta \cdot n_\Sigma  \:\d \ell . $$
This terminates the proof of \cref{prop.sd3dprint}.
\end{proof}

\section{Mathematical treatment of the fault line detection example}\label{app.fault}

\noindent This appendix provides mathematical details about the treatment of the fault line detection example of \cref{sec.fault}: the consistency of the penalized approximation of the ``ideal'', discontinuous displacement field introduced in there is detailed in \cref{app.justifbroken}, and the shape derivative of the least-square functional $J(\Gamma)$ in \cref{eq.sopbfrac} used for inversion is calculated in \cref{sec.sdfault}. 

\subsection{Justification of the approximate model for broken spaces}\label{app.justifbroken}

\noindent In this section, we rigorously justify the approximate physical model considered in \cref{sec.fault}. 
For simplicity, we assume that the discontinuity line $\Gamma$ is closed:  the smooth, bounded computational domain $D \subset \R^2$ is divided into two complementary regions $\Omega_1, \Omega_2$, separated by the interface $\Gamma$: 
$$\Omega_1 \cap \Omega_2 = \emptyset, \:\: \overline{D} = \overline{\Omega_1} \cup \overline{\Omega_2}, \quad \Gamma = \partial \Omega_1 \cap \partial \Omega_2.$$
The precise situation of \cref{sec.fault}, where the support $\Gamma$ of the discontinuity in the model is an open line, is treated analogously since the imposed jump $g_\Gamma$ on $\Gamma$ has a smooth extension by $0$ to a closed curve $\widetilde\Gamma$ extending $\Gamma$, i.e. $g \in \widetilde{H}^{1/2}(\Gamma)^2$.
Still for simplicity, we consider the counterpart situation of that in \cref{sec.fault} arising in the physical context of the conductivity equation: the physical behavior of the system presenting the fault is described by the solution  $u \in H^1(D \setminus \Gamma)$ to the following problem:
\begin{equation}\label{eq.conducfrac}
\left\{
\begin{array}{cl}
-\dv(\gamma \nabla u) = 0 & \text{in } D, \\
u = 0 & \text{on } \partial D_D, \\
\gamma \frac{\partial u}{\partial n} = 0 & \text{on } \partial D \setminus \overline{\partial D_D},\\ 
\left[u\right] = g_\Gamma \text{ and } \left[\gamma \frac{\partial u}{\partial n}\right] = 0 & \text{on } \Gamma,
\end{array}
\right.
\end{equation}
where $n$ stands for the unit normal vector to $\Gamma$, pointing outward $\Omega_2$, see \cref{fig.fracset} (a).
Note that since $\Gamma$ is fixed in this section, for notational simplicity, we omit the $_\Gamma$ subscript referring to it in the solution of \cref{eq.conducfrac}.

According to \cref{sec.fault}, the approximate version of \cref{eq.conducfrac} is obtained by considering two functions $u_{1,\e}$, $u_{2,\e}$ accounting for (approximations of) the restrictions of $u$ to $\Omega_1$ and $\Omega_2$, respectively: 
these are both defined on $D$ as a whole by filling the complementary regions $D \setminus \overline{\Omega_1}$  and $D \setminus \overline{\Omega_2}$ with a material with ``soft'' conductivity. Precisely, we consider the variational problem: 
\begin{multline}\label{eq.conduceappfrac}
\text{Search for } (u_{1,\e}, u_{2,\e}) \in H^1_{\partial D_D}(D)^2 \text{ s.t. }  \forall (v_1, v_2) \in H^1_{\partial D_D}(D)^2, \\
 \int_D \gamma_{1,\e} \nabla u_{1,\e} \cdot \nabla v_1 \:\d x +  \int_D \gamma_{2,\e} \nabla u_{2,\e} \cdot \nabla v_2 \:\d x + \frac{1}{\e} \int_\Gamma (u_{2,\e} - u_{1,\e}) (v_2 -v_1) \:\d s =
\frac{1}{\e} \int_\Gamma g_\Gamma (v_2 -v_1)\:\d s,
\end{multline}
where we have introduced the conductivity coefficients: 
$$\gamma_{1,\e}(x) =\left\{ \begin{array}{cl}
\gamma(x) & \text{if } x \in \Omega_1,\\
\e \gamma(x) & \text{if } x \in \Omega_2,
\end{array}
\right. 
\text{ and }\gamma_{2,\e}(x) =\left\{ \begin{array}{cl}
\e\gamma(x) & \text{if } x \in \Omega_1,\\
 \gamma(x) & \text{if } x \in \Omega_2,
\end{array}
\right.  $$

In order to examine the consistency of the approximation \cref{eq.conduceappfrac} with the exact problem \cref{eq.conducfrac} as $\e\to 0$, let us introduce the Dirichlet extensions $u_i$ of $u$ from $\Omega_i$ to $D$ for $i=1,2$, i.e.:
$$ 
\left\{
\begin{array}{cl}
u_i= u & \text{in } \Omega_i, \\
-\dv(\gamma \nabla u_i) = 0 & \text{on } D\setminus \overline{\Omega_i}.
\end{array}
\right.
$$
Note that, as a consequence of the standard regularity theory for elliptic equations, both functions are smooth on $\overline{\Omega_1}$ and $\overline{\Omega_2}$:
\begin{equation}\label{eq.ustsmooth}
 \lvert\lvert u_1\lvert\lvert_{\calC^{1,\alpha}(\Omega_1 \cup \Omega_2)} +  \lvert\lvert u_2\lvert\lvert_{\calC^{1,\alpha}(\Omega_1 \cup \Omega_2)} \:\: \leq \:\:C \lvert\lvert g_\Gamma \lvert\lvert_{\calC^{1,\alpha}(\Gamma)},
 \end{equation}
see e.g. \S 9.6 in \cite{brezis2010functional}.
The desired consistency result is then the following.
\begin{proposition}
The following convergence holds true:
$$\lvert\lvert u_{1,\e} - u_1 \lvert\lvert_{H^1(D)} + \lvert\lvert u_{2,\e} - u_2 \lvert\lvert_{H^1(D)} \leq C \e^{\frac12}  \lvert\lvert g_\Gamma\lvert\lvert_{\calC^{1,\alpha}(\Gamma)}.$$
\end{proposition}
\begin{proof}
Let us define the error functions $r_{1,\e} := u_{1,\e} - u_1$ and $r_{2,\e} := u_{2,\e} - u_2 \in H^1_{\partial D_D}(D)$.
Then, $r_\e := (r_{1,\e}, r_{2,\e}) \in H^1_{\partial D_D}(D)^2$ satisfies, for an arbitrary test couple $v :=(v_1, v_2) \in H^1_{\partial D_D}(D)^2$:
\begin{multline*}
\int_D \gamma_{1,\e} \nabla r_{1,\e} \cdot \nabla v_1 \:\d x +  \int_D \gamma_{2,\e} \nabla r_{2,\e} \cdot \nabla v_2 \:\d x + \frac{1}{\e} \int_\Gamma ( r_{2,\e} - r_{1,\e}) (v_2 -v_1) \:\d s = \\
 - \int_D \gamma_{1,\e} \nabla u_1 \cdot \nabla v_1 \:\d x -  \int_D \gamma_{2,\e} \nabla u_2 \cdot \nabla v_2 \:\d x ,
\end{multline*}
where we have used the fact that, by construction $\left[u \right] = g_\Gamma$ on $\Gamma$. Inserting $v_1 = r_{1,\e}$, $v_2 = r_{2,\e}$ in this identity, decomposing the last two integrals in the above right-hand side onto $\Omega_1$ and $\Omega_2$ and integrating by parts, we obtain:
\begin{multline*}
\int_D \gamma_{1,\e} \lvert \nabla r_{1,\e} \lvert^2 \:\d x +  \int_D \gamma_{2,\e} \lvert \nabla r_{2,\e} \lvert^2 \:\d x + \frac{1}{\e} \int_\Gamma  ( r_{2,\e} - r_{1,\e})^2 \:\d s = \\ 
\int_{\Gamma}  \gamma_0 \frac{\partial u}{\partial n}  ( r_{2,\e} - r_{1,\e}) \:\d s - \e \int_{\Omega_2} \gamma_0 \nabla u_1 \cdot \nabla r_{1,\e} \:\d x  - \e \int_{\Omega_1} \gamma_0 \nabla u_2 \cdot \nabla r_{2,\e} \:\d x,
\end{multline*}
where we have used the continuity of the flux of $u$ through $\Gamma$. Now using the smoothness \cref{eq.ustsmooth} of $u_1$ and $u_2$, we arrive at: 
\begin{multline*}
\lvert\lvert \nabla r_{1,\e} \lvert\lvert^2_{L^2(\Omega_1)^2} + \lvert\lvert \nabla r_{2,\e} \lvert\lvert^2_{L^2(\Omega_2)^2}  + \e \lvert\lvert \nabla r_{1,\e} \lvert\lvert^2_{L^2(\Omega_2)^2} + \e \lvert\lvert \nabla r_{2,\e} \lvert\lvert^2_{L^2(\Omega_1)^2}  +\frac{1}{\e} \lvert\lvert  r_{2,\e} - r_{1,\e} \lvert\lvert^2_{L^2(\Gamma)}  \leq \\
 C \e^{1/2} \lvert\lvert g_\Gamma \lvert\lvert_{\calC^{1,\alpha}(\Gamma)} \Big( \frac{1}{\e^{1/2}}\lvert\lvert  r_{2,\e} - r_{1,\e} \lvert\lvert_{L^2(\Gamma)}  + \e^{1/2}   \lvert\lvert \nabla r_{1,\e} \lvert\lvert_{L^2(\Omega_2)^2} + \e^{1/2} \lvert\lvert \nabla r_{2,\e} \lvert\lvert_{L^2(\Omega_1)^2}  \Big). 
\end{multline*}
The Cauchy-Schwarz inequality now yields: 
$$ \lvert\lvert \nabla r_{1,\e} \lvert\lvert^2_{L^2(\Omega_1)^2} + \lvert\lvert \nabla r_{2,\e} \lvert\lvert^2_{L^2(\Omega_2)^2}  + \e \lvert\lvert \nabla r_{1,\e} \lvert\lvert^2_{L^2(\Omega_2)^2} + \e \lvert\lvert \nabla r_{2,\e} \lvert\lvert^2_{L^2(\Omega_1)^2}  +\frac{1}{\e} \lvert\lvert r_{2,\e} - r_{1,\e}\lvert\lvert^2_{L^2(\Gamma)}  \leq 
 C \e \lvert\lvert g_\Gamma \lvert\lvert_{\calC^{1,\alpha}(\Gamma)},  $$
 which readily implies the desired estimate thanks to the Poincar\'e's inequality. 
\end{proof}

\subsection{Calculation of the shape derivative}\label{sec.sdfault}

\noindent In this section, we detail the calculation of the shape derivative of the function $J(\Gamma)$ at play in the shape optimization problem \cref{eq.sopbfrac} used to reconstruct a fracture set inside a background medium. 

\subsubsection{The model context of the conductivity equation} 

\noindent To simplify the exposition, we first detail this derivation in the scalar context of the conductivity equation and for a generic functional $J(\Gamma)$, of the form: 
\begin{equation}\label{eq.Jfracgen}
J(\Gamma) = \int_{D} j(x,u_\Gamma) \:\d x, 
\end{equation}
where $j : D \times \R \to \R$ is a given, smooth function, and $u_\Gamma \in H^1_{\Gamma_D}(D \setminus \Gamma)$ is the solution to the conductivity equation \cref{eq.conducfrac}. \par\medskip

Our analysis starts with a lemma about the Lagrangian derivative of the normalized arc length function $s_\Gamma : \Gamma \to [0,1]$, involved in the definition \cref{eq.slipvec} of the slip function $g_\Gamma$. 
Throughout this section, we denote by $\tau:\Gamma \to \R^2$ the unit tangent vector to $\Gamma$, oriented in such a way that for any point $x \in \Gamma$, $(\tau(x),n(x))$ is a direct orthonormal frame of $\R^2$. In particular, $\tau(c_0) = -n_\Sigma(c_0)$ and $\tau(c_1) = n_\Sigma(c_1)$.

\begin{lemma}\label{lem.derarclength}
Let $\Gamma$ be a smooth, simple open curve of class $\calC^2$. For any $\theta \in \calC^{1,\infty}(\R^2; \R^2)$, let $\overline{s_\Gamma}(\theta) := s_{\Gamma_\theta} \circ (\Id + \theta)$ be the transported arc length function back to the reference curve $\Gamma$.
The function $\overline{s_\Gamma}(\theta)$ is Fr\'echet differentiable at $\theta=0$ and its derivative reads:
\begin{equation}\label{eq.derarclangthvol}
\text{For each point } x \in \Gamma, \quad \mathring{s_\Gamma}(\theta)(x) = \frac{1}{\lvert \Gamma\lvert} \left( \int_{\Gamma_{c_0,x}} \dv_\Gamma(\theta) \:\d s - s_\Gamma(x) \int_\Gamma \dv_\Gamma(\theta) \:\d s \right).
\end{equation}
Alternatively, this formula rewrites:
$$\mathring{s_\Gamma}(\theta)(x) = \frac{1}{\lvert \Gamma\lvert} \Bigg( (\theta\cdot \tau)(x) - s_\Gamma(x) (\theta\cdot \tau)(c_1)  -(1-s_\Gamma(x)) (\theta\cdot\tau)(c_0) + \int_{\Gamma_{c_0,x}} \kappa \theta\cdot n \:\d s - s_\Gamma(x) \int_\Gamma \kappa \theta\cdot n \:\d s  \Bigg). $$
\end{lemma}
\begin{proof}
By definition, the arc length of a point $y$ on the perturbed curve $\Gamma_\theta$ is given by 
$$s_{\Gamma_\theta}(y) = \frac{\lvert (\Gamma_\theta)_{c_0, y} \lvert }{\lvert \Gamma_\theta \lvert}.  $$
Hence, using the change of variables of \cref{prop.chgvarintsurf}, the transported arc length function at a point $x \in \Gamma$ equals:
$$\overline{s_\Gamma}(\theta)(x) =  \frac{1}{\lvert \Gamma_\theta \lvert} \lvert (\Gamma_\theta)_{(\Id + \theta)(c_0), (\Id + \theta)(x)} \lvert  = \frac{1}{\lvert \Gamma_\theta \lvert} \int_{\Gamma_{c_0,x}} \lvert \com(\I + \nabla \theta)n \lvert \:\d s.$$
Using the expression \cref{eq.dertgtJac} for the derivative of the tangential Jacobian, this yields: 
$$\mathring{s_\Gamma}(\theta)(x) = \frac{1}{\lvert \Gamma\lvert} \left( \int_{\Gamma_{c_0,x}} \dv_\Gamma(\theta) \:\d s -\frac{\lvert \Gamma_{c_0,x}\lvert}{\lvert \Gamma\lvert} \int_\Gamma \dv_\Gamma(\theta) \:\d s \right), $$
which is the desired formula \cref{eq.derarclangthvol}. The second expression follows by integration by parts on $\Gamma$, according to \cref{prop.IPPsurf}. 
\end{proof} 

We now come to the main result of this section.

\begin{proposition}\label{prop.sdfraclap}
The functional $J(\Gamma)$ in \cref{eq.Jfracgen} is shape differentiable and its shape derivative equals:
$$J^\prime(\Gamma)(\theta) = \int_\Gamma v_\Gamma \: \theta \cdot n \:\d s + \alpha_0 (\theta\cdot \tau)(c_0) +\alpha_1 (\theta\cdot \tau)(c_1), $$
 where the scalar field $v_\Gamma : \Gamma \to \R$ is defined by: 
\begin{multline*}
v_\Gamma =  -\left[ j(x,u_\Gamma) \right]    - \frac{1}{\lvert\Gamma\lvert} \gamma \frac{\partial p_\Gamma}{\partial \tau}  g^\prime(s_\Gamma(x))     -\frac{1}{\lvert\Gamma\lvert}   \left( \int_{\Gamma_{x,c_1}} \gamma \frac{\partial p_\Gamma}{\partial n}(y) g^\prime(s_\Gamma(y)) \d s(y)\right) \kappa  \\
  + \left( \frac{1}{\lvert\Gamma\lvert } \int_\Gamma \gamma \frac{\partial p_\Gamma}{\partial n} s_\Gamma(x) g^\prime(s_\Gamma(x)) \:\d s \right) \kappa,
\end{multline*} 
and the scalars $\alpha_0, \alpha_1$ are given by: 
$$ \alpha_0 = \frac{1}{\lvert\Gamma\lvert} \int_\Gamma \gamma \frac{\partial p_\Gamma}{\partial n} (1-s_\Gamma(x)) g^\prime(s_\Gamma(x)) \:\d s  \:\:\text{ and } \:\: \alpha_1 = \frac{1}{\lvert\Gamma\lvert} \int_\Gamma \gamma \frac{\partial p_\Gamma}{\partial n} s_\Gamma(x) g^\prime(s_\Gamma(x)) \:\d s  .$$   
In these formulas, the adjoint state $p_\Gamma$ is the unique solution in $H^1_{\partial D_D}(D)$ to the following boundary-value problem: 
\begin{equation}\label{eq.adjfault}
\left\{
\begin{array}{cl}
-\dv(\gamma \nabla p_\Gamma) = -\frac{\partial j}{\partial u}(x,u_\Gamma) & \text{in } D, \\
p_\Gamma = 0 & \text{on } \partial D_D, \\
\gamma \frac{\partial p_\Gamma}{\partial n} = 0 &\text{on } \partial D \setminus \overline{\partial D_D}.
\end{array}
\right.
\end{equation}
\end{proposition}

\begin{proof}[Hint of proof]
We proceed along the same strategy as in the proof of \cref{prop.sd3dprint}.
Let us introduce the transported version $\overline{u_\Gamma}(\theta) = u_{\Gamma_\theta} \circ (\Id + \theta) \in H^1_{\partial D_D}(D \setminus \overline\Gamma)$ of $u_{\Gamma_\theta}$ from the perturbed configuration back to the reference one. \par\medskip

\noindent \textit{Step 1: We prove the differentiability of the mapping $\theta \mapsto \overline{u_\Gamma}(\theta)$ and we identify its derivative.}

This task starts from the following variational characterization of $u_{\Gamma_\theta} \in H^1_{\partial D_D}(D \setminus \overline{\Gamma_\theta})$: 
$$
\left[ u_{\Gamma_\theta} \right] = g(s_{\Gamma_\theta}) \text{ on } \Gamma_\theta,  \text{ and } \forall v \in H^1_{\partial D_D}(D), \:\:  \int_D \gamma \nabla u_{\Gamma_\theta} \cdot \nabla v \:\d x= 0.
 $$
A change of variables and of test functions in the previous identity shows that $\overline{u_\Gamma}(\theta)$ is the unique solution in $H^1_{\partial D_D}(D \setminus \overline{\Gamma})$ to the following problem:
\begin{equation}\label{eq.ubarfrac}
\left[ \overline{u_\Gamma}(\theta) \right] = g(\overline{s_{\Gamma}}(\theta)) \text{ on } \Gamma,  \text{ and } \forall w \in H^1_{\partial D_D}(D), \:\:  \int_D \gamma A(\theta) \nabla \overline{u_{\Gamma}}(\theta) \cdot \nabla w \:\d x= 0,
\end{equation}
where, again, we set:
$$m(\theta) = \lvert \det(\I + \nabla \theta) \lvert,  \text{ and }  A(\theta ) =  m(\theta) (\I + \nabla\theta )^{-1} (\I + \nabla \theta)^{-T}. $$
The implicit function theorem shows that the mapping $\theta \mapsto \overline{u_{\Gamma}}(\theta)$ is Fr\'echet differentiable, see again \cite{murat1976controle}. Its derivative $\mathring{u_{\Gamma}}(\theta) \in H^1_{\partial D_D}(D \setminus \overline{\Gamma})$ then satisfies the following variational problem, obtained by taking derivatives in \cref{eq.ubarfrac}: 
\begin{multline}\label{eq.varflagder}
\left[ \mathring{u_\Gamma}(\theta) \right] = g^\prime(s_\Gamma) \mathring{s_\Gamma}(\theta) \text{ on } \Gamma,  \text{ and }\\
 \forall w \in H^1_{\partial D_D}(D), \:\:  \int_D \gamma \nabla \mathring{u_\Gamma}(\theta) \cdot \nabla w \:\d x= - \int_D \Big(\dv(\theta) \I - \nabla \theta - \nabla \theta^T \Big) \nabla u_\Gamma \cdot \nabla w \:\d x,
\end{multline}
where the Lagrangian derivative $ \mathring{s_\Gamma}(\theta)$ of the arc length function is provided by \cref{lem.derarclength}. 
\par\medskip 

\noindent \textit{Step 2. We calculate the derivative of $J(\Gamma)$ in terms of the Lagrangian derivative $\mathring{u_{\Gamma}}(\theta)$.} 

The definition \cref{eq.Jfracgen} of $J(\Gamma)$ and a change of variables based on \cref{prop.chgvarintsurf} together imply that: 
$$ J(\Gamma_\theta) = \int_D m(\theta) j(x+\theta(x),\overline{u_{\Gamma}}(\theta)) \:\d x.$$
By taking derivatives in this formula, we readily obtain: 
\begin{equation}\label{eq.derJnoadj}
J^\prime(\Gamma)(\theta) = \int_D \dv(\theta) j(x,u_{\Gamma}) \:\d x +  \int_D \nabla_x j(x,u_{\Gamma}) \cdot \theta \:\d x 
+  \int_D \frac{\partial j}{\partial u}(x,u_{\Gamma}) \mathring{u_{\Gamma}}(\theta)  \:\d x.
\end{equation}
\par\medskip 

\noindent \textit{Step 3. We transform this expression by introducing the adjoint state $p_{\Gamma}$.}

To achieve this, we inject the strong form of the defining boundary-value problem \cref{eq.adjfault} into \cref{eq.derJnoadj}, before integrating by parts. This yields:
\begin{equation}\label{eq.volformfrac}
\begin{array}{>{\displaystyle}cc>{\displaystyle}l} 
J^\prime(\Gamma)(\theta) &=&    \int_D \dv(\theta) j(x,u_{\Gamma}) \:\d x +  \int_D \nabla_x j(x,u_{\Gamma}) \cdot \theta \:\d x 
+ \int_D \dv(\gamma \nabla p_\Gamma) \mathring{u_{\Gamma}}(\theta)  \:\d x  \\[1em]
&=&    \int_D \dv(\theta) j(x,u_{\Gamma}) \:\d x +  \int_D \nabla_x j(x,u_{\Gamma}) \cdot \theta \:\d x -  \int_\Gamma \gamma \frac{\partial p_\Gamma}{\partial n} \left[ \mathring{u_{\Gamma}}(\theta)\right] \:\d s - \int_D \gamma \nabla \mathring{u_\Gamma}(\theta) \cdot \nabla p_\Gamma \:\d x \\[1em]
&=&    \int_D \dv(\theta) j(x,u_{\Gamma}) \:\d x +  \int_D \nabla_x j(x,u_{\Gamma}) \cdot \theta \:\d x -  \int_\Gamma \gamma \frac{\partial p_\Gamma}{\partial n} g^\prime(s_\Gamma(x)) \mathring{s_\Gamma}(\theta)(x) \:\d s \\[1em]
&& \hspace{6cm} +\int_D \gamma \Big(\dv(\theta) \I - \nabla \theta - \nabla \theta^T\Big)\nabla u_\Gamma \cdot \nabla p_\Gamma \:\d x,
\end{array}  
\end{equation}
where we have used the continuity of the normal derivative of $p_\Gamma$ through $\Gamma$ to pass from the first to the second line, and the variational characterization \cref{eq.varflagder} of the Lagrangian derivative $\mathring{u_\Gamma}(\theta)$ to obtain the final line. This formula is the volume form of the shape derivative $J^\prime(\Gamma)(\theta)$.
\par\medskip 

\noindent \textit{Step 4. We perform integration by parts in the volume form, discarding all terms that should not contribute, in view of the expected structure \cref{eq.structJp}.} 

This relies on similar calculations as in the proof of \cref{prop.sd3dprint}. Throughout the rest of the proof, we denote by $R(\theta)$ a (possibly changing from line to line) collection of integrals on $D$ depending only on $\theta$ (not on its derivatives), or surface integrals on $\Gamma$ involving only the tangential component of $\theta$ inside $\Gamma$. 

Since $\theta$ vanishes on $\partial D$ the first two integrals in the volume form \cref{eq.volformfrac} equal: 
$$ \int_D \dv(\theta) j(x,u_\Gamma) \:\d x +  \int_D \nabla_x j(x,u_\Gamma) \cdot \theta \:\d x =  -\int_\Gamma \left[ j(x,u_\Gamma) \right] \theta\cdot n \:\d s + R(\theta), $$
where the sign comes from the orientation of the normal vector $n$, see \cref{fig.fracset} (a), and from the definition \cref{eq.defjump} of the jump of a discontinuous quantity across $\Gamma$.

Furthermore, using the integration by parts formulas \cref{eq.Greenav1,eq.Greenav3,eq.Greenav3} after decomposition of $D$ into both subdomains $\Omega_1$, $\Omega_2$ located on either side of the fracture $\Gamma$, we obtain:
\begin{equation*}
\begin{array}{>{\displaystyle}cc>{\displaystyle}l} 
\int_D  \gamma \Big(\dv(\theta) \I - \nabla \theta - \nabla\theta^T \Big) \nabla u_{\Gamma} \cdot \nabla p_{\Gamma} \:\d x &=&  - \int_\Gamma \left[ \gamma \nabla_\Gamma u_{\Gamma} \cdot \nabla_\Gamma p_{\Gamma} \right] \:\theta\cdot n \:\d s 
+\int_\Gamma \left[ \gamma \frac{\partial u_{\Gamma}}{\partial n}  \frac{\partial p_{\Gamma}}{\partial n} \right] \:\theta\cdot n \:\d s  + R(\theta) \\[1em]
&=& - \int_\Gamma \gamma  g^\prime(s_\Gamma(x)) \frac{\partial p_\Gamma}{\partial \tau} \frac{\partial s_\Gamma}{\partial \tau} \theta \cdot n \:\d s + R(\theta) \\[1em]
&=& - \frac{1}{\lvert\Gamma\lvert} \int_\Gamma \gamma  \frac{\partial p_\Gamma}{\partial \tau}  g^\prime(s_\Gamma(x)) \: \theta \cdot n \:\d s + R(\theta),
\end{array}
\end{equation*}
where we have used the boundary conditions satisfied by $u_\Gamma$ and $p_\Gamma$ on $\Gamma$ to pass from the first line to the second one, and notably the continuity of $p_\Gamma$ and of the normal derivatives of $u_\Gamma$ and $p_\Gamma$. To obtain the second line, we have also used the formula $\frac{\partial s_\Gamma}{\partial \tau} = \frac{1}{\lvert\Gamma\lvert}$, which follows from the very definition of the tangential derivative.

Combining all these results, we arrive at: 
\begin{equation}\label{eq.Jpfaultap1}
J^\prime(\Gamma)(\theta) =    -\int_\Gamma \left[ j(x,u_\Gamma) \right] \theta\cdot n \:\d s -  \int_\Gamma \gamma \frac{\partial p_\Gamma}{\partial n} g^\prime(s_\Gamma(x)) \mathring{s_\Gamma}(\theta)(x) \:\d s    - \frac{1}{\lvert\Gamma\lvert} \int_\Gamma \gamma \frac{\partial p_\Gamma}{\partial \tau}  g^\prime(s_\Gamma(x))  \theta \cdot n \:\d s  + R(\theta).
\end{equation}
We now invoke \cref{lem.derarclength} to expand the second term in the above right-hand side: 
\begin{multline*}
 - \int_\Gamma \gamma \frac{\partial p_\Gamma}{\partial n} g^\prime(s_\Gamma(x)) \mathring{s_\Gamma}(\theta)(x) \:\d s = \left(\frac{1}{\lvert\Gamma\lvert} \int_\Gamma \gamma \frac{\partial p_\Gamma}{\partial n} s_\Gamma(x) g^\prime(s_\Gamma(x)) \:\d s \right) \theta\cdot \tau(c_1) \\
   +  \left(\frac{1}{\lvert\Gamma\lvert} \int_\Gamma \gamma \frac{\partial p_\Gamma}{\partial n} (1-s_\Gamma(x)) g^\prime(s_\Gamma(x)) \:\d s \right) \theta\cdot \tau(c_0)
   -\frac{1}{\lvert\Gamma\lvert} \int_\Gamma  \gamma \frac{\partial p_\Gamma}{\partial n} g^\prime(s_\Gamma(x)) \left( \int_{\Gamma_{c_0,x}} \kappa \theta\cdot n \:\d s \right) \:\d s(x) \\
  + \left( \frac{1}{\lvert\Gamma\lvert } \int_\Gamma \gamma \frac{\partial p_\Gamma}{\partial n} s_\Gamma(x) g^\prime(s_\Gamma(x)) \:\d s \right) \int_\Gamma  \kappa \theta\cdot n \:\d s +R(\theta).
\end{multline*}
Finally, thanks to Fubini's theorem, this rewrites:
\begin{multline}\label{eq.Jpfaultap2}
 - \int_\Gamma \gamma \frac{\partial p_\Gamma}{\partial n} g^\prime(s_\Gamma(x)) \mathring{s_\Gamma}(\theta)(x) \:\d s = \left(\frac{1}{\lvert\Gamma\lvert} \int_\Gamma \gamma \frac{\partial p_\Gamma}{\partial n} s_\Gamma(x) g^\prime(s_\Gamma(x)) \:\d s \right) \theta\cdot \tau(c_1) \\
   +  \left(\frac{1}{\lvert\Gamma\lvert} \int_\Gamma \gamma \frac{\partial p_\Gamma}{\partial n} (1-s_\Gamma(x)) g^\prime(s_\Gamma(x)) \:\d s \right) \theta\cdot \tau(c_0)
   -\frac{1}{\lvert\Gamma\lvert} \int_\Gamma  \left( \int_{\Gamma_{x,c_1}} \gamma \frac{\partial p_\Gamma}{\partial n}(y) g^\prime(s_\Gamma(y)) \d s(y)\right) \kappa \theta\cdot n \:\d s(x)  \\
  + \left( \frac{1}{\lvert\Gamma\lvert } \int_\Gamma \gamma \frac{\partial p_\Gamma}{\partial n} s_\Gamma(x) g^\prime(s_\Gamma(x)) \:\d s \right) \int_\Gamma  \kappa \theta\cdot n \:\d s + R(\theta).
\end{multline}
Combining \cref{eq.Jpfaultap1,eq.Jpfaultap2} and after the elementary albeit tedious verification that $R(\theta)$ cancels, we obtain the desired result.
\end{proof}

\subsubsection{Extension in the context of the elasticity equations} 

\noindent We now proceed to the calculation of the shape derivative of the functional $J(\Gamma)$ of interest, framed in the context of the linear elasticity system.
Here again, to set ideas, we consider a function of the domain of the form: 
$$ J(\Gamma) = \int_D j(x,u_\Gamma(x)) \:\d x,$$
where $j: D \times \R^2 \to \R$ is a smooth enough function, and $u_\Gamma$ is the elastic displacement, solution to the system \cref{eq.elasfrac}. 
The result of interest is the following. 

\begin{proposition}\label{prop.sdfaultelas}
The functional $J(\Gamma)$ is shape differentiable and its shape derivative equals:
$$J^\prime(\Gamma)(\theta) = \int_\Gamma v_\Gamma \theta \cdot n \:\d s + \alpha_0 (\theta\cdot \tau)(c_0) +\alpha_1 (\theta\cdot \tau)(c_1), $$
 where the scalar field $v_\Gamma : \Gamma \to \R$ is defined by: 
\begin{multline*}
v_\Gamma =  -\left[ j(x,u_\Gamma) \right]    - \frac{1}{\lvert\Gamma\lvert} (Ae(p_\Gamma) \tau) \cdot g^\prime(s_\Gamma(x))     -\frac{1}{\lvert\Gamma\lvert}   \left( \int_{\Gamma_{x,c_1}} (Ae(p_\Gamma)n)(y) \cdot g^\prime(s_\Gamma(y)) \d s(y)\right) \kappa(x)  \\
  + \left( \frac{1}{\lvert\Gamma\lvert } \int_\Gamma s_\Gamma(x)  (Ae(p_\Gamma)n) \cdot g^\prime(s_\Gamma(x)) \:\d s \right) \kappa(x),
\end{multline*} 
and the scalars $\alpha_0, \alpha_1$ are given by: 
$$ \alpha_0 = \frac{1}{\lvert\Gamma\lvert} \int_\Gamma  (1-s_\Gamma(x))  (Ae(p_\Gamma)n)(x) \cdot g^\prime(s_\Gamma(x)) \:\d s \:\:\text{ and } \:\: \alpha_1 = \frac{1}{\lvert\Gamma\lvert} \int_\Gamma s_\Gamma(x)  (Ae(p_\Gamma)n)(x) \cdot g^\prime(s_\Gamma(x)) \:\d s  .$$   
In these formulas, the adjoint state $p_\Gamma$ is the unique solution in $H^1_{\partial D_D}(D)^2$ to the following boundary-value problem: 
\begin{equation}\label{eq.adjfaultelas}
\left\{
\begin{array}{cl}
-\dv(A e( p_\Gamma)) = -\nabla_u j (x,u_\Gamma) & \text{in } D, \\
p_\Gamma = 0 & \text{on } \partial D_D, \\
Ae(p_\Gamma) n= 0 &\text{on } \partial D \setminus \overline{\partial D_D}.
\end{array}
\right.
\end{equation}
\end{proposition}

\begin{proof}[Hint of proof]
As the proof mirrors that of \cref{prop.sdfraclap}, for brevity, we only report on the needed technical adaptations.
\par\medskip
\noindent \textit{Step 1.}
The elastic displacement $u_{\Gamma_\theta} \in H^1_{\partial D_D}(D \setminus \overline{\Gamma_\theta})^2$ is characterized by: 
$$
\left[ u_{\Gamma_\theta} \right] = g(s_{\Gamma_\theta}) \text{ on } \Gamma_\theta,  \text{ and } \forall v \in H^1_{\partial D_D}(D)^2, \:\:  \int_D A e(u_{\Gamma_\theta}): e( v) \:\d x= 0.
 $$
A change of variables and of test functions in the previous identity shows that $\overline{u_\Gamma}(\theta)$ is the unique solution in $H^1_{\partial D_D}(D \setminus \overline{\Gamma})^2$ to the following problem:
$$\left[ \overline{u_\Gamma}(\theta) \right] = g(\overline{s_{\Gamma}}(\theta)) \text{ on } \Gamma,  \text{ and } \forall w \in H^1_{\partial D_D}(D)^2, \:\:  \int_D m(\theta) A  E( \overline{u_\Gamma}(\theta), \theta): E( w, \theta) \:\d x= 0,$$
where
$$m(\theta) = \lvert \det(\I + \nabla \theta) \lvert,  \text{ and }  E(w,\theta) := \frac{1}{2}\Big( \nabla w (I+\nabla \theta)^{-1} +  (I+\nabla \theta)^{-T}\nabla w^T \Big).$$
Again, the implicit function theorem shows that the mapping $\theta \mapsto \overline{u_{\Gamma}}(\theta)$ is Fr\'echet differentiable, and its derivative $\mathring{u_{\Gamma}}(\theta) \in H^1_{\partial D_D}(D \setminus \overline{\Gamma})^2$ satisfies the following variational problem: 
\begin{multline}\label{eq.varflagder}
\left[ \mathring{u_\Gamma}(\theta) \right] = g^\prime(s_\Gamma) \mathring{s_\Gamma}(\theta) \text{ on } \Gamma,  \text{ and }
 \forall w \in H^1_{\partial D_D}(D)^2, \\
   \int_D  Ae( \mathring{u_\Gamma}(\theta)): e(w) \:\d x=  - \int_D \dv(\theta) Ae(u_\Gamma): e(w) \:\d x + \int_D A C(u_\Gamma,\theta) : e(w) \:\d x + \int_D A e(u_\Gamma) : C(w,\theta)\:\d x,
\end{multline}
with the shortcut $C(v,\theta) := \frac{1}{2} \Big(\nabla v \nabla\theta + \nabla\theta^T \nabla v^T \Big)$.
\par\medskip 
\noindent \textit{Step 2.} A straightforward calculation yields:
$$ J^\prime(\Gamma)(\theta) = \int_D \dv(\theta) j(x,u_{\Gamma}) \:\d x +  \int_D \nabla_x j(x,u_{\Gamma}) \cdot \theta \:\d x 
+  \int_D \nabla_u j(x,u_{\Gamma}) \cdot \mathring{u_{\Gamma}}(\theta)  \:\d x.$$
\par\medskip
\noindent \textit{Step 3.} By an adjoint calculation, similar to that conducted in the proof of \cref{prop.sdfraclap}, we obtain:
\begin{multline}\label{eq.volformfracelas}
 J^\prime(\Gamma)(\theta) = \int_D \dv(\theta) j(x,u_{\Gamma}) \:\d x +  \int_D \nabla_x j(x,u_{\Gamma}) \cdot \theta \:\d x 
-  \int_\Gamma Ae(p_\Gamma)n \cdot\left[ \mathring{u_{\Gamma}}(\theta) \right] \:\d x  \\
 + \int_D \dv(\theta) Ae(u_\Gamma): e(p_\Gamma) \:\d x  - \int_D A C(u_\Gamma,\theta) : e(p_\Gamma) \:\d x - \int_D A e(u_\Gamma) : C(p_\Gamma,\theta)\:\d x.
\end{multline}
\par\medskip 

\noindent \textit{Step 4.} At first, we have: 
$$ \int_D \dv(\theta) j(x,u_{\Gamma}) \:\d x +  \int_D \nabla_x j(x,u_{\Gamma}) \cdot \theta \:\d x  = -\int_{\Gamma} \left[ j(x,u_\Gamma) \right] \:\theta\cdot n \:\d s + R(\theta), $$
where again, $R(\theta)$ is a (possibly changing from line to line) collection of integrals on $D$ depending only on $\theta$ (not on its derivatives), or surface integrals on $\Gamma$ involving only the tangential component of $\theta$ inside $\Gamma$.

Then, a use of the Green's formula yields: 
$$\begin{array}{>{\displaystyle}cc>{\displaystyle}l} 
 \int_D \dv(\theta) Ae(u_\Gamma): e(p_\Gamma) \:\d x  &=&  -\int_\Gamma \left[ Ae(u_\Gamma): e(p_\Gamma) \right] (\theta\cdot n) \:\d s  \\[1em]
 &=&  -\int_\Gamma (Ae(p_\Gamma) n ) \cdot \left[ \nabla u_\Gamma n\right] (\theta\cdot n) \:\d s -\int_\Gamma (Ae(p_\Gamma) \tau ) \cdot \left[ \nabla u_\Gamma \tau \right] (\theta\cdot n) \:\d s \\[1em] 
 &=&  -\int_\Gamma (Ae(p_\Gamma) n ) \cdot \left[ \nabla u_\Gamma n\right] (\theta\cdot n) \:\d s - \frac{1}{\lvert\Gamma\lvert} \int_\Gamma (Ae(p_\Gamma) \tau ) \cdot  g^\prime(s_\Gamma(x)) (\theta\cdot n) \:\d s. \\[1em] 
 \end{array}
 $$
Likewise, we have: 
$$\begin{array}{>{\displaystyle}cc>{\displaystyle}l} 
  \int_D A C(u_\Gamma,\theta) : e(p_\Gamma) \:\d x &=& \int_D \nabla u_\Omega \nabla \theta : \left( Ae(p_\Gamma) \right) \:\d x  \\[1em]
   &=& \int_D  \nabla \theta : \left( \nabla u_\Omega^T Ae(p_\Gamma) \right) \:\d x  \\
   &=& - \int_\Gamma  \left( Ae(p_\Gamma)n \right) \cdot \left[ \nabla u_\Gamma n\right] \:\theta\cdot n  \:\d s + R(\theta).
  \end{array} 
$$
Finally, 
$$  \int_D A C(p_\Gamma,\theta) : e(u_\Gamma) \:\d x = - \int_\Gamma  \left( Ae(u_\Gamma)n \right) \cdot \left[ \nabla p_\Gamma  n\right] \:\theta\cdot n  \:\d s + R(\theta) = R(\theta),$$
since $p_\Gamma$ is smooth in the neighborhood of $\Gamma$.

Injecting these formulas into the volume form \cref{eq.volformfracelas} and verifying that the integrals in $R(\theta)$ vanish, 
we obtain the desired expression.
\end{proof}
\bibliographystyle{siam}
\bibliography{genbib.bib}

\end{document}